\documentclass[10pt]{amsart}
\usepackage{amssymb}
\usepackage{amscd}
\usepackage[all]{xy}

\newcounter{TmpEnumi}

\numberwithin{equation}{section}

\def\today{\number\day\space\ifcase\month\or   January\or February\or
   March\or April\or May\or June\or   July\or August\or September\or
   October\or November\or December\fi\   \number\year}

\theoremstyle{definition}
\newtheorem{thm}{Theorem}[section]
\newtheorem{lem}[thm]{Lemma}
\newtheorem{prp}[thm]{Proposition}
\newtheorem{dfn}[thm]{Definition}
\newtheorem{cor}[thm]{Corollary}

\newtheorem{cnv}[thm]{Convention}

\newtheorem{ntn}[thm]{Notation}
\newtheorem{exa}[thm]{Example}
\newtheorem{pbm}[thm]{Problem}

\newtheorem{qst}[thm]{Question}

\newcommand{\beq}{\begin{equation}}
\newcommand{\eeq}{\end{equation}}
\newcommand{\beqr}{\begin{eqnarray*}}
\newcommand{\eeqr}{\end{eqnarray*}}
\newcommand{\bal}{\begin{align*}}
\newcommand{\eal}{\end{align*}}
\newcommand{\bei}{\begin{itemize}}
\newcommand{\eei}{\end{itemize}}
\newcommand{\limi}[1]{\lim_{{#1} \to \infty}}

\newcommand{\af}{\alpha}
\newcommand{\bt}{\beta}
\newcommand{\gm}{\gamma}
\newcommand{\dt}{\delta}
\newcommand{\ep}{\varepsilon}
\newcommand{\zt}{\zeta}
\newcommand{\et}{\eta}

\newcommand{\io}{\iota}
\newcommand{\te}{\theta}
\newcommand{\ld}{\lambda}
\newcommand{\sm}{\sigma}
\newcommand{\kp}{\kappa}
\newcommand{\ph}{\varphi}
\newcommand{\ps}{\psi}
\newcommand{\rh}{\rho}
\newcommand{\om}{\omega}
\newcommand{\ta}{\tau}

\newcommand{\Ld}{\Lambda}

\newcommand{\Q}{{\mathbb{Q}}}
\newcommand{\Z}{{\mathbb{Z}}}
\newcommand{\R}{{\mathbb{R}}}
\newcommand{\C}{{\mathbb{C}}}
\newcommand{\N}{{\mathbb{Z}}_{> 0}}
\newcommand{\Nz}{{\mathbb{Z}}_{\geq 0}}

\newcommand{\cN}{{\mathcal{N}}}

\newcommand{\tsr}{{\mathrm{tsr}}}
\newcommand{\RR}{{\mathrm{RR}}}

\newcommand{\tr}{{\mathrm{tr}}}
\newcommand{\id}{{\mathrm{id}}}
\newcommand{\ev}{{\mathrm{ev}}}
\newcommand{\sint}{{\mathrm{int}}}

\newcommand{\Prim}{{\mathrm{Prim}}}
\newcommand{\diag}{{\mathrm{diag}}}
\newcommand{\supp}{{\mathrm{supp}}}

\newcommand{\card}{{\mathrm{card}}}
\newcommand{\Aut}{{\mathrm{Aut}}}
\newcommand{\Ad}{{\mathrm{Ad}}}

\newcommand{\Kern}{\mathrm{Ker}}

\newcommand{\dirlim}{\varinjlim}
\newcommand{\invlim}{\varprojlim}

\newcommand{\andeqn}{\qquad {\mbox{and}} \qquad}

\newcommand{\ts}[1]{{\textstyle{#1}}}
\newcommand{\ds}[1]{{\displaystyle{#1}}}

\newcommand{\set}[1]{\left\{ #1 \right\}}
\newcommand{\bset}[1]{\bigl\{ #1 \bigr\}}

\newcommand{\CXDHat}[2]{C (#1, #2 )^{\wedge}}

\newcommand{\Wolog}{Without loss of generality}

\newcommand{\ifo}{if and only if}

\newcommand{\ca}{C*-algebra}
\newcommand{\uca}{unital C*-algebra}

\newcommand{\hm}{homomorphism}

\newcommand{\fd}{finite dimensional}
\newcommand{\tst}{tracial state}

\newcommand{\pj}{projection}

\newcommand{\nzp}{nonzero projection}

\newcommand{\mvnt}{Murray-von Neumann equivalent}

\newcommand{\ct}{continuous}
\newcommand{\cfn}{continuous function}
\newcommand{\nbhd}{neighborhood}
\newcommand{\cms}{compact metric space}

\newcommand{\chs}{compact Hausdorff space}
\newcommand{\hme}{homeomorphism}
\newcommand{\mh}{minimal homeomorphism}

\newcommand{\tgca}{transformation group \ca}

\newcommand{\JS}{{\mathcal{Z}}}

\renewcommand{\S}{\subset}

\newcommand{\SM}{\setminus}
\newcommand{\I}{\infty}
\newcommand{\E}{\varnothing}
\newcommand{\ten}{\otimes}

\newcommand{\D}{\mathcal{E}}

\newcommand{\mdim}{\mathrm{mdim}}

\newcommand{\Lem}[1]{Lemma~\ref{#1}}
\newcommand{\Def}[1]{Definition~\ref{#1}}

\title[Crossed products of $C (X, D)$]{The structure of
 crossed products by automorphisms of $C (X, D)$}
\author{Dawn Archey
\and
Julian Buck
\and
N.~Christopher Phillips}

\date{25~September 2020}

\address{University of Detroit Mercy, Department of Mathematics,
       4001 W. McNichols Rd., Detroit MI 48221, USA.}

\address{Department of Mathematics, Okanagan College, 1000 KLO Road,
       Kelowna BC, V1Y 4X8 Canada}

\address{Department of Mathematics, University  of Oregon,
       Eugene OR 97403-1222, USA.}

\thanks{This material is based upon work supported by the
  US National Science Foundation under
  Grant DMS-1101742 and the Simons Foundation Collaboration Grant
  for Mathematicians \#587103.}

\begin{document}

\begin{abstract}
We construct centrally large subalgebras in crossed products of
the form $C^* \bigl( \Z, \, C (X, D), \, \af \bigr)$
in which $D$ is simple, $X$ is compact metrizable,
$\af$ induces a minimal homeomorphism $h \colon X \to X$,
and a mild technical assumption holds.
We use this construction to prove structural properties of
the crossed product,
such as (tracial) ${\mathcal{Z}}$-stability, stable rank one,
real rank zero, and pure infiniteness,
in a number of examples.
Our examples are not accessible
via methods based on finite Rokhlin dimension,
either because $D$ is not ${\mathcal{Z}}$-stable
or because $X$ is infinite dimensional.
\end{abstract}

\maketitle

\setcounter{section}{-1}

\section{Introduction}
Significant progress has been made in recent years on the
classification of crossed product \ca{s} arising from
finite dimensional minimal dynamical systems.
The long unpublished
preprint \cite{QLinPhDiff} of Q.~Lin and N.~C.\  Phillips
(see also the
survey articles \cite{QLinPh1} and \cite{QLinPh2}) provides a
thorough description of the transformation group \ca{s}
arising from minimal diffeomorphisms of finite dimensional smooth
compact manifolds in terms of a direct limit decomposition.
In
\cite{HLinPh} and \cite{TomsWinter}, it is shown that crossed
products arising from minimal homeomorphisms of infinite compact
metrizable spaces with finite covering dimension are classified by
their ordered K-theory in the presence of sufficiently many projections
(for instance, when projections separate traces).
In
\cite{TomsWinter} it is further proved that crossed products by such
minimal homeomorphisms have finite nuclear dimension, and hence
absorb the Jiang-Su algebra $\JS$ tensorially
(that is, are $\JS$-stable).
Finally, G.~A.\  Elliott and Z.~Niu (\cite{EllNiu})
have shown that crossed products by minimal
homeomorphisms of compact metric spaces with mean dimension zero
(including all minimal
homeomorphisms of finite dimensional compact metric spaces)
are ${\mathcal{Z}}$-stable,
from which it follows that they are classifiable
in the sense of the Elliott program by Corollary D of
Theorem A of \cite{CETWW}.

Not as much is known for crossed products of
\ca{s} of the form $C (X, D)$ for a
noncommutative \ca~$D$.
Hua (\cite{Hua}) has shown
that such crossed products have tracial rank zero
when $X$ is the Cantor set, $D$ has tracial rank zero,
the action on $X$ is minimal,
and some additional K-theoretic assumptions are made.
In this paper we consider the structure of crossed products of
the form $C^* \bigl( \Z, \, C (X, D), \, \af \bigr)$,
in which $X$ is a compact metric space,
$D$ is simple unital a C*-algebra, and $\af$ is an automorphism
of $C (X, D)$ which ``lies over'' a minimal homeomorphism
(as described in Definition~\ref{D_4Y12_OverX}).
In Section~\ref{Sec_Over}, we describe the types of actions on
$C (X, D)$ which will be of interest here.
In Section~\ref{Sec_OrbBreak} we introduce the generalization
of the orbit breaking subalgebras of \cite{PhLg}
for actions on $C (X, D)$,
and show that these are large in various senses
defined in \cite{PhLg}.
Section~\ref{RecStruct} introduces a
``$D$-fibered'' generalization of the recursive subhomogeneous
algebras in \cite{PhRsha1}, then demonstrates (following a
development analogous to that in \cite{QLinPhDiff}) that our
orbit breaking subalgebras have such a recursive structure.
In Section~\ref{StructAY} we use the results of
Sections~\ref{Sec_OrbBreak} and~\ref{RecStruct} to obtain
stronger structural properties for the orbit breaking
subalgebras, and deduce structural
properties of the crossed product from those of the
orbit breaking subalgebra under appropriate additional
assumptions.
In Section~\ref{Sec_9929_MinProd} we
establish minimality for products of certain Denjoy
homeomorphisms that will be used to produce examples.
Section~\ref{Sec_InfRed} gives a large collection of
examples of crossed products for which we can use the
theory developed here to deduce structural properties
that do not seem accessible using previously known methods.
Finally, in Section~\ref{Sec_Qs} we pose some
open questions for further research.

We recall Cuntz comparison and the Cuntz semigroup.
For a much fuller discussion,
in a form useful for work with large subalgebras,
we refer to Section~1 of~\cite{PhLg}.

\begin{ntn}\label{N_9Y30_CzComp}
If $A$ is a \ca{} and $a, b \in M_{\infty} (A)_{+}$,
we write $a \precsim_A b$
to mean that $a$ is Cuntz subequivalent to~$b$ over~$A$,
that is, there is a sequence $(v_n)_{n = 1}^{\infty}$
in $M_{\infty} (A)_{+}$
such that
$\limi{n} v_n b v_n^* = a$.
We write $a \sim_A b$
to mean that $a$ is Cuntz equivalent to~$b$ over~$A$,
that is, $a \precsim_A b$ and $b \precsim_A a$.
\end{ntn}

We specify~$A$ in the notation
because Cuntz subequivalence with respect to proper
subalgebras will play a key role.

The Cuntz semigroup $W (A)$ is then defined to be the set of
Cuntz equivalence classes $M_{\infty} (A)_{+} / \sim_A$,
with addition given by direct sum and order coming from
Cuntz subequivalence.

\begin{ntn}\label{ntn1-181221D}
For any \ca{} $A$, we denote the set of normalized $2$-quasitraces
on~$A$ by ${\operatorname{QT}} (A)$.
We use the word quasitrace to mean normalized $2$-quasitrace.
\end{ntn}

For a \ca~$A$,
the topology on $\Aut (A)$
is always pointwise convergence in the norm of~$A$.
That is, $x \mapsto \af_x$ is \ct{}
\ifo{} $x \mapsto \af_x (a)$ is \ct{} for all $a \in A$.
(This is the usual topology.)
To be explicit,
we point out that $\af \mapsto \af^{-1}$ is \ct{} in this topology,
as can be seen from the equation
\[
\| \af^{-1} (a) - \bt^{-1} (a) \|
  = \bigl\| \bt^{-1}
      \bigl( \af (\af^{-1} (a)) - \bt (af^{-1} (a)) \bigr) \bigr\|
\]
for $a \in A$ and $\af, \bt \in \Aut (A)$,
and the fact that $\bt^{-1}$ is isometric.

\begin{ntn}\label{N_0108_DimX}
For a \chs~$X$,
we denote the covering dimension of~$X$
(Definition 3.1.1 of~\cite{Prs}) by $\dim (X)$.
If $h \colon X \to X$ is a \hme,
its mean dimension
(Definition 2.6 of~\cite{LndWs})
is denoted by $\mdim (h)$.
\end{ntn}

\section{Preliminaries on actions on $C (X, D)$ lying over actions
  on~$X$}\label{Sec_Over}

In the section,
we give a few basic facts about actions of
groups on \ca{s} of the form $C (X, D)$
which ``lie over'' actions on~$X$.
We also introduce several technical conditions which
will be needed as hypotheses later,
and give some cases in which they are automatically satisfied.

\begin{ntn}\label{N_5418_ActCX}
Let $G$ be a locally compact group and let $X$ be a
locally \chs~$X$ on which $G$ acts.
We take the corresponding action $\af \colon G \to \Aut (C_0 (X))$
to be be given by
$\af_g (f) (x) = f (g^{-1} x)$ for $f \in C (X)$,
$g \in G$, and $x \in X$.
For a \hme{} $h \colon X \to X$,
this means that the corresponding action of~$\Z$ on $C_0 (X)$
is generated by
the automorphism $\af (f) = f \circ h^{-1}$ for $f \in C_0 (X)$.
\end{ntn}

\begin{dfn}\label{D_4Y12_OverX}
Let $X$ be a locally \chs,
let $G$ be a topological group,
and let $D$ be a \ca.
Let $(g, x) \mapsto g x$ be an action of $G$ on~$X$,
and let $\af \colon G \to \Aut (C_0 (X, D))$
be an action of $G$ on $C_0 (X, D)$.
We say that
{\emph{$\af$ lies over the action $(g, x) \mapsto g x$}}
if there exists a function
$(g, x) \mapsto \af_{g, x}$ from $G \times X$ to $\Aut (D)$
such that $\af_g (a) (x) = \af_{g, x} (a (g^{-1} x) )$
for all $g \in G$, $x \in X$,
and $a \in C_0 (X, D)$.

We say that an automorphism $\af$ of $C_0 (X, D)$
{\emph{lies over}} a \hme{} $h \colon X \to X$
if the action generated by $\af$
lies over the action generated by~$h$.
\end{dfn}

In Definition~\ref{D_4Y12_OverX},
for $g \in G$ and $d \in D$,
the function $x \mapsto \af_{g, x} (d)$
must be \ct.
The following elementary lemma,
which will be used without comment,
shows that if $G$ is discrete
then this is the only continuity condition that is needed.
If $G$ is not discrete,
there are additional continuity conditions.

\begin{lem}\label{L_9Y30_MixCont}
Let $D$ be a \ca,
let $X$ be a locally \chs,
and let $x \mapsto \af_{x}$
be a \cfn{} from $X$ to $\Aut (D)$.
The for every $a \in C_0 (X, D)$,
the function $b (x) = \af_x (a (x))$ is also in $C_0 (X, D)$.
\end{lem}

\begin{proof}
It is immediate that $b$ vanishes at infinity.
For continuity,
let $x_0 \in X$ and let $\ep > 0$.
Choose an open set $U \S X$ such that $x_0 \in U$ and
for all $x \in U$ we have
\[
\| a (x) - a (x_0) \| < \frac{\ep}{2}
\andeqn
\| \af_x (a (x_0)) - \af_{x_0} (a (x_0)) \| < \frac{\ep}{2}.
\]
Then, using $\| \af_x \| = 1$ for all $x \in X$,
one sees that $x \in U$
implies $\| \af_x (a (x)) - \af_{x_0} (a (x_0)) \| < \ep$.
\end{proof}

For any group~$G$,
there are also algebraic conditions
relating the automorphisms $\af_{g, x}$,
coming from the requirement that $g \mapsto \af_g$
be a group \hm.
If $G = \Z$,
then we really need only the function $x \mapsto \af_{1, x}$.
For reference,
we give the relevant statement as a lemma.

\begin{lem}\label{L_4Y12_GetAuto}
Let $X$ be a locally \chs,
let $h \colon X \to X$ be a \hme,
let $D$ be a \ca.
Then there is a one to one correspondence between actions of $\Z$ on
$C_0 (X, D)$
that lie over $h$ and
continuous functions from $X$ to $\Aut (D)$,
given as follows.

For any function $x \mapsto \af_{x}$
from $X$ to $\Aut (D)$
such that $x \mapsto \af_{x} (d)$ is \ct{} for all $d \in D$,
there is an automorphism $\af \in \Aut ( C_0 (X, D))$
given by $\af (a) (x) = \af_x (a (h^{-1} (x))$
for all $a \in C_0 (X, D)$ and $x \in X$,
and this automorphism lies over~$h$.

Conversely,
if $\af \in \Aut ( C_0 (X, D))$ lies over~$h$,
then there is a function $x \mapsto \af_{x}$
from $X$ to $\Aut (D)$
such that $x \mapsto \af_{x} (d)$ is \ct{} for all $d \in D$
and such that $\af (a) (x) = \af_x (a (h^{-1} (x))$
for all $a \in C_0 (X, D)$ and $x \in X$.
\end{lem}

\begin{proof}
Using Lemma~\ref{L_9Y30_MixCont}, this is immediate.
\end{proof}

There is a conflict in the notation in Lemma~\ref{L_4Y12_GetAuto}:
if $n \in \Z$
then $\af_n$
is one of the automorphisms in the action on $C_0 (X, D)$
(namely $\af_n$),
while if $x \in X$
then $\af_x \in \Aut (D)$.
We use this notation anyway to avoid having more letters.
To distinguish the two uses,
take $\af$ to be the action
and write $x \mapsto \af_x$
when the function from $X$ to $\Aut (D)$ is intended.

If $D$ is prime,
then every action on $C_0 (X, D)$
lies over an action of $G$ on~$X$.

\begin{lem}\label{N_4Y12_IfDSimple}
Let $X$ be a locally \chs,
let $D$ be a prime \ca,
let $G$ be a topological group,
and let $\af \colon G \to \Aut (C_0 (X, D))$
be an action of $G$ on $C_0 (X, D)$.
Then there exists an action of $G$ on~$X$
such that $\af$ lies over this action.
\end{lem}

\begin{proof}
The action of $G$ on~$X$
is obtained from the identification
$X \cong \Prim ( C_0 (X, D) )$.
\end{proof}

\begin{prp}\label{P_4Y15_CPSimple}
Let $G$ be a discrete group,
let $X$ be a compact space,
and suppose $G$ acts on $X$ in such a way that
the action is minimal and
for every finite set $S \subset G \setminus \{ 1 \}$,
the set
\[
\big\{ x \in X \colon {\mbox{$g x \neq x$ for all $g \in S$}} \big\}
\]
is dense in~$X$.
Let $D$ be a simple \uca,
and let $\af \colon G \to \Aut (C (X, D))$
be an action of $G$ on $C (X, D)$
which lies over the given action of $G$ on $X$
(in the sense of Definition~\ref{D_4Y12_OverX}).
Then $C^*_{\mathrm{r}} \big( G, \, C (X, D), \, \af \big)$
is simple.
\end{prp}

\begin{proof}
For any \ca~$A$,
let ${\widehat{A}}$ be the space of unitary equivalence classes
of irreducible representations of~$A$,
with the hull-kernel topology.
Since the primitive ideals
of $C (X, D)$ are exactly the kernels of the point evaluations,
there is an obvious map $q \colon \CXDHat{X}{D} \to X$,
and the open sets in $\CXDHat{X}{D}$ are exactly
the sets $q^{-1} (U)$ for open sets $U \subset X$.
It is now immediate that for every finite set
$S \subset G \setminus \{ 1 \}$,
the set
\[
\big\{ x \in \CXDHat{X}{D} \colon
  {\mbox{$g x \neq x$ for all $g \in S$}} \big\}
\]
is dense in $\CXDHat{X}{D}$.
That is, the action of $G$ on $\CXDHat{X}{D}$
is topologically free in the sense of Definition~1 of~\cite{AS}.

Let $J \subset C^*_{\mathrm{r}} \big( G, \, C (X, D), \, \af \big)$
be a nonzero ideal.
Let
\[
\pi \colon C^* \big( G, \, C (X, D), \, \af \big)
      \to C^*_{\mathrm{r}} \big( G, \, C (X, D), \, \af \big)
\]
be the quotient map.
Theorem~1 of~\cite{AS}
implies that $\pi^{-1} (J)$ has nonzero intersection
with the canonical copy of $C (X, D)$
in $C^* \big( G, \, C (X, D), \, \af \big)$.
Therefore $J$ has nonzero intersection
with the canonical copy of $C (X, D)$
in $C^*_{\mathrm{r}} \big( G, \, C (X, D), \, \af \big)$.
Since $D$ is simple,
this intersection has the form $C_0 (U, D)$
for some nonempty open set $U \subset X$.

Since the action of $G$ on $X$ is minimal and $X$ is compact,
there exist $n \in \N$ and $g_1, g_2, \ldots, g_n \in \N$
such that the sets
$g_1^{-1} U, \, g_2^{-1} U, \, \ldots, g_n^{-1} U$ cover~$X$.
Choose
\[
f_1, f_2, \ldots, f_n
 \in C (X)
 \subset C (X, D)
 \subset C^*_{\mathrm{r}} \big( G, \, C (X, D), \, \af \big)
\]
such that $\supp (f_k) \subset g_k^{-1} U$ for $k = 1, 2, \ldots, n$
and $\sum_{k = 1}^n f_k = 1$.
For $k = 1, 2, \ldots, n$, the functions
$\af_{g_k} (f_k)$
are in $C_0 (U) \subset C_0 (U, D) \subset J$,
so
\[
1 = \sum_{k = 1}^n f_k
  = \sum_{k = 1}^n u_{g_k}^* \af_{g_k} (f_k) u_{g_k}
  \in J.
\]
So $J = C^*_{\mathrm{r}} \big( G, \, C (X, D), \, \af \big)$.
\end{proof}

\begin{prp}\label{P_5418_CPStFin}
Assume the hypotheses of Proposition~\ref{P_4Y15_CPSimple},
and in addition assume that
$G$ is amenable
and $D$ has a tracial state.
Then $C^*_{\mathrm{r}} \big( G, \, C (X, D), \, \af \big)$
has a tracial state
and is stably finite.
\end{prp}

\begin{proof}
Since $C^*_{\mathrm{r}} \big( G, \, C (X, D), \, \af \big)$
is simple by Proposition~\ref{P_4Y15_CPSimple},
it suffices to show that
$C^*_{\mathrm{r}} \big( G, \, C (X, D), \, \af \big)$
has a tracial state.
We know that the tracial state space
${\operatorname{T}} (C (X, D))$ is nonempty,
since one can compose a tracial state on~$D$
with a point evaluation
$C (X, D) \to D$.
Since $G$ is amenable,
combining Theorem 2.2.1 and 3.3.1 of \cite{Gr}
shows that $C (X, D)$ has a $G$-invariant tracial state~$\ta$.
Standard methods show that the composition of $\ta$
with the conditional expectation
from $C^*_{\mathrm{r}} \big( G, \, C (X, D), \, \af \big)$
to $C (X, D)$
is a tracial state on
$C^*_{\mathrm{r}} \big( G, \, C (X, D), \, \af \big)$.
\end{proof}

The following condition is a technical hypothesis
which we need for the proof of the main large subalgebra result
(Theorem~\ref{T_5418_AYLg}).

\begin{dfn}\label{D_4Y12_PsPer}
Let $D$ be a \ca,
and let $S \subset \Aut (D)$.
We say that $S$ is {\emph{pseudoperiodic}}
if for every $a \in D_{+} \setminus \{ 0 \}$
there is $b \in D_{+} \setminus \{ 0 \}$
such that for every $\af \in S \cup \{ \id_D \}$
we have $b \precsim \af (a)$.
\end{dfn}

The interpretation of pseudoperiodicity is roughly as follows.
Suppose $S \subset \Aut (D)$ is pseudoperidic.
Then there is no sequence $(\alpha_n)_{n \in \N}$
in $S$ for which there is a nonzero element
$\eta \in W (D)$
such that the sequence $(\alpha_{n} (\eta))_{n \in \N}$
becomes arbitrarily small in $W (D)$
in a heuristic sense.

We give some conditions which imply pseudoperiodicity.

\begin{lem}\label{L_4Y12_ApproxInn}
Let $D$ be a \uca.
Then the set of approximately inner automorphisms of $D$
is pseudoperiodic in the sense of Definition~\ref{D_4Y12_PsPer}.
\end{lem}

\begin{proof}
Let $a \in D_{+} \setminus \{ 0 \}$.
It suffices to prove that
$a \precsim \af (a)$ for every approximately inner automorphism
$\af \in \Aut (D)$.
To see this, let $\ep > 0$ be arbitrary,
and use approximate innerness
to chose a unitary $u \in D$ such that
$\| u \af (a) u^* - a \| < \ep$.
\end{proof}

\begin{lem}\label{Lem1_190201D}
Let $D$ be a simple \ca.
Let $S \S \Aut (D)$
be a subset which is compact in the topology
of pointwise convergence in the norm on~$D$.
Then $S$ is pseudoperiodic
in the sense of Definition~\ref{D_4Y12_PsPer}.
\end{lem}

\begin{proof}
Let $a \in D_{+} \setminus \{ 0 \}$.
\Wolog{} $\| a \| = 1$.
For $\bt \in S \cup \{ \id_D \}$ set
\[
U_{\bt}
 = \left\{ \af \in S \cup \{ \id_D \} \colon
    \| \af (a) - \bt (a) \| < \frac{1}{2} \right\}.
\]
By compactness,
there are $\bt_1, \bt_2, \ldots, \bt_n \in S \cup \{ \id_D \}$
such that $U_{\bt_1}, U_{\bt_2}, \ldots, U_{\bt_n}$
cover $S \cup \{ \id_D \}$.
Since $\bigl( \bt (a) - \frac{1}{2} \bigr)_{+} \neq 0$
for all $\bt \in \Aut (D)$,
by Lemma~2.6 of~\cite{PhLg} there is $b \in D_{+} \setminus \{ 0 \}$
such that $b \precsim \bigl( \bt_k (a) - \frac{1}{2} \bigr)_{+}$
for $j = 1, 2, \ldots, n$.
Let $\af \in S \cup \{ \id_D \}$.
Choose $k \in \{ 1, 2, \ldots, n \}$ such that $\af \in U_{\bt_k}$.
Then $\| \bt_k (a) - \af (a) \| < \frac{1}{2}$,
so
$b \precsim \bigl( \bt_k (a) - \frac{1}{2} \bigr)_{+}
   \precsim \af (a)$.
\end{proof}

The following result will not be used,
since large subalgebras are not used in our proofs
when $D$ is purely infinite.
It is included as a further example of pseudoperiodicity.

\begin{lem}\label{L_9212_PISCA}
Let $D$ be a purely infinite simple \ca.
Then $\Aut (D)$
is pseudoperiodic in the sense of Definition~\ref{D_4Y12_PsPer}.
\end{lem}

\begin{proof}
Let $a \in D_{+} \setminus \{ 0 \}$.
Then $a \precsim b$ for all $b \in D_{+} \setminus \{ 0 \}$.
In particular,
$a \precsim \af (a)$ for all $\af \in \Aut (D)$.
\end{proof}

\begin{lem}\label{L_4Y12_TrPrsv}
Let $D$ be a simple \uca{}
which has strict comparison of positive elements.
Then $\Aut (D)$
is pseudoperiodic in the sense of Definition~\ref{D_4Y12_PsPer}.
\end{lem}

\begin{proof}
If $D$ is \fd,
the conclusion is immediate.
Otherwise,
let $a \in D_{+} \setminus \{ 0 \}$.
\Wolog{} $\| a \| \leq 1$.
Then $\ta (a) \leq d_{\ta} (a)$
for all $\ta \in {\operatorname{QT}} (D)$.
Moreover,
since ${\operatorname{QT}} (D)$ is compact and $D$ is simple,
the number
$\dt = \inf_{\ta \in {\operatorname{QT}} (D)} \ta (a)$
satisfies $\dt > 0$.
Use Corollary~2.5 of~\cite{PhLg}
to find
$b \in D_{+} \setminus \{ 0 \}$
such that
$d_{\ta} (\langle b \rangle ) < \dt$
for all $\ta \in {\operatorname{QT}} (D)$.
Then
for every $\ta \in {\operatorname{QT}} (D)$,
using $\ta \circ \af \in {\operatorname{QT}} (D)$
at the second step,
we have
\[
d_{\ta} (b)
 < \dt
 \leq (\ta \circ \af) (a)
 \leq d_{\ta} ( \af (a) ).
\]
The strict comparison hypothesis therefore implies
that $b \precsim_A \af (a)$.
\end{proof}

\begin{dfn}\label{Def2-181221D}
Let $A$ be a \ca.
We say that the
{\emph{order on projections over $A$
is determined by quasitraces}}
if whenever $p,q \in M_{\infty} (A)$ are projections such that
$\tau (p) < \tau(q) $
for all $\tau \in \operatorname{QT} (A)$, then $p \precsim q$.
\end{dfn}

\begin{lem}\label{L_5417_SP}
Let $D$ be a simple \uca{}
with Property~(SP)
and such that
the order on projections over~$A$ is determined by quasitraces.
Then $\Aut (D)$
is pseudoperiodic in the sense of Definition~\ref{D_4Y12_PsPer}.
\end{lem}

\begin{proof}
If $D$ is \fd,
the conclusion is immediate.
Otherwise,
let $a \in D_{+} \setminus \{ 0 \}$.
Choose a \nzp{} $p \in {\overline{a D a}}$.
Since ${\operatorname{QT}} (D)$ is compact and $D$ is simple,
the number
$\dt = \inf_{\ta \in {\operatorname{QT}} (D)} \ta (p)$
satisfies $\dt > 0$.
Choose $n \in \N$
such that $\frac{1}{n} < \dt$.
Use Lemma~2.3 of~\cite{PhLg}
to choose a unitary $u \in A$
and a nonzero positive element $b \in A$
such that the elements
\[
b, \, u b u^{-1}, \, u^2 b u^{-2}, \, \ldots, \, u^n b u^{-n}
\]
are pairwise orthogonal.
Choose a \nzp{} $q \in {\overline{b D b}}$.
Then
for every $\ta \in {\operatorname{QT}} (D)$,
using $\ta \circ \af \in {\operatorname{QT}} (D)$
at the third step,
we have
\[
\ta (q)
 \leq \frac{1}{n + 1}
 < \dt
 \leq (\ta \circ \af) (p).
\]
The strict comparison hypothesis therefore implies
that $q$ is \mvnt{} to a subprojection of~$\af (p)$.
It follows that $q \precsim_A \af (a)$.
\end{proof}

The following definition is intended only for convenience
in this paper.
(The condition occurs several times as a hypothesis,
and is awkward to state.)

\begin{dfn}\label{D_4Y16_PsPrGen}
Let $X$ be a locally \chs,
let $G$ be a topological group,
and let $D$ be a \ca.
Let $\af \colon G \to \Aut (C_0 (X, D))$
be an action of $G$ on $C_0 (X, D)$
which lies over an action of $G$ on~$X$,
and let $(g, x) \mapsto \af_{g, x} \in \Aut (D)$
be as in Definition~\ref{D_4Y12_OverX}.
We say that $\af$ is
{\emph{pseudoperiodically generated}}
if
\[
\big\{ \af_{g, x} \colon {\mbox{$g \in G$ and $x \in X$}} \big\}
\]
is pseudoperiodic in the sense of Definition~\ref{D_4Y12_PsPer}.
\end{dfn}

\begin{lem}\label{L_4Y16_WhenPPG}
Let $X$ be a \cms, let $h \colon X \to X$ be a \hme,
let $D$ be a simple \uca,
and let $\af \in \Aut (C (X, D))$ lie over~$h$.
As in Lemma~\ref{L_4Y12_GetAuto},
let $( \af_{x} )_{x \in X}$
be the family in $\Aut (D)$
such that $\af (a) (x) = \af_x (a (h^{-1} (x))$
for all $a \in C_0 (X, D)$ and $x \in X$.
Suppose that the subgroup $H$ of $\Aut (D)$ generated by
$\{ \af_x \colon x \in X \}$ is pseudoperiodic.
Then the action of $\Z$ generated by $\af$
is pseudoperiodically generated.
\end{lem}

\begin{proof}
One checks that if $(n, x) \mapsto \af_{n, x} \in \Aut (D)$
is determined
(following the notation of Definition~\ref{D_4Y12_OverX})
by $\af^n (a) (x) = \af_{n, x} (a (h^{-n} (x))$,
then $\af_{n, x} \in H$ for all $n \in \Z$ and $x \in X$.
\end{proof}

\section{The orbit breaking subalgebra for a nonempty set meeting
each orbit at most once}\label{Sec_OrbBreak}

\indent
Let $h \colon X \to X$ be a \hme{}
of a \chs~$X$,
and let $D$ be a simple \uca.
For $Y \subset X$ closed,
following Putnam~\cite{Pt1},
in Definition~7.3 of~\cite{PhLg}
we defined the $Y$-orbit breaking subalgebra
$C^* (\Z, X, h)_Y \subset C^* (\Z, X, h)$.
Here,
for an automorphism $\af \in \Aut ( C (X, D) )$
which lies over~$h$
we define $C^* (\Z,  C (X, D), \af)_Y$.
We prove that if $X$ is infinite,
$h$~is minimal,
$Y$ intersects each orbit at most once,
and an additional technical condition is satisfied
(namely, that the action of $\Z$ generated by~$\af$
is pseudoperiodically generated),
then $C^* (\Z,  C (X, D), \af)_Y$
is a large subalgebra of $C^* (\Z,  C (X, D), \af)$
of crossed product type,
in the sense of Definition~4.9 of~\cite{PhLg}.
This is a generalization
of Theorem~7.10 of~\cite{PhLg}.

\begin{ntn}\label{N_4Y12_CPN}
Let $G$ be a discrete group,
let $A$ be a \ca,
and let $\af \colon G \to \Aut (A)$ be an action of $G$ on~$A$.
We identify $A$ with a subalgebra of $C^*_{\mathrm{r}} (G, A, \af)$
in the standard way.
We let $u_g \in M (C^*_{\mathrm{r}} (G, A, \af))$
be the standard unitary
corresponding to $g \in G$.
When $G = \Z$, we write just $u$ for the unitary~$u_1$
corresponding to the generator $1 \in \Z$.
We let $A [G]$ denote the dense *-subalgebra
of $C^*_{\mathrm{r}} (G, A, \af)$
consisting of sums $\sum_{g \in S} a_g u_g$
with $S \subset G$ finite and $a_g \in A$ for $g \in S$.
We may always assume $1 \in S$.
We let $E_{\af} \colon C^*_{\mathrm{r}} (G, A, \af) \to A$
denote the standard conditional expectation,
defined on $A [G]$ by
$E_{\af} \left( \sum_{g \in S} a_g u_g \right) = a_1$.
When $\af$ is understood, we just write~$E$.

When $G$ acts on a \chs~$X$,
we use obvious analogs of this notation for $C^*_{\mathrm{r}} (G, X)$,
with the action as in Notation~\ref{N_5418_ActCX}.
In particular,
if $G = \Z$ and the action of generated by a \hme{}
$h \colon X \to X$,
we have $u f u^* = f \circ h^{-1}$.
\end{ntn}

\begin{ntn}\label{N_4Y12_C0UD}
For a locally \chs~$X$
and a \ca~$D$,
we identify $C_0 (X, D) = C_0 (X) \otimes D$
in the standard way.
For an open subset $U \subset X$,
we use the abbreviation
\[
C_0 (U, D)
 = \big\{ a \in C_0 (X, D) \colon
     {\mbox{$a (x) = 0$ for all $x \in X \setminus U$}} \big\}
 \subset C_0 (X, D).
\]
This subalgebra is of course canonically isomorphic to
the usual algebra $C_0 (U, D)$ when $U$ is considered
as a locally \chs{} in its own right.
\end{ntn}

In particular,
if $Y \subset X$ is closed, then
\begin{equation}\label{Eq_4Y12_C0XYD}
C_0 (X \setminus Y, \, D)
 = \big\{ a \in C_0 (X, D) \colon
  {\mbox{$a (x) = 0$ for all $x \in Y$}} \big\}.
\end{equation}

The following definition is the analog of
Definition~7.3 of~\cite{PhLg}.

\begin{dfn}\label{D_4Y12_OrbSubalg}
Let $X$ be a locally \chs,
let $h \colon X \to X$ be a \hme,
let $D$ be a \ca,
and let $\af \in \Aut ( C_0 (X, D))$
be an automorphism which lies over~$h$.
Let $Y \subset X$ be a nonempty closed subset,
and, following~(\ref{Eq_4Y12_C0XYD}), define
\[
C^* \big( \Z, \, C_0 (X, D), \, \af \big)_Y
 = C^* \big(  C_0 (X, D), \, C_0 (X \setminus Y, \, D) u \big)
  \subset C^* \big( \Z, \, C_0 (X, D), \, \af \big).
\]
We call it the {\emph{$Y$-orbit breaking subalgebra}}
of $C^* \big( \Z,  C_0 (X, D), \af \big)$.
\end{dfn}

We show that if $Y$ intersects each orbit of~$h$ at most once,
and the action of $\Z$ generated by~$\af$
is pseudoperiodically generated,
then $C^* (\Z,  C_0 (X, D), \af)_Y$
is a large subalgebra of $C^* \big( \Z, \, C_0 (X, D), \, \af \big)$
of crossed product type.

The following lemma is the analog of Proposition~7.5 of~\cite{PhLg}.

\begin{lem}\label{L_4X21_AYStruct}
Let $X$ be a \chs, let $h \colon X \to X$ be a \hme,
let $D$ be a \uca,
and let $\af \in \Aut (C (X, D))$ lie over~$h$.
Let
\[
u \in C^* \bigl( \Z, \, C (X, D), \, \af \bigr)
\andeqn
E_{\af} \colon C^* \bigl( \Z, \, C (X, D), \, \af \bigr) \to C (X, D)
\]
be as in Notation~\ref{N_4Y12_CPN}.
Let $Y \subset X$ be a nonempty closed subset.
For $n \in \Z$, set
\[
Y_n = \begin{cases}
   \bigcup_{j = 0}^{n - 1} h^j (Y)    & \hspace{3em} n > 0
        \\
   \varnothing                        & \hspace{3em} n = 0
       \\
   \bigcup_{j = 1}^{- n} h^{-j} (Y)   & \hspace{3em} n < 0.
\end{cases}
\]
Then
\begin{align}\label{Eq:2816CharOB}
& C^* \bigl( \Z, \, C (X, D), \, \af \bigr)_Y
\\
& = \big\{ a \in C^* \bigl( \Z, \, C (X, D), \, \af \bigr) \colon
    {\mbox{$E_{\af} (a u^{-n}) \in C_0 (X \setminus Y_n, \, D)$
          for all $n \in \Z$}} \big\}
       \notag
\end{align}
and
\begin{equation}\label{Eq:2816Dense}
{\overline{C^* \bigl( \Z, \, C (X, D), \, \af \bigr)_Y \cap
C (X, D) [\Z]}}
 = C^* \bigl( \Z, \, C (X, D), \, \af \bigr)_Y.
\end{equation}
\end{lem}

\begin{proof}
By Lemma~\ref{L_4Y12_GetAuto},
there exists a function $x \mapsto \af_{x}$
from $X$ to $\Aut (D)$
such that $x \mapsto \af_{x} (d)$ is \ct{} for all $d \in D$
and which satisfies $\af (a) (x) = \af_x (a (h^{-1} (x))$
for all $a \in C_0 (X, D)$ and $x \in X$.

Most of the proof of Proposition~7.5 of~\cite{PhLg}
goes through with only the obvious changes.
In analogy with that proof,
define
\[
B = \big\{ a \in C^* \bigl( \Z, \, C (X, D), \, \af \bigr) \colon
    {\mbox{$E_{\af} (a u^{-n}) \in C_0 (X \setminus Y_n, \, D)$
          for all $n \in \Z$}} \big\}
\]
and
\[
B_0 = B \cap C (X) [\Z].
\]
Then $B_0$ is dense in~$B$ by the same reasoning as in~\cite{PhLg}
(using Ces\`{a}ro means and Theorem VIII.2.2 of~\cite{Dv}).

The proof of Proposition~7.5 of~\cite{PhLg}
shows that
when $0 \leq m \leq n$ and also when $0 \geq m \geq n$,
we have $Y_m \subset Y_n$,
that for $n \in \Z$,
we have
\begin{equation}\label{Eq_4Y14_hnyn}
h^{-n} (Y_n) = Y_{-n},
\end{equation}
and that for $m, n \in \Z$,
we have $Y_{m + n} \subset Y_m \cup h^m (Y_n)$.
It then follows,
as in~\cite{PhLg},
that $B_0$ is a *-algebra,
and that
$C^* \bigl( \Z, \, C (X, D), \, \af \bigr)_Y
 \subset {\overline{B_0}} = B$.

We next claim that for all $n \in \Z$
and $f \in C_0 (X \setminus Y_n, \, D)$,
we have $f u^n \in C^* \bigl( \Z, \, C (X, D), \, \af \bigr)_Y$.
The changes to the proof of Proposition~7.5 of~\cite{PhLg}
at this point are more substantial.

For $n = 0$ the claim is trivial.
Let $n > 0$
and let $f \in C_0 (X \setminus Y_n, \, D)$.
Define $b = (f^* f)^{1 / (2 n)}$.
Let $s \in C_0 (X \setminus Y_n, \, D)''$
be the partial isometry in the polar decomposition of~$f$,
so that $f = s (f^* f)^{1 / 2} = s b^n$.
It follows from Proposition~1.3 of~\cite{Cu0}
that the element $a_0 = s b$ is in $C_0 (X \setminus Y_n, \, D)$.
Moreover,
$a_0 (f^* f)^{\frac{1}{2} - \frac{1}{2 n}} = f$.
Define $a_1 \in C (X, D)$
by $a_1 (x) = \af_{h (x)}^{-1} \big( b (h (x)) \big)$
for $x \in X$,
and for $k = 1, 2, \ldots, n - 2$ inductively define
$a_{k + 1} \in C (X, D)$
by $a_{k + 1} (x) = \af_{h (x)}^{-1} \big( a_k (h (x)) \big)$
for $x \in X$.
The definition of $Y_n$
implies that
$a_1, a_2, \ldots, a_{n - 1} \in C_0 (X \setminus Y, \, D)$,
and we already have
$a_0 \in C_0 (X \setminus Y_n, \, D) \S C_0 (X \setminus Y, \, D)$.
Therefore the element
\[
a = (a_0 u) (a_1 u) \cdots (a_{n - 1} u)
\]
is in $C^* \bigl( \Z, \, C (X, D), \, \af \bigr)_Y$.
For $k = 1, 2, \ldots, n - 2$ and $x \in X$, we have
\[
\af (a_{k + 1} ) (x)
 = \af_x (a_{k + 1} (h^{-1} (x))
 = \af_x \big( \af_{x}^{-1} ( a_k (x) ) \big)
 = a_k (x),
\]
so $\af (a_{k + 1} ) = a_k$.
Similarly, $\af (a_1) = b$.
Now
\begin{align*}
a
& = a_0 (u a_1 u^{-1}) (u^2 a_2 u^{-2})
    \cdots \big( u^{n - 1} a_{n - 1} u^{- (n - 1)} \big) u^n
  \\
& = a_0 \af (a_1) \af^2 (a_2)
    \cdots \af^{n - 1} (a_{n - 1}) u^n
  = (s b) b^{n - 1} u^n
  = f u^n.
\end{align*}
So $f u^n \in C^* \bigl( \Z, \, C (X, D), \, \af \bigr)_Y$.
Finally,
suppose $n < 0$,
and let $f \in C_0 (X \setminus Y_n, \, D)$.
It follows from~(\ref{Eq_4Y14_hnyn})
that $f \circ h^n \in C_0 (X \setminus Y_{- n}, \, D)$,
whence also $(f \circ h^n)^* \in C_0 (X \setminus Y_{- n}, \, D)$.
Since $- n > 0$,
we therefore get
\[
f u^n = \big( u^{-n} f^* \big)^*
      = \big( (f \circ h^n)^* u^{- n} \big)^*
      \in C^* \bigl( \Z, \, C (X, D), \, \af \bigr)_Y.
\]
The claim is proved.

It now follows that
$B_0 \subset C^* \bigl( \Z, \, C (X, D), \, \af \bigr)_Y$.
Combining this result with ${\overline{B_0}} = B$
and $C^* \bigl( \Z, \, C (X, D), \, \af \bigr)_Y \subset B$,
we get $C^* \bigl( \Z, \, C (X, D), \, \af \bigr)_Y = B$.
\end{proof}

The following lemma is the analog of Lemma~7.8 of~\cite{PhLg}.

\begin{lem}\label{L_4X21_78}
Let $G$ be a discrete group,
let $X$ be a compact space,
and suppose $G$ acts on $X$ in such a way that
for every finite set $S \subset G \setminus \{ 1 \}$,
the set
\[
\big\{ x \in X \colon {\mbox{$g x \neq x$ for all $g \in S$}} \big\}
\]
is dense in~$X$.
Let $D$ be a \uca,
and let $\af \colon G \to \Aut (C (X, D))$
be an action of $G$ on $C (X, D)$
which lies over the given action of $G$ on $X$
(in the sense of Definition~\ref{D_4Y12_OverX}).
Following Notation~\ref{N_4Y12_CPN},
let
$a \in C (X, D) [G]
 \subset C^*_{\mathrm{r}} \big( G, \, C (X, D), \, \af \big)$
and let $\ep > 0$.
Then there exists $f \in C (X) \subset C (X, D)$
such that
\[
0 \leq f \leq 1,
\qquad
f a^* a f \in C (X, D),
\andeqn
\| f a^* a f \| \geq \| E_{\af} (a^* a) \| - \ep.
\]
\end{lem}

\begin{proof}
We follow the proof of Lemma~7.8 of~\cite{PhLg}.
Set $b = a^* a$.
There are a finite set $T \subset G$
with $1 \in T$,
and elements $b_g \in C (X, D)$ for $g \in T$,
such that $b = \sum_{g \in T} b_g u_g$.
If $\| b_1 \| \leq \ep$ take $f = 0$.
Otherwise, define
\[
U = \big\{ x \in X \colon
   \| b_1 (x) \| > \| E (a^* a) \| - \ep \big\},
\]
which is a nonempty open subset of~$X$.
Choose $V$, $W$, $f$, and $x_0$
as in the proof of Lemma~7.8 of~\cite{PhLg}.
Then,
as there, $f b f = f b_1 f$.
Moreover,
\[
\| f b_1 f \|
 \geq f (x_0) \| b_1 (x_0) \| f (x_0)
 = \| b_1 (x_0) \|
 > \| E_{\af} (a^* a) \| - \ep.
\]
This completes the proof.
\end{proof}

The following lemma is the analog of Lemma~7.9 of~\cite{PhLg}.

\begin{lem}\label{L_4X21_79}
Let $G$, $X$, the action of $G$ on $X$,
$D$, and $\af \colon G \to \Aut (C (X, D))$
be as in Lemma~\ref{L_4X21_78}.
Let $B \subset C^*_{\mathrm{r}} \big( G, \, C (X, D), \, \af \big)$
be a unital subalgebra
such that,
following Notation~\ref{N_4Y12_CPN},
\begin{enumerate}
\item\label{L_4X21_79:CXI}
$C (X, D) \subset B$.
\item\label{L_4X21_79:FSD}
$B \cap C (X, D)[G]$ is dense in~$B$.
\end{enumerate}
Let $a \in B_{+} \setminus \{ 0 \}$.
Then there exists $b \in C (X, D)_{+} \setminus \{ 0 \}$
such that $b \precsim_{B} a$.
\end{lem}

\begin{proof}
We follow the proof of Lemma~7.9 of~\cite{PhLg},
using our Lemma~\ref{L_4X21_78}
in place of Lemma~7.8 of~\cite{PhLg},
except that $c \in B \cap C (X, D)[G]$.
The element $(f c^* c f - 2 \ep)_{+}$
we obtain now satisfies
$(f c^* c f - 2 \ep)_{+} \in C (X, D)_{+} \setminus \{ 0 \}$
and $(f c^* c f - 2 \ep)_{+} \precsim_{B} a$.
\end{proof}

In Lemma~\ref{L_4X21_79},
we really want to have $b \in C (X)_{+} \setminus \{ 0 \}$.
When $G = \Z$ and under the pseudoperiodicity hypothesis
of \Def{D_4Y16_PsPrGen}, this is possible.

\begin{lem}\label{L_4Y11_Comp}
Let $X$ be a \cms, let $h \colon X \to X$ be a minimal homeomorphism,
let $D$ be a simple \uca,
and let $\af \in \Aut (C (X, D))$ lie over~$h$.
Assume that the action generated by $\af$
is pseudoperiodically generated.
Let $Y \subset X$ be a compact set
such that $h^n (Y) \cap Y = \varnothing$
for all $n \in \Z \setminus \{ 0 \}$.
Set $B = C^* \bigl( \Z, \, C (X, D), \, \af \bigr)_Y$.
Then for every $a \in C (X, D)_{+} \setminus \{ 0 \}$
there exists $f \in C (X)_{+} \setminus \{ 0 \} \subset B$
such that $f \precsim_B a$.
\end{lem}

\begin{proof}
Use Kirchberg's Slice Lemma (Lemma 4.1.9 of~\cite{Rrd})
to find $g_0 \in C (X)_{+} \setminus \{ 0 \}$
and $d_0 \in D_{+} \setminus \{ 0 \}$
such that
\begin{equation}\label{Eq_5417_Zero}
g_0 \otimes d_0 \precsim_{C (X, D)} a.
\end{equation}
As in Definition~\ref{D_4Y12_OverX},
let $(m, x) \mapsto \af_{m, x}$
be the function $\Z \times X \to \Aut (D)$
such that $\af^m (a) (x) = \af_{m, x} (a (h^{-m} (x))$
for $m \in \Z$, $a \in C_0 (X, D)$, and $x \in X$.
Since the action generated by $\af$
is pseudoperiodically generated,
there exists $d_1\in D_{+} \setminus \{ 0 \}$
such that
\begin{equation}\label{Eq_5418_PsPer}
\af_{m, x} (d_1) \precsim d_0
\end{equation}
for all $x \in X$ and $m \in \Z$.
\Wolog{} $\| d_1 \| = 1$.
Set $d = \big( d_1 - \tfrac{1}{2} \big)_{+}$.
Use Corollary 1.14 of~\cite{PhLg}
to find $k \in \N$ and $w_1, w_2, \ldots, w_k \in D$
such that
\begin{equation}\label{Eq_5418_Sum1}
\sum_{j = 1}^k w_j d w_j^* = 1.
\end{equation}

The set $X \setminus Y$ is dense in~$X$,
so there is $x_0 \in X \setminus Y$ such that $g_0 (x_0) \neq 0$.
Choose $g_1 \in C (X)_{+}$ such that $g_1 (x_0) = 1$
and $g_1 |_Y = 0$.
Then
\begin{equation}\label{Eq_5417_One}
g_1 g_0 \otimes d_0
 \precsim_{C (X, D)} g_0 \otimes d_0.
\end{equation}
Set
$U_0 = \big\{ x \in X \colon (g_1 g_0) (x) \neq 0 \big\}$.
Choose a nonempty open set $U \subset X$
such that ${\overline{U}} \subset U_0$.
Set
\begin{equation}\label{Eq_9Z01_NewLabel}
\rh = \inf_{x \in {\overline{U}}} (g_0 g_1) (x).
\end{equation}
Then $\rh > 0$.
The set
\[
X \setminus \bigcup_{n \in \Z} h^n (Y)
 = \bigcap_{n \in \Z} h^n (X \setminus Y)
\]
is dense by the Baire Category Theorem.
So we can choose
$y_0 \in U \cap \left( X \setminus
\bigcup_{n \in \Z} h^n (Y) \right)$.
The forward orbits $\{ h^n (y_0) \colon n \geq N \}$
of $y_0$ are all dense,
so there exist $n_1, n_2, \ldots, n_k \in \Nz$
with $0 = n_1 < n_2 < \cdots< n_k$
and $h^{n_j} (y_0) \in U$ for $j = 1, 2, \ldots, k$.

Choose an open set $V$ containing $y_0$ which is so small that
the following hold:
\begin{enumerate}
\item\label{L_4Y11_Comp_Cont}
${\overline{V}} \subset X \setminus
\bigcup_{m = 0}^{n_k} h^{- m} (Y)$.
\item\label{L_4Y11_Comp_InU}
$h^{n_j} ( {\overline{V}} ) \subset U$
for $j = 1, 2, \ldots, k$.
\item\label{L_4Y11_Comp_Disj}
The sets
${\overline{V}}, \, h ( {\overline{V}} ), \, \ldots, \,
   h^{n_k} ( {\overline{V}} )$
are disjoint.
\item\label{L_4Y11_Comp_Small}
For $y \in V$ and $m = 0, 1, \ldots, n_k$,
we have
\[
\bigl\| \af_{m, h^m (y)} (d_1) - \af_{m, h^m (y_0)} (d_1) \bigr\|
< \frac{1}{4}.
\]
\setcounter{TmpEnumi}{\value{enumi}}
\end{enumerate}
Choose $f_0, f \in C (X)_{+}$
such that:
\begin{enumerate}
\setcounter{enumi}{\value{TmpEnumi}}
\item\label{L_4Y11_Comp_1}
$\| f_0 \|, \, \| f \| \leq 1$.
\item\label{L_4Y11_Comp_Supp}
$\supp (f_0) \subset V$.
\item\label{L_4Y11_Comp_b0b}
$f_0 f = f$.
\item\label{L_4Y11_Comp_Aty0}
$f (y_0) = 1$.
\end{enumerate}
For $m = 0, 1, \ldots, n_k - 1$,
define $c_m \in C (X, D)$
by $c_m (x) = f ( h^{- m} (x)) \af_{m, x} (d)$
for $x \in X$.
Thus
\begin{equation}\label{Eq_5417_cm}
c_m = \af^m (f \otimes d).
\end{equation}
Also set
\begin{equation}\label{Eq_5418_DefBt}
\bt_m = \af_{m, h^m (y_0)}.
\end{equation}

We claim that for $m = 0, 1, \ldots, n_k$
we have
\[
c_m \precsim_{C (X, D)} (f \circ h^{-m}) \otimes \bt_m (d_1).
\]

To prove the claim, let $\ep > 0$.
Set $L = h^{m} ( \supp (f_0) )$.
Define $b \in C (L, D)$
by $b (x) = \af_{m, x} (d_1)$ for $x \in L$.
Since $L \subset h^m (V)$,
condition~(\ref{L_4Y11_Comp_Small})
implies that
for $x \in L$
we have $\| \af_{m, x} (d_1) - \bt_m (d_1) \| < \frac{1}{2}$.
In $C (L, D)$ we therefore get
$\| b - 1 \otimes \bt_m (d_1) \|
  \leq \frac{1}{4} < \frac{1}{2}$.
So
\[
\bigl( b - \tfrac{1}{2} \bigr)_{+}
  \precsim_{C (L, D)} 1 \otimes \bt_m (d_1).
\]
Therefore there exists $v_0 \in C (L, D)$
such that
\[
\big\| v_0^* ( 1 \otimes \bt_m (d_1) ) v_0
   - \bigl( b - \tfrac{1}{2} \bigr)_{+} \big\|
 < \frac{\ep}{2}.
\]
Define $v, a \in C (X, D)$ by
\[
v (x) = \begin{cases}
   f_0 (h^{- m} (x)) v_0 (x) & x \in L
        \\
   0                     & x \not\in L
\end{cases}
\andeqn
a = v^* \bigl[ (f \circ h^{-m}) \otimes \bt_m (d_1) \bigr] v.
\]
We show that $\| a - c_m \| < \ep$.

For $x \in X \setminus L$,
we have $a (x) = c_m (x) = 0$.
For $x \in L$,
using $f_0 f = f$ at the second step
and using $d = \bigl( d_1 - \tfrac{1}{2} \bigr)_{+}$
and the definition of~$b$ at the third step,
we get
\begin{align*}
& \| a (x) - c_m (x) \|
\\
& \hspace*{3em} {\mbox{}}
  = \big\| f_0 (h^{- m} (x))^2 f (h^{- m} (x)) v_0 (x)^*
          \bt_m (d_1) v_0 (x)
       - f (h^{- m} (x)) \af_{m, x} (d) \big\|
\\
& \hspace*{3em} {\mbox{}}
  = \big\| f (h^{- m} (x))
         [ v_0 (x)^* \bt_m (d_1) v_0 (x) - \af_{m, x} (d) ] \big\|
\\
& \hspace*{3em} {\mbox{}}
  = f (h^{- m} (x))
       \big\| v_0 (x)^* \bt_m (d_1) v_0 (x)
              - \bigl( b (x) - \tfrac{1}{2} \bigr)_{+} \big\|
\\
& \hspace*{3em} {\mbox{}}
  \leq f (h^{- m} (x))
        \big\| v_0^* ( 1 \otimes \bt_m (d_1) ) v_0
            - \bigl( b - \tfrac{1}{2} \bigr)_{+} \big\|
  < \frac{\ep}{2}.
\end{align*}
Taking the supremum over $x \in X$ gives
$\| a - c_m \| \leq \frac{\ep}{2} < \ep$.
Thus
\[
\big\| v^* [ (f \circ h^{-m}) \otimes \bt_m (d_1) ] v - c_m \big\|
 < \ep.
\]
Since $\ep > 0$ is arbitrary,
the claim follows.

For $m = 0, 1, \ldots, n_k$,
the functions $f \circ h^{-m}$ are orthogonal
since the sets $h^m ( {\overline{V}} )$ are disjoint.
The claim therefore implies that
\begin{equation}\label{Eq_5417_Two}
\sum_{j = 1}^k c_{n_j}
 \precsim_{C (X, D)}
     \sum_{j = 1}^k (f \circ h^{-n_j}) \otimes \bt_{n_j} (d_1).
\end{equation}

We now claim that
\begin{equation}\label{Eq_5417_Three}
f \otimes 1 \precsim_{B} \sum_{j = 1}^k c_{n_j}.
\end{equation}
To prove this claim,
for $j = 1, 2, \ldots, k$
define
\[
v_j = [(f \circ h^{-n_j}) \otimes 1] u^{n_j} (1 \otimes w_j)^*
    \in C^* \bigl( \Z, \, C (X, D), \, \af \bigr).
\]
Combining (\ref{L_4Y11_Comp_Supp}) and~(\ref{L_4Y11_Comp_Cont}),
we see that $f_0$ vanishes in particular on the sets
\[
h^{-1} ( Y ), \, h^{-2} ( Y ), \, \ldots, \, h^{- n_j} ( Y ),
\]
whence $f_0 \circ h^{- n_j}$ vanishes on the sets
\[
Y, \, h ( Y ), \, \ldots, \, h^{n_j - 1} ( Y ).
\]
So $v_j \in B$ by
\Lem{L_4X21_AYStruct}.
Using (\ref{Eq_5417_cm}) at the first step
and $f_0 f = f$ at the last step,
we calculate:
\begin{align*}
v_j^* c_{n_j} v_j
& = (1 \otimes w_j) u^{- n_j} [(f \circ h^{-n_j}) \otimes 1]
     \af^{n_j} (f \otimes d) [(f \circ h^{-n_j}) \otimes 1]
     u^{n_j} (1 \otimes w_j)^*
\\
& = (1 \otimes w_j) (f_0 \otimes 1) (f \otimes d)
     (f_0 \otimes 1) (1 \otimes w_j^*)
  = f \otimes w_j d w_j^*.
\end{align*}
We apply (\ref{Eq_5418_Sum1}) at the first step
and use orthogonality
of $c_0, c_1, \ldots, c_{n_k}$
and Lemma 1.4(12) of~\cite{PhLg} at the second step,
to get
\[
f \otimes 1
 = \sum_{j = 1}^k f \otimes w_j d w_j^*
 \precsim_B \sum_{j = 1}^k c_{n_j}.
\]
This proves the claim~(\ref{Eq_5417_Three}).

We next claim that
\begin{equation}\label{Eq_5417_Four}
\sum_{j = 1}^k (f \circ h^{-n_j}) \otimes \bt_{n_j} (d_1)
 \precsim_{C (X, D)} g_0 g_1 \otimes d.
\end{equation}
To prove this,
combine (\ref{Eq_5418_PsPer}), (\ref{Eq_5418_DefBt}),
and Lemma 1.11 of~\cite{PhLg}
to get
\[
(f \circ h^{-n_j}) \otimes \bt_{n_j} (d_1)
 \precsim_{C (X, D)} (f \circ h^{-n_j}) \otimes d_0
\]
for $j = 1, 2, \ldots, k$.
Since the functions
$f \circ h^{-n_j}$ are orthogonal,
it follows that
\[
\sum_{j = 1}^k (f \circ h^{-n_j}) \otimes \bt_{n_j} (d_1)
 \precsim_{C (X, D)}
 \left( \sum_{j = 1}^k (f \circ h^{-n_j}) \right) \otimes d_0.
\]
Using orthogonality
of the functions
$f \circ h^{-n_j}$ again,
together with
\[
\supp (f \circ h^{-n_j}) \subset h^{n_j} ( {\overline{V}} )
 \subset U
\andeqn
0 \leq f \circ h^{-n_j} \leq 1
\]
for $j = 1, 2, \ldots, k$,
we see that
$0 \leq \sum_{j = 1}^k f \circ h^{-n_j} \leq 1$
and,
by~(\ref{L_4Y11_Comp_InU}), that this sum is supported in~$U$.
Since $g_0 g_1 \geq \rh$ on ${\overline{U}}$
by~(\ref{Eq_9Z01_NewLabel}),
we get the first step in the following computation;
the second step is clear:
\[
\sum_{j = 1}^k (f \circ h^{-n_j}) \otimes d_0
 \leq \frac{1}{\rh} ( g_0 g_1 \otimes d_0 )
 \sim_{C (X, D)} g_0 g_1 \otimes d_0.
\]
The claim is proved.

Now combine
(\ref{Eq_5417_Zero}),
(\ref{Eq_5417_One}),
(\ref{Eq_5417_Four}),
(\ref{Eq_5417_Two}),
and
(\ref{Eq_5417_Three})
to get $f \otimes 1 \precsim_{B} a$.
\end{proof}

\begin{cor}\label{C_5418_CompCX}
Let $X$ be a \cms, let $h \colon X \to X$ be a minimal homeomorphism,
let $D$ be a simple \uca,
and let $\af \in \Aut (C (X, D))$ lie over~$h$.
Assume that the action generated by $\af$
is pseudoperiodically generated.
Let $Y \subset X$ be a compact set
(possibly empty)
such that $h^n (Y) \cap Y = \varnothing$
for all $n \in \Z \setminus \{ 0 \}$.
Set $B = C^* \bigl( \Z, \, C (X, D), \, \af \bigr)_Y$.
Let $a \in B_{+} \setminus \{ 0 \}$.
Then there exists $f \in C (X)_{+} \setminus \{ 0 \} \subset B$
such that $f \precsim_B a$.
\end{cor}

\begin{proof}
\Lem{L_4X21_AYStruct}
implies that $B$ satisfies the hypotheses
of \Lem{L_4X21_79}
with $G = \Z$.
Therefore,
by \Lem{L_4X21_79},
there exists $b \in C (X, D)_{+} \setminus \{ 0 \}$
such that $b \precsim_{B} a$.
Now \Lem{L_4Y11_Comp}
provides $f \in C (X)_{+} \setminus \{ 0 \}$
such that $f \precsim_B a$.
\end{proof}

We can now prove the analog of Theorem~7.10 of~\cite{PhLg}.

\begin{thm}\label{T_5418_AYLg}
Let $X$ be a \cms, let $h \colon X \to X$ be a minimal homeomorphism,
let $D$ be a simple \uca{} which has a tracial state,
and let $\af \in \Aut (C (X, D))$ lie over~$h$.
Assume that the action generated by $\af$
is pseudoperiodically generated.
Let $Y \subset X$ be a compact subset
such that $h^n (Y) \cap Y = \varnothing$
for all $n \in \Z \setminus \{ 0 \}$.
Then $C^* \bigl( \Z, \, C (X, D), \, \af \bigr)_Y$
is a large subalgebra of $C^* \bigl( \Z, \, C (X, D), \, \af \bigr)$
of crossed product type
in the sense of Definition~4.9 of~\cite{PhLg}.
\end{thm}

\begin{proof}
We verify the hypotheses of
Proposition~4.11 of~\cite{PhLg}.
We follow Notation~\ref{N_4Y12_CPN}.
Set
\[
A = C^* \bigl( \Z, \, C (X, D), \, \af \bigr),
\qquad
B = C^* \bigl( \Z, \, C (X, D), \, \af \bigr)_Y,
\]
\[
C = C (X, D),
\andeqn
G = \{ u \}.
\]

The algebra $A$ is simple
by Proposition~\ref{P_4Y15_CPSimple}
and finite by Proposition~\ref{P_5418_CPStFin}.
In particular,
condition~(1)
in Proposition 4.11 of~\cite{PhLg} holds.

We next verify condition~(2)
in Proposition 4.11 of~\cite{PhLg}.
All parts are obvious except~(2d).
So let $a \in A_{+} \setminus \{ 0 \}$
and $b \in B_{+} \setminus \{ 0 \}$.
Apply Corollary~\ref{C_5418_CompCX} twice,
the first time with $Y = \varnothing$ and $a$ as given
and the second time with $Y$ as given
and with $b$ in place of~$a$.
We get $a_0, b_0 \in C (X)_{+} \setminus \{ 0 \}$
such that $a_0 \precsim_{A} a$ and $b_0 \precsim_{B} b$.

We can now argue as in the corresponding part of the
proof of Theorem~7.10 of~\cite{PhLg}
to find $f \in C (X)$
such that
\[
f \precsim_{C (X)} a_0 \precsim_{A} a
\andeqn
f \precsim_{B} b_0 \precsim_{B} b,
\]
completing the proof of condition~(2d).
Alternatively,
$a_0, b_0
 \in C^* (\Z, X, h)_Y
 \subset C^* \bigl( \Z, \, C (X, D), \, \af \bigr)_Y$,
and $C^* (\Z, X, h)_Y$ is a large subalgebra of $C^* ( \Z, X, h)$
of crossed product type
by Theorem~7.10 of~\cite{PhLg},
so the existence of $f$
follows from condition~(2d)
in Proposition 4.11 of~\cite{PhLg}.

We next verify condition~(3)
in Proposition 4.11 of~\cite{PhLg}.
Let $m \in \N$,
let $a_1, a_2, \ldots, a_m \in A$,
let $\ep > 0$,
and let $b \in B_{+} \setminus \{ 0 \}$.
We follow the corresponding part of the
proof of Theorem~7.10 of~\cite{PhLg}.
Choose $c_1, c_2, \ldots, c_m \in C (X, D) [\Z]$
such that $\| c_j - a_j \| < \ep$
for $j = 1, 2, \ldots, m$.
(This estimate is condition~(3b).)
Choose $N \in \N$
such that for $j = 1, 2, \ldots, m$
there are $c_{j, l} \in C (X, D)$
for $l = -N, \, - N + 1, \, \ldots, \, N - 1, \, N$
with
\[
c_j = \sum_{l = -N}^N c_{j, l} u^l.
\]

Apply Corollary~\ref{C_5418_CompCX} to~$B$,
to find $f \in C (X)_{+} \setminus \{ 0 \}$
such that $f \precsim_{B} b$.
Set $U = \{ x \in X \colon f (x) \neq 0 \}$,
and choose nonempty disjoint open sets
$U_l \subset U$ for $l = -N, \, - N + 1, \, \ldots, \, N - 1, \, N$.
For each such~$l$,
use Lemma~7.7 of~\cite{PhLg}
to choose $f_l, r_l \in C (X)_{+}$
such that $r_l (x) = 1$ for all $x \in h^l (Y)$,
such that $0 \leq r_l \leq 1$,
such that $\supp (f_l) \subset U_l$,
and such that $r_l \precsim_{C^* (\Z, X, h)_Y} f_l$.
Then also $r_l \precsim_{B} f_l$.

Choose an open set $W$ containing~$Y$
such that
\[
h^{-N} (W), \, h^{- N + 1} (W),
 \, \ldots, \, h^{N - 1} (W), \, h^{N} (W)
\]
are disjoint,
and choose $r \in C (X)$ such that $0 \leq r \leq 1$,
$r (x) = 1$ for all $x \in Y$,
and $\supp (r) \subset W$.
Set
\[
g_0 = r \cdot \prod_{l = - N}^{N} r_l \circ h^{l}.
\]
Set $g_l = g_0 \circ h^{- l}$
for $l = -N, \, - N + 1, \, \ldots, \, N - 1, \, N$.
Then $0 \leq g_l \leq r_l \leq 1$.
Set $g = \sum_{l = - N}^{N} g_l$.
The supports of the functions $g_l$ are disjoint,
so $0 \leq g \leq 1$.
This is condition~(3a)
in Proposition 4.11 of~\cite{PhLg}.
The proof of Theorem~7.10 of~\cite{PhLg}
shows that
$g \precsim_{C^* (\Z, X, h)_Y} f \precsim_B b$.
Since $C^* (\Z, X, h)_Y \subset B$,
it follows that $g \precsim_B b$.
This is condition~(3d)
in Proposition 4.11 of~\cite{PhLg}.

It remains to verify condition~(3c)
in Proposition 4.11 of~\cite{PhLg}.
This is done the same way as
the corresponding part of the proof of Theorem~7.10 of~\cite{PhLg},
except using \Lem{L_4X21_AYStruct}
in place of Proposition~7.5 of~\cite{PhLg}.
\end{proof}

The following corollary is the analog of Corollary~7.11
of~\cite{PhLg}.

\begin{cor}\label{C_5418_AYStabLg}
Let $X$ be a \cms, let $h \colon X \to X$ be a minimal
homeomorphism, let $D$ be a simple \uca{}
which has a tracial state,
and let $\af \in \Aut (C (X, D))$ lie over~$h$.
Assume that the action generated by $\af$
is pseudoperiodically generated.
Let $Y \subset X$ be a compact subset
such that $h^n (Y) \cap Y = \varnothing$
for all $n \in \Z \setminus \{ 0 \}$.
Then $C^* \bigl( \Z, \, C (X, D), \, \af \bigr)_Y$
is a stably large subalgebra
of $C^* \bigl( \Z, \, C (X, D), \, \af \bigr)$
in the sense of Definition~5.1 of~\cite{PhLg}.
\end{cor}

\begin{proof}
Since $C^* \bigl( \Z, \, C (X, D), \, \af \bigr)$ is stably finite
(by Proposition~\ref{P_5418_CPStFin}),
we can combine Theorem~\ref{T_5418_AYLg},
Proposition~4.10 of~\cite{PhLg},
and Corollary~5.8 of~\cite{PhLg}.
\end{proof}

\begin{cor}\label{AyCentLarge}
Let $X$ be a \cms, let $h \colon X \to X$ be a minimal
homeomorphism, let $D$ be a simple \uca{}
which has a tracial state,
and let $\af \in \Aut (C (X, D))$ lie over~$h$.
Assume that the action generated by $\af$
is pseudoperiodically generated.
Let $Y \subset X$ be a compact subset
such that $h^n (Y) \cap Y = \varnothing$
for all $n \in \Z \setminus \{ 0 \}$.
Then $C^* \bigl( \Z, \, C (X, D), \, \af \bigr)_Y$
is a centrally large subalgebra
of $C^* \bigl( \Z, \, C (X, D), \, \af \bigr)$
in the sense of Definition~3.2 of~\cite{ArPh}.
\end{cor}

\begin{proof}
Since $C^* \bigl( \Z, \, C (X, D), \, \af \bigr)$ is stably finite
(by Proposition~\ref{P_5418_CPStFin}),
we can combine Theorem~\ref{T_5418_AYLg}
and Theorem~4.6 of~\cite{ArPh}.
\end{proof}

We conclude by giving some conditions on $D$ and~$\af$ which
guarantee the hypotheses of Corollary~\ref{AyCentLarge}.
These are more natural to consider than the awkward
pseudoperiodicity hypothesis.

\begin{cor}\label{AyCentLargeConds}
Let $X$ be a compact metric space,
let $h \colon X \to X$ be a minimal homeomorphism,
let $Y \subset X$ be a compact subset
such that $h^{n} (Y) \cap Y = \varnothing$
for all $n \in \Z \setminus \set{0}$,
let $D$ be a simple unital \ca{}
which has a tracial state,
and let $\af \in \Aut (C (X, D))$ lie over~$h$.
Let $x \mapsto \af_{x}$
be the corresponding map from $X$ to $\Aut (D)$,
as in Lemma~\ref{L_4Y12_GetAuto}.
Assume one of the following conditions holds.
\begin{enumerate}
\item\label{9212_AyCentLargeConds_AppInn}
All elements of $\set{ \af_{x} \colon x \in X} \subset \Aut (D)$
are approximately inner.
\item\label{9212_AyCentLargeConds_StrCmp}
$D$ has strict comparison of positive elements.
\item\label{9212_AyCentLargeConds_SP}
$D$ has property (SP) and the order on projections over~$D$
is determined by quasitraces.
\item\label{9212_AyCentLargeConds_CptGen}
The set
$\set{ \af_{x} \colon x \in X}$
is contained in a subgroup of $\Aut (D)$
which is compact in the topology
of pointwise convergence in the norm on~$D$.
\end{enumerate}
Then
$C^* \bigl( \Z, \, C (X, D), \, \af \bigr)_Y$
is a centrally large subalgebra of
$C^* \bigl( \Z, \, C (X, D), \, \af \bigr)$.
\end{cor}

\begin{proof}
We claim that, under any of the conditions
(\ref{9212_AyCentLargeConds_AppInn}),
(\ref{9212_AyCentLargeConds_StrCmp}),
(\ref{9212_AyCentLargeConds_SP}),
or (\ref{9212_AyCentLargeConds_CptGen}),
the set $\set{ \af_{x} \colon x \in X}$
is contained in a pseudoperiodic subgroup of $\Aut (D)$.
For~(\ref{9212_AyCentLargeConds_AppInn}),
this follows from Lemma~\ref{L_4Y12_ApproxInn};
for~(\ref{9212_AyCentLargeConds_StrCmp}),
from Lemma~\ref{L_4Y12_TrPrsv};
for~(\ref{9212_AyCentLargeConds_SP}),
from Lemma~\ref{L_5417_SP};
and for~(\ref{9212_AyCentLargeConds_CptGen}),
from Lemma~\ref{Lem1_190201D}.
Now apply Lemma~\ref{L_4Y16_WhenPPG}
to see that
the hypotheses of Corollary~\ref{AyCentLarge} are satisfied.
\end{proof}

The case in which $D$ is purely infinite and simple is
much easier.
We give a direct proof, not depending on large subalgebras,
that $C^*_{\mathrm{r}} \big( G, \, C (X, D), \, \af \big)$
is purely infinite simple for any discrete group~$G$.
It is still true,
and will be proved below (Proposition \ref{T_9921_PICase}),
that, under the other hypotheses of Theorem~\ref{T_5418_AYLg}
the subalgebra $C^* \bigl( \Z, \, C (X, D), \, \af \bigr)_Y$
is a large subalgebra of $C^* \bigl( \Z, \, C (X, D), \, \af \bigr)$
of crossed product type.
This fact seems potentially useful,
but does not help with the analysis of any of the examples
in Section~\ref{Sec_InfRed}.

\begin{thm}\label{T_0101_PICase}
Let $G$ be a discrete group,
let $X$ be a compact space,
and suppose $G$ acts on $X$ in such a way that
the action is minimal and
for every finite set $S \subset G \setminus \{ 1 \}$,
the set
\[
\big\{ x \in X \colon {\mbox{$g x \neq x$ for all $g \in S$}} \big\}
\]
is dense in~$X$.
Let $D$ be a purely infinite simple \uca,
and let $\af \colon G \to \Aut (C (X, D))$
be an action of $G$ on $C (X, D)$
which lies over the given action of $G$ on $X$
(in the sense of Definition~\ref{D_4Y12_OverX}).
Then $C^*_{\mathrm{r}} \big( G, \, C (X, D), \, \af \big)$
is purely infinite simple.
\end{thm}

We saw in Proposition~\ref{P_4Y15_CPSimple}
that $C^*_{\mathrm{r}} \big( G, \, C (X, D), \, \af \big)$
is simple,
but we won't use that in this proof.

\begin{proof}[Proof of Theorem~\ref{T_0101_PICase}]
For convenience of notation,
set $A = C^*_{\mathrm{r}} \big( G, \, C (X, D), \, \af \big)$.
We also freely identify $C (X, D)$ with $C (X) \otimes D$.
We prove the result by showing that $1_A \precsim a$
for all $a \in A_{+} \setminus \{ 0 \}$.
By Lemma~\ref{L_4X21_79}
(taking $B$ there to be~$A$),
we can assume that $a \in C (X, D)_{+} \setminus \{ 0 \}$.

Use Kirchberg's Slice Lemma (Lemma 4.1.9 of~\cite{Rrd})
to find $f \in C (X)_{+} \setminus \{ 0 \}$
and $b \in D_{+} \setminus \{ 0 \}$
such that
\begin{equation}\label{Eq_0201_Slice}
f \otimes b \precsim_{C (X, D)} a.
\end{equation}
\Wolog{} $\| f \| = 1$.
Since $D$ is purely infinite and simple,
we have $1_D \precsim_{D} b$, and so it follows that
\begin{equation}\label{Eq_0201_InD}
f \otimes 1_D \precsim_{D} f \otimes b.
\end{equation}
Set $U = \bigl\{ x \in X \colon f (x) > \frac{1}{2} \bigr\}$.
By minimality of the action,
the sets $g U$ for $g \in G$ cover~$X$.
So there are $n \in \N$ and $g_1, g_2, \ldots, g_n \in G$
such that the sets $g_1 U, g_2 U, \ldots, g_n U$ cover~$X$.
The function $x \mapsto \sum_{k = 1}^n f (g_k^{-1} x)$
is strictly positive on~$X$.
Using this fact at the first step, pure infiniteness of~$D$
at the second last step,
and (\ref{Eq_0201_InD}) and~(\ref{Eq_0201_Slice})
at the last step,
we get
\begin{align*}
1_A
& \precsim_{C (X)} \sum_{k = 1}^n \af_{g_k} (f \otimes 1_D)
\\
& \precsim_{C (X)} \diag \bigl( \af_{g_1} (f \otimes 1_D), \,
       \af_{g_2} (f \otimes 1_D), \, \ldots, \,
        \af_{g_n} (f \otimes 1_D) \bigr)
\\
& = \diag \bigl( u_{g_1} (f \otimes 1_D) u_{g_1}^*, \,
       u_{g_2} (f \otimes 1_D) u_{g_2}^*, \, \ldots, \,
        u_{g_n} (f \otimes 1_D) u_{g_n}^* \bigr)
\\
& \sim_{A} \diag \bigl( f \otimes 1_D, \,
       f \otimes 1_D, \, \ldots, \,
        f \otimes 1_D \bigr)
  = f \otimes 1_{M_n (D)}
  \precsim_{C (X, D)} f \otimes 1_D
  \precsim_{A} a.
\end{align*}
This completes the proof.
\end{proof}

As promised,
we now give a result on large subalgebras when $G = \Z$.
We use two lemmas,
the first of which has a similar proof to that of
Theorem~\ref{T_0101_PICase}.

\begin{lem}\label{L_0101_SubalgPI}
Let $X$ be a \cms, let $h \colon X \to X$ be a minimal homeomorphism,
let $D$ be a purely infinite simple \uca,
and let $\af \in \Aut (C (X, D))$ lie over~$h$.
Let $Y \subset X$ be a compact subset
such that $h^n (Y) \cap Y = \varnothing$
for all $n \in \Z \setminus \{ 0 \}$.
Then $C^* \bigl( \Z, \, C (X, D), \, \af \bigr)_Y$
is purely infinite and simple.
\end{lem}

\begin{proof}
For convenience of notation,
set $B = C^* \big( \Z, \, C (X, D), \, \af \big)_Y$.
We also freely identify $C (X, D)$ with $C (X) \otimes D$.
We prove the result by showing that $1_A \precsim a$
for all $a \in A_{+} \setminus \{ 0 \}$.
By Lemma~\ref{L_4X21_79}
and Lemma~\ref{L_4X21_AYStruct},
we can assume that $a \in C (X, D)_{+} \setminus \{ 0 \}$.
Using Kirchberg's Slice Lemma
and pure infiniteness of~$D$
as in the proof of Theorem~\ref{T_0101_PICase},
there is $f \in C (X)_{+} \setminus \{ 0 \}$
such that $f \otimes 1_D \precsim_{C (X, D)} a$.

We now claim that for every $x \in X$ there is $l_x \in C (X)_{+}$
such that $l_x \otimes 1_D \precsim_{B} f \otimes 1_D$
and $l_x (x) > \frac{1}{2}$.
Given this,
the proof is finished as in the last computation
in the proof of
Theorem~\ref{T_0101_PICase}.

To prove the claim,
set $U = \bigl\{ x \in X \colon f (x) > \frac{1}{2} \bigr\}$.
If $x \in U$,
take $l_x = f$.

Suppose next that $x \not\in U$
and $x \not\in \bigcup_{m = 0}^{\I} h^m (Y)$.
By minimality of~$h$,
there is $n \in \N$ such that $x \in h^{n} (U)$.
Choose $r \in C (X)_{+}$ such that $r (y) = 0$
for all $y \in \bigcup_{m = 0}^{n - 1} h^m (Y)$
and $r (x) = 1$.
Then, letting $u \in C^* \big( \Z, \, C (X, D), \, \af \big)$
be the standard unitary (as in Notation~\ref{ntn1-181221D}),
we have $(r \otimes 1_D) u^n \in B$
by Lemma~\ref{L_4X21_AYStruct}.
Taking $l_x = r^2 \cdot (f \circ h^{- n})$,
we have $l_x (x) = f ( h^{- n} (x)) > \frac{1}{2}$.
Also, $(r u^n) (f \otimes 1_D) (r u^n)^* = l_x \otimes 1_D$,
so $l_x \otimes 1_D \precsim_{B} f \otimes 1_D$.

Finally, suppose $x \not\in U$
and $x \in \bigcup_{m = 0}^{\I} h^m (Y)$.
Then $x \not\in \bigcup_{m = 1}^{\I} h^{- m} (Y)$.
By minimality of~$h$,
there is $n \in \N$ such that $x \in h^{- n} (U)$.
Choose $r \in C (X)_{+}$ such that $r (y) = 0$
for all $y \in \bigcup_{m = 1}^{n} h^{- m} (Y)$
and $r (x) = 1$.
Then $(r \otimes 1_D) u^{- n} \in B$
by Lemma~\ref{L_4X21_AYStruct}.
Taking $l_x = r^2 \cdot (f \circ h^{n})$,
we have $l_x (x) = f ( h^{n} (x)) > \frac{1}{2}$.
Also, $(r u^{- n}) (f \otimes 1_D) (r u^{- n})^* = l_x \otimes 1_D$,
so $l_x \otimes 1_D \precsim_{B} f \otimes 1_D$.
This completes the proof.
\end{proof}

\begin{lem}\label{L_9Z01_KeepLarge}
Let $X$ be a \chs,
let $G$ be a discrete group,
and let $D$ be a \ca.
Let $(g, x) \mapsto g x$ be an action of $G$ on~$X$,
and let $\af \colon G \to \Aut (C (X, D))$
be an action of $G$ on $C (X, D)$
which lies over $(g, x) \mapsto g x$
in the sense of Definition~\ref{D_4Y12_OverX}.
Then for every
$b \in C^*_{\mathrm{r}} \bigl( G, \, C (X, D), \, \af \bigr)$
and every $\ep > 0$,
there is a finite set $T \S G$ and a nonempty open set $V \S X$
such that,
whenever $x \in V$
and $f \in C (X)$ satisfy $0 \leq f \leq 1$
and $f (g x) = 1$ for all $g \in T$,
then $\| f b \| > \| b \| - \ep$.
\end{lem}

If $D$ is not unital, then the product $f b$ is realized via
the inclusion of $C (X)$ in
$M \bigl( C^* \bigl( G, \, C (X, D), \, \af \bigr) \bigr)$.

\begin{proof}[Proof of Lemma~\ref{L_9Z01_KeepLarge}]
Let $(g, x) \mapsto \af_{g, x}$ be as in Definition~\ref{D_4Y12_OverX}.

Fix a faithful representation $\rh$ of $D$ on a Hilbert space~$H$.
Then for every $x \in X$ there is a representation
$\rh \circ \ev_x \colon C (X, D) \to L (H)$.
Let $(v_x, \pi_x)$ be the corresponding regular covariant
representation of $\bigl( G, \, C (X, D), \, \af \bigr)$
on $l^2 (G, H)$,
and let
$\sm_x \colon C^* \bigl( G, \, C (X, D), \, \af \bigr) \to l^2 (G, H)$
be its integrated form.
We identify $l^2 (G, H)$ with $l^2 (G) \otimes H$,
and we write $\dt_g \in l^2 (G)$ for the standard basis vector
corresponding to $g \in G$.
For later use, we recall that if $S_1, S_2 \S G$ are finite sets,
$d \in C (X, D) [G]$
is given as $d = \sum_{g \in S_1} d_g u_g$
with $d_g \in C (X, D)$ for $g \in S_1$,
and $\xi \in l^2 (G) \otimes H$ has the form
$\xi = \sum_{h \in S_2} \dt_h \otimes \xi_h$
with $\xi_h \in H$ for $h \in S_2$,
then
\begin{equation}\label{Eq_9X01_IndRepFormula}
\sm_x (d) \xi
 = \sum_{g \in S_1} \sum_{h \in S_2}
   \dt_{g h} \otimes
       \rh \bigl( \af_{h^{-1} g^{-1}, x} (d_g (g h x)) \bigr) \xi_h.
\end{equation}

The representation $\bigoplus_{x \in X} \rh \circ \ev_x$
is a faithful representation of $C (X, D)$,
so $\bigoplus_{x \in X} \sm_x$
is a faithful representation of
$C^*_{\mathrm{r}} \bigl( G, \, C (X, D), \, \af \bigr)$.
Therefore there exists $x_0 \in X$ such that
$\| \sm_{x_0} (b) \| > \| b \| - \frac{\ep}{5}$.
Choose $c \in C (X, D) [G]$ such that
$\| b - c \| < \frac{\ep}{5}$.
In particular,
$\| \sm_{x_0} (c) \| > \| b \| - \frac{2 \ep}{5}$.
Choose $\xi \in l^2 (G, H)$ with finite support such that
$\| \xi \| = 1$ and
$\| \sm_{x_0} (c) \xi \| > \| b \| - \frac{3 \ep}{5}$.
Write $c = \sum_{g \in S_1} c_g u_g$
and $\xi = \sum_{h \in S_2} \dt_h \otimes \xi_h$
with $S_1, S_2 \S G$ finite,
$c_g \in C (X, D)$ for $g \in S_1$,
and $\xi_h \in H$ for $h \in S_2$.
Use Lemma~\ref{L_9Y30_MixCont} to choose
an open set $V \S X$ such that $x_0 \in V$
and such that for all $x \in V$, $g \in S_1$, and $h \in S_2$, we have
\begin{equation}\label{Eq_9Z01_Choice}
\bigl\| \af_{h^{-1} g^{-1}, x} (c_g (g h x))
    - \af_{h^{-1} g^{-1}, x_0} (c_g (g h x_0) ) \bigr\|
  < \frac{\ep}{5 \card (S_1) \card (S_2) + 1}.
\end{equation}
Set
\[
T = \bigl\{ g h \colon {\mbox{$g \in S_1$ and $h \in S_2$}} \bigr\}.
\]
Now let $x \in V$,
and suppose $f \in C (X)$ satisfies $0 \leq f \leq 1$
and $f (g x) = 1$ for all $g \in T$.
The condition on~$f$
and the formula~(\ref{Eq_9X01_IndRepFormula})
imply that $\sm_x (f c) \xi = \sm_x (c) \xi$.
Applying~(\ref{Eq_9X01_IndRepFormula}) again,
and using~(\ref{Eq_9Z01_Choice}) at the second step,
we get
\begin{equation}\label{Eq_9Z01_xx0}
\| \sm_x (c) \xi - \sm_{x_0} (c) \xi \|
 \leq \sum_{g \in S_1} \sum_{h \in S_2}
   \bigl\| \af_{h^{-1} g^{-1}, x} (c_g (g h x))
    - \af_{h^{-1} g^{-1}, x_0} (c_g (g h x_0) ) \bigr\|
 < \frac{\ep}{5}.
\end{equation}
Therefore, using $\| f \| \leq 1$ at the second step
and $\| b - c \| < \frac{\ep}{5}$ and the second and fifth steps,
as well as (\ref{Eq_9Z01_xx0}) at the fourth step,
\begin{align*}
\| f b \|
& \geq \| \sm_x (f b) \|
  > \| \sm_x (f c) \xi \| - \frac{\ep}{5}
  = \| \sm_x (c) \xi \| - \frac{\ep}{5}
  > \| \sm_{x_0} (c) \xi \| - \frac{2 \ep}{5}
\\
& > \| \sm_{x_0} (b) \xi \| - \frac{3 \ep}{5}
  > \| \sm_{x_0} (b) \| - \frac{4 \ep}{5}
  > \| \sm_{x_0} (b) \| - \frac{4 \ep}{5}
  > \| b \| - \ep,
\end{align*}
as desired.
\end{proof}

\begin{prp}\label{T_9921_PICase}
Let $X$ be a compact metric space,
let $h \colon X \to X$ be a minimal homeomorphism,
let $Y \subset X$ be a compact subset
such that $h^{n} (Y) \cap Y = \varnothing$
for all $n \in \Z \setminus \set{0}$,
let $D$ be a purely infinite simple unital \ca,
and let $\af \in \Aut (C (X, D))$ lie over~$h$.
Then
$C^* \bigl( \Z, \, C (X, D), \, \af \bigr)_Y$
is a large subalgebra of
$C^* \bigl( \Z, \, C (X, D), \, \af \bigr)$
of crossed product type.
\end{prp}

\begin{proof}
We verify directly the conditions of Definition~4.9 of~\cite{PhLg}.
We follow Notation~\ref{N_4Y12_CPN}.
Set
\[
A = C^* \bigl( \Z, \, C (X, D), \, \af \bigr),
\qquad
B = C^* \bigl( \Z, \, C (X, D), \, \af \bigr)_Y,
\]
\[
C = C (X, D),
\andeqn
G = \{ u \}.
\]
All parts of condition~(1)
in Definition~4.9 of~\cite{PhLg}
are obvious.

We verify condition~(2) there.
Let $m \in \N$,
let $a_1, a_2, \ldots, a_m \in A$,
let $\ep > 0$,
let $b_1 \in A_{+}$ satisfy $\| b_1 \| = 1$,
and let $b_2 \in B_{+} \SM \{ 0 ]$.
We will part of the verification of
condition~(3)
in Proposition 4.11 of~\cite{PhLg}
in the proof of Theorem~\ref{T_5418_AYLg}.
Choose $c_1, c_2, \ldots, c_m \in C (X, D) [\Z]$
such that $\| c_j - a_j \| < \ep$
for $j = 1, 2, \ldots, m$.
(This estimate is condition~(2b).)
Choose $N \in \N$
such that for $j = 1, 2, \ldots, m$
there are $c_{j, l} \in C (X, D)$
for $l = -N, \, - N + 1, \, \ldots, \, N - 1, \, N$
with
\[
c_j = \sum_{l = -N}^N c_{j, l} u^l.
\]
Apply Lemma~\ref{L_9Z01_KeepLarge},
getting a finite set $T \S \Z$ and a nonempty open set $U \S X$
such that,
whenever $x \in U$
and $f \in C (X)$ satisfy $0 \leq f \leq 1$
and $f (h^k (x)) = 1$ for all $k \in T$,
then $\| f b_1^{1 / 2} \| > 1 - \frac{\ep}{2}$.
Define
\[
S = \bigl\{ n - k \colon
      {\mbox{$k \in T$ and $n \in [- N, N] \cap \Z$}} \bigr\},
\]
which is a finite subset of~$\Z$.
Then $\bigcup_{n \in S} h^n (Y)$ is nowhere dense in~$X$,
so there exists $x_0 \in W$
such that $x_0 \not\in \bigcup_{n \in S} h^n (Y)$.
This choice implies that
\[
\{ h^k (x_0) \colon k \in T \} \cap \bigcup_{n = - N}^{N} h^n (Y)
  = \E,
\]
so there is $g \in C (X)$
such that $0 \leq g \leq 1$,
$g (y) = 1$ for all $y \in \bigcup_{n = - N}^{N} h^n (Y)$,
and $g (h^k (x)) = 0$ for all $k \in T$.

Condition~(1a) in Definition~4.9 of~\cite{PhLg}
holds by construction.
The proof that $(1 - g) c_j \in B$ for $j = 1, 2, \ldots, m$
(condition~(2c) in Definition~4.9 of~\cite{PhLg})
is the same as at the end of the proof of Theorem~\ref{T_5418_AYLg}:
follow
the corresponding part of the proof of Theorem~7.10 of~\cite{PhLg},
except using \Lem{L_4X21_AYStruct}
in place of Proposition~7.5 of~\cite{PhLg}.
Lemma~\ref{L_0101_SubalgPI}
implies $1_A \precsim_A b_1$ and $1_A \precsim_B b_2$,
so the requirements $g \precsim_A b_1$ and $g \precsim_B b_2$
(condition~(2d) in Definition~4.9 of~\cite{PhLg})
follow immediately.
For condition~(2e) in Definition~4.9 of~\cite{PhLg},
we estimate, using $1 - g (h^k (x_0)) = 1$
for all $k \in T$ at the second step:
\[
\bigl\| (1 - g) b_1 (1 - g) \bigr\|
 = \| (1 - g) b_1^{1 / 2} \|^2
 > \bigl( 1 - \frac{\ep}{2} \bigr)^2
 > \ep.
\]
This completes the proof.
\end{proof}

\section{Recursive structure for orbit breaking
 subalgebras}\label{RecStruct}

In this section, under appropriate conditions
we will show that the orbit breaking subalgebras
of Definition~\ref{D_4Y12_OrbSubalg}
have a tractable recursive structure,
as subalgebras of certain homogeneous algebras
tensored with $D$. We start by introducing the
formalism for a generalization of the recursive
subhomogeneous algebras introduced in
\cite{PhRsha1} that were crucial for the analysis in
\cite{QLinPhDiff} and \cite{HLinPh}.

\begin{dfn}\label{pullback}
Let $A$, $B$, and $C$ be \ca{s}, and let $\ph \colon
A \to C$ and $\psi \colon B \to C$ be homomorphisms.
Then the associated {\emph{pullback \ca{}}} $A
\oplus_{C, \ph, \psi} B$ is defined by
\[
A \oplus_{C, \ph, \psi} B
 = \bset{ (a, b) \in A \oplus B \colon \ph (a) = \psi (b) }.
\]
\end{dfn}

We frequently write $A \oplus_{C} B$ when the maps $\ph$
and $\psi$ are understood.
If $C = 0$ we just get $A \oplus B$.

\begin{dfn}\label{RSHA}
Let $D$ be a simple unital \ca.
The class of
{\emph{recursive subhomogeneous algebras over $D$}} is
the smallest class $\mathcal{R}$ of \ca{s} that is
closed under isomorphism and such that:
\begin{enumerate}
\item\label{9223_RSHA_1}
If $X$ is a compact Hausdorff space and $n \geq 1$,
then $C (X, M_{n} (D)) \in \mathcal{R}$.
\item\label{9223_RSHA_2}
If $B \in \mathcal{R}$, $X$ is compact Hausdorff,
$n \geq 1$, $X^{(0)} \subset X$ is closed (possibly empty),
$\ph \colon B \to C (X^{(0)}, \, M_{n} (D))$
is any unital homomorphism
(the zero \hm{} if $X^{(0)}$ is empty), and
$\rho \colon C (X, M_{n} (D)) \to C (X^{(0)}, \, M_{n} (D))$ is the
restriction homomorphism, then the pullback
\begin{align*}
& B \oplus_{C (X^{(0)}, \, M_{n} (D)), \, \ph, \, \rh)} C (X, M_{n} (D))
\\
& \hspace*{3em} {\mbox{}}
  = \bset{ (b, f) \in B \oplus C (X, M_{n} (D)) \colon
        \ph (b) = f |_{(X^{(0)}} }
\end{align*}
is in $\mathcal{R}$.
\end{enumerate}
\end{dfn}

Taking $D = \C$ in this definition gives the usual definition
for the class of recursive subhomogeneous algebras.
(See \cite{PhRsha1}.)
This definition makes sense for any unital \ca~$D$,
but it is not clear whether it is appropriate in this generality.

\begin{dfn}\label{RSHAMachinery}
We adopt the following standard notation for recursive
subhomogeneous algebras over $D$.
The definition implies
that if $R$ is any recursive subhomogeneous algebra over
$D$, then $R$ may be written in the form
\[
R \cong
  \left[ \cdots \left[ \left[ C_{0} \oplus_{C_{1}^{(0)}, \ph_1, \rh_1}
              C_{1} \right]
    \oplus_{C_{2}^{(0)}, \ph_2, \rh_2} \right] \cdots \right]
   \oplus_{C_{l}^{(0)}, \ph_l, \rh_l} C_{l},
\]
with $C_{k} = C (X_{k}, \, M_{n (k)} (D))$ for compact Hausdorff
spaces $X_{k}$ and positive integers $n (k)$, and with
$C_{k}^{(0)} = C \bigl( X_{k}^{(0)}, \, M_{n (k)} (D) \bigr)$
for compact
subsets $X_{k}^{(0)} \subset X_{k}$ (possibly empty),
where the maps $\rho_{k} \colon C_{k} \to C_{k}^{(0)}$ are
always the restriction maps.
An expression of this type for
$R$ will be referred to as a {\emph{decomposition}} of~$R$,
and the notation that appears here will be referred to as the
{\emph{standard notation for a decomposition}}.
We associate the following objects to this decomposition.
\begin{enumerate}
\item\label{9223_RSHAMachinery_1}
Its {\emph{length}} $l$.
\item\label{9223_RSHAMachinery_2}
The {\emph{$k$-th stage algebra}}
\[
R^{(k)}
 = \left[ \cdots \left[ \left[ C_{0} \oplus_{C_{1}^{(0)}}
      C_{1} \right] \oplus_{C_{2}^{(0)}} C_{2} \right] \cdots \right]
      \oplus_{C_{k}^{(0)}} C_{k}.
\]
\item\label{9223_RSHAMachinery_3}
Its {\emph{base spaces}} $X_{0},\ldots,X_{l}$ and {\emph{total space}}
$X = \coprod_{k = 0}^{l} X_{k}$.
\item\label{9223_RSHAMachinery_4}
Its {\emph{matrix sizes}} $n (0),\ldots,n (l)$ and {\emph{matrix
size function}} $m \colon X \to \Z_{\geq 0}$ defined by $m(x) = n (k)$
when $x \in X_{k}$.
(This is called the {\emph{matrix size of $R$ at~$x$}}.)
\item\label{9223_RSHAMachinery_5}
Its {\emph{minimum matrix size}} $\min_{k} n (k)$ and
{\emph{maximum matrix size}} $\max_{k} n (k)$.
\item\label{9223_RSHAMachinery_6}
Its {\emph{topological dimension}}
$\dim (X) = \max_{k} \dim (X_{k})$
and {\emph{topological dimension function}}
$d \colon X \to \Z_{\geq 0}$, defined by $d(x) = \dim (X_{k})$
for $x \in X_{k}$.
(This is called the {\emph{topological dimension of $R$ at $x$}}.)
\item\label{9223_RSHAMachinery_7}
Its {\emph{standard representation}}
$\sm
 = \sm_{R} \colon R
   \to \bigoplus_{k = 0}^{l} C (X_{k}, \, M_{n (k)} (D))$,
defined by
forgetting the restriction to a subalgebra in each of the fibered
products in the decomposition.
\item\label{9223_RSHAMachinery_8}
The associated {\emph{evaluation maps}}
$\ev_{x} \colon R \to M_{n (k)} (D)$,
defined to be the restriction of the usual evaluation
map on $\bigoplus_{k = 0}^{l} C (X_{k}, \, M_{n (k)} (D))$ to $R$,
when $R$
is identified with a subalgebra of this algebra through the standard
representation $\sm_{R}$.
\end{enumerate}
\end{dfn}

A decomposition of an algebra $R$ as a
recursive subhomogeneous algebra over $D$ may be highly
nonunique, which in the case $D = \C$ is clear from examples
in~\cite{PhRsha1}.
Moreover, the matrix sizes are not uniquely
determined even if all the other data has already been chosen,
which is easily seen through
the example of the $2^{\infty}$~UHF algebra.

\begin{ntn}\label{StronglySelfAbs}
Throughout, we let $\D$ denote a separable, unital \ca{}
that is strongly selfabsorbing in the sense of \cite{TomsWinter2}:
there
exists an isomorphism $\D \cong \D \ten \D$ that is unitarily
equivalent to $a \mapsto a \ten 1_{\D}$.
\end{ntn}

Examples include the Cuntz algebras
$\mathcal{O}_{2}$ and $\mathcal{O}_{\infty}$,
and the Jiang-Su algebra $\JS$,
which is a simple, separable, unital, infinite dimensional,
strongly selfabsorbing, nuclear \ca{} having the same
Elliott invariant
as the complex numbers $\C$.
(See \cite{JiangSu} for its construction).

\begin{dfn}\label{DStableDefn}
Adopt Notation~\ref{StronglySelfAbs}.
A separable \ca{} $D$ is called {\emph{$\D$-stable}}
if there is
an isomorphism $D \ten \D \cong D$.
\end{dfn}

It is clear that if $D$ is $\D$-stable, then so are $M_{n} (D)$
and $C (X, D)$.
With appropriate assumptions on the algebra $D$, we can give some
results about the $\D$-stability of recursive subhomogeneous algebras
over $C$.

\begin{lem}\label{DStablePullback}
Adopt Notation~\ref{StronglySelfAbs}.
Let $A$, $B$, and $C$ be separable unital \ca{s}
(allowing $C = 0$),
and let $\ph \colon A \to C$ and $\psi \colon B \to C$
be unital homomorphisms.
Let $P = A \oplus_{C, \ph, \ps} B$ be the associated pullback
(Definition~\ref{pullback}).
If $\ps$
is surjective and both $A$ and $B$ are $\D$-stable, then $P$ is
$\D$-stable.
\end{lem}

\begin{proof}
Define
$\pi \colon P \to A$ by $\pi (a, b) = a$.
There is an isomorphism
$\io \colon \Kern (\ps) \to \Kern (\pi)$
given by $\io (b) = (0, b)$ for $b \in \Kern (\ps)$.
Moreover, surjectivity of $\ps$ is easily seen to imply
surjectivity of $\pi$.
We obtain an exact
sequence
\[
0 \longrightarrow \Kern (\ps)
 \stackrel{\iota}{\longrightarrow} P
 \stackrel{\pi}{\longrightarrow} A
 \longrightarrow 0.
\]
Now $\Kern (\ph)$ is $\D$-stable by Corollary 3.1 of
\cite{TomsWinter2}.
Since $B$ is also $\D$-stable, Theorem 4.3 of
\cite{TomsWinter2} implies that $P$ is $\D$-stable.
\end{proof}

\begin{prp}\label{DStableRSHA}
Adopt Notation~\ref{StronglySelfAbs}.
Let $D$ be a simple separable $\D$-stable \ca,
and let $R$ be a recursive subhomogeneous algebra over $D$.
Then $R$ is $\D$-stable.
\end{prp}

\begin{proof}
We proceed by induction on the length of a decomposition for $R$ as
a recursive subhomogeneous algebra over $D$.
The base case is when
$R = C (X, M_{n} (D))$, and this is $\D$-stable whenever $D$ is
$\D$-stable.
For the inductive step, we may assume that
there are an $\D$-stable unital \ca~$R'$,
$n \in \N$, a compact Hausdorff space~$X$,
a closed subset $X^{(0)} \subset X$,
and a unital homomorphism
$\ph \colon R' \to C \bigl( X^{(0)}, \, M_{n} (D) \bigr)$
such that
\[
R = \bset{ (a, f) \in R' \oplus C (X, M_{n} (D))
   \colon \ph (a) = f |_{X^{(0)}}  }
\]
Since $f \mapsto f |_{X^{(0)}}$ is surjective and both
$R'$ and $C (X, M_{n} (D))$
are $\D$-stable, it follows from Lemma~\ref{DStablePullback}
that $R$ is $\D$-stable,
which completes the induction.
\end{proof}

\begin{prp}\label{AyDirLim}
Let $X$ be a \chs, let $h \colon X \to X$ be a homeomorphism,
let $D$ be a \uca,
and let $\af \in \Aut (C (X, D))$ lie over~$h$.
Let $Y \subset X$ be closed, and let
$Y_{1} \supset Y_{2} \supset \cdots$
be closed subsets of $X$ such that $\bigcap_{n = 1}^{\infty}
Y_{n} = Y$.
Then
\[
C^* \bigl( \Z, \, C (X, D), \, \af \bigr)_Y
 = {\overline{ {\ts{ {\ds{\bigcup}}_{n = 1}^{\infty} }} A_{Y_{n}} }}
 \cong \dirlim A_{Y_{n}}.
\]
\end{prp}

\begin{proof}
The proof is easy, and is omitted.
\end{proof}

The results in the remainder of this section are
mostly generalizations of ones in Section 1 of
\cite{QLinPhDiff}.
We follow a slightly more modern development,
adapted from Section 11.3 of \cite{PhNotes}, since \cite{QLinPhDiff}
has not been published.
(The proofs in \cite{PhNotes} are mostly
taken from \cite{QLinPhDiff}, with some changes in notational
conventions.)
Most proofs go through with changes only to the notation,
such as replacing the action of a minimal homeomorphism
$h$ with the action $\af$ on $C (X, D)$.
The biggest technical
differences are in the proof of Lemma~\ref{L_4X21_AYStruct}.
We begin with the well-known Rokhlin tower construction.

\begin{ntn}\label{ModifiedRokhlinTower}
Let $X$ be a \chs{} and let $h \colon X \to X$ be a \mh.
Let $Y \subset X$ be a closed set with $\sint (Y) \neq \varnothing$.
For $y \in Y$, define
$r (y) = \inf \bigl( \set{m \geq 1 \colon h^{m} (y) \in Y} \bigr)$.
Then $\sup_{y \in Y} r (y) < \infty$
(see Lemma \ref{RokhlinTowerProps}(\ref{9Z26_new})),
so there are finitely many
distinct values $n (0) < n (1) < \cdots < n (l)$ in the range of~$r$.
For $k = 0, 1, \ldots l$, set
\[
Y_{k} = \overline{ \set{ y \in Y \colon r (y) = n (k) } }
\andeqn
Y_{k}^{\circ}
 = \sint \bigl( \bset{ y \in Y \colon r (y) = n (k) } \bigr).
\]
\end{ntn}

We warn that, while
$Y_{k}^{\circ} \subset \sint (Y_k)$,
the inclusion may be proper.

\begin{lem}\label{RokhlinTowerProps}
Let $Y \subset X$ be a closed set with $\sint (Y) \neq \varnothing$,
and adopt Notation \ref{ModifiedRokhlinTower}.
Then:
\begin{enumerate}
\item\label{9Z26_new}
$\sup_{y \in Y} r (y) < \infty$.
\item\label{Item_RokhlinTowerProps_1}
The sets $h^{j} (Y_{k}^{\circ})$ are pairwise disjoint for $0
\leq k \leq l$ and $1 \leq j \leq n (k)$.
\item\label{Item_RokhlinTowerProps_2}
$\bigcup_{k = 0}^{l} Y_{k} = Y$.
\item\label{Item_RokhlinTowerProps_3}
$\bigcup_{k = 0}^{l} \bigcup_{j= 0}^{n (k)-1} h^{j} (Y_{k}) = X$.
\end{enumerate}
\end{lem}

\begin{proof}
This is Lemma 11.3.5 of~\cite{PhNotes}.
\end{proof}

We first consider the specific situation of the structure of
$C^* \bigl( \Z, \, C (X, D), \, \af \bigr)_Y$
when $X$ is the Cantor set.
This is needed later.

\begin{lem}\label{CantorSetCase}
Adopt Notation \ref{ModifiedRokhlinTower}, let $X$ be the Cantor
set, and let $Y \subset X$ be a nonempty compact open subset.
Then
\[
C^* \bigl( \Z, \, C (X, D), \, \af \bigr)_Y
 \cong \bigoplus_{k = 0}^{l} C (Y_{k}, \, M_{n (k)} (D)).
\]
\end{lem}

\begin{proof}
This is a straightforward adaption of the proof of Lemma 11.2.20
of~\cite{PhNotes}.
\end{proof}

If we set $B_{k} = C (Y_{k}, M_{n (k)})$, then $B_{k}$ is an
AF algebra for each~$k$
(since $Y_{k}$ is totally disconnected), and hence
$C^* \bigl( \Z, \, C (X, D), \, \af \bigr)_Y$
is a direct sum of what might be called ``AF $D$-algebras''.

\begin{prp}\label{BanachSpaceDirSum}
Let $X$ be a \chs, let $h \colon X \to X$ be a \mh,
let $D$ be a \uca,
let $\af \in \Aut (C (X, D))$ lie over~$h$,
let $Y \subset X$ be a closed set with $\sint (Y) \neq \varnothing$,
and adopt Notation \ref{ModifiedRokhlinTower}.
Following the notation of Lemma~\ref{L_4X21_AYStruct}, there is
$N \in \N$ such that $C^* \bigl( \Z, \, C (X, D), \, \af \bigr)_Y$
has the Banach space direct sum decomposition
\[
C^* \bigl( \Z, \, C (X, D), \, \af \bigr)_Y
 = \bigoplus_{n = - N}^{N} C_{0} (X \setminus Y_{n}, \, D) u^{n}.
\]
\end{prp}

\begin{proof}
The proof is the same as that of Corollary 11.3.7 of~\cite{PhNotes}.
\end{proof}

\begin{prp}\label{AyHom}
Let $X$ be a \chs, let $h \colon X \to X$ be a \mh,
let $D$ be a \uca,
let $\af \in \Aut (C (X, D))$ lie over~$h$,
let $Y \subset X$ be a closed set with $\sint (Y) \neq \varnothing$,
and adopt Notation \ref{ModifiedRokhlinTower}.
Define the unitary $s_{k} \in C (Y_{k}, M_{n (k)} (D))$ by
\[
s_{k} = \left[ \begin{array}{ccccccc}
0 & 0 & \cdots & \cdots & 0 & 0 & 1 \\
1 & 0 & \cdots & \cdots & 0 & 0 & 0 \\
0 & 1 & \cdots & \cdots & 0 & 0 & 0 \\
\vdots & \vdots & \ddots &  & \vdots & \vdots & \vdots \\
\vdots & \vdots & & \ddots & \vdots & \vdots & \vdots \\
0 & 0 & \cdots & \cdots & 1 & 0 & 0 \\
0 & 0 & \cdots & \cdots & 0 & 1 & 0 \end{array} \right].
\]
For $k = 0, 1, \ldots, l$ there is a unique linear map
\[
\gm_{k} \colon C^* \bigl( \Z, \, C (X, D), \, \af \bigr)_Y
 \to C (Y_{k}, M_{n (k)} (D))
\]
such that for $m = 0, 1, \ldots, n (k) - 1$
and $f \in C_{0} (X \setminus Z_{m}, D)$ we have:
\begin{enumerate}
\item\label{9222_AyHom_1}
$\gm_{k} (f u^{m}) = \diag \bigl( f |_{Y_{k}}, \,
  \af^{-1} (f) |_{Y_{k}}, \,
  \ldots, \, \af^{-[n (k)-1]} (f) |_{Y_{k}} \bigr) \cdot s_{k}^{m}$.
\item\label{9222_AyHom_2}
$\gm_{k} (u^{-m} f) = s_{k}^{-m} \cdot
  \diag \bigl( f |_{Y_{k}}, \, \af^{-1} (f)
  |_{Y_{k}}, \, \ldots, \,
   \af^{-[n (k)-1]} (f) |_{Y_{k}} \bigr)$.
\end{enumerate}
Moreover, the map
\[
\gm \colon C^* \bigl( \Z, \, C (X, D), \, \af \bigr)_Y
      \to \bigoplus_{k = 0}^{l} C (Y_{k}, M_{n (k)} (D))
\]
given by $\gm (a) = (\gm_{0} (a), \gm_{1} (a), \ldots, \gm_{l} (a))$
is a
$*$-homomorphism.
\end{prp}

\begin{proof}
The computations in this proof are analogous to those in
the proof of Proposition 11.3.9 of~\cite{PhNotes}, using
Lemma~\ref{L_4X21_AYStruct}
in place of Proposition 11.3.6 of~\cite{PhNotes}.
\end{proof}

\begin{lem}\label{SubdiagProjnDecomp}
Let $X$ be a \chs, let $h \colon X \to X$ be a \mh,
let $D$ be a \uca,
let $\af \in \Aut (C (X, D))$ lie over~$h$,
and let $Y \subset X$ be a closed set with $\sint (Y) \neq \varnothing$.
Adopt Notation \ref{ModifiedRokhlinTower}.
Identify $C (Y_{k}, \, M_{n (k)} (D))$
with $M_{n (k)} (C (Y_{k}, D))$ in the
obvious way.
Define maps
\[
E_{k}^{(m)} \colon C (Y_{k}, M_{n (k)} (D))
\to C (Y_{k}, M_{n (k)} (D))
\]
by $E_{k}^{(m)} (b)_{m+j,j} = b_{m+j,j}$
for $1 \leq j \leq n (k) - m$ (if $m \geq 0$)
and for $-m + 1 \leq j \leq n (k)$ (if $m \leq 0$),
and $E_{k}^{(m)} (b)_{i,j} = 0$ for
all other pairs $(i,j)$.
(Thus, $E_{k}^{(m)}$ is the projection map
on the $m$-th subdiagonal.)
Write
\[
G_{m} = \bigoplus_{k = 0}^{l} E_{k}^{(m)} (C (Y_{k}, M_{n (k)} (D)).
\]
Then:
\begin{enumerate}
\item\label{SubdiagProjnDecomp_DirSum}
There is a Banach space direct sum decomposition
\[
\bigoplus_{k = 0}^{l} C (Y_{k}, M_{n (k)} (D)) =
\bigoplus_{m=-n (l)}^{n (l)} G_{m}.
\]
\item\label{SubdiagProjnDecomp_gmIs}
For $k = 0, 1,\ldots,l$, $m \geq 0$,
$f \in C_{0} (X \setminus Z_{m}, D)$,
and $x \in Y_{k}$,
the expression $\gm_{k} (f u^{m}) (x)$
is given by the following matrix, in which the first nonzero
entry is in row $m+1$:
\[
\gm_{k} (f u^{m}) (x)
 = \left[ \begin{array}{ccccccc} 0 & 0 & \cdots &
\cdots  & \cdots & \cdots & 0
\\
\vdots & \vdots & & & & & \vdots \\
0 & 0 & & & & & \vdots
\\
\af^{-m} (f) (x) & 0 & & & & & \vdots \\
0 & \af^{-[m+1]} (f) (x) & & & & & \vdots
\\
\vdots & & \ddots & & &
& \vdots
\\
0 & \cdots & \cdots & \af^{-[n (k)-1]} (f) (x) & 0
& \cdots & 0
\end{array} \right].
\]
\item\label{SubdiagProjnDecomp_3}
For $m \geq 0$ and $f \in C_{0} (X \setminus Z_{m}, D)$, we
have
\[
\gm_{k} (f u^{m}) \in E_{k}^{(m)} (C (Y_{k}, M_{n (k)} (D))),
\qquad
\gm (f u^{m}) \in G_{m},
\]
\[
\gm_{k} (u^{-m} f) \in E_{k}^{(-m)} (C (Y_{k}, M_{n (k)} (D))),
\andeqn
\gm (u^{-m} f) \in G_{-m}.
\]
\item\label{SubdiagProjnDecomp_Compatible}
The homomorphism $\gm$ is compatible with the
direct sum decomposition of Proposition~\ref{BanachSpaceDirSum}
on its domain and the direct sum decomposition of part (1) on its
codomain.
\end{enumerate}
\end{lem}

\begin{proof}
The direct sum decomposition is essentially immediate from the
definition of the maps $E_{k}^{(m)}$, while the other statements
follow from Proposition~\ref{BanachSpaceDirSum} and some
straightforward matrix calculations as in the proof of Lemma 11.3.15
of~\cite{PhNotes}.
\end{proof}

\begin{cor}\label{GammaInj}
Let $X$ be a \chs, let $h \colon X \to X$ be a \mh,
let $D$ be a \uca,
and let $\af \in \Aut (C (X, D))$ lie over~$h$.
Let $Y \subset X$ be closed with $\sint (Y) \neq \varnothing$.
Then
the homomorphism $\gm$ of Proposition~\ref{BanachSpaceDirSum}
is injective.
\end{cor}

\begin{proof}
The proof is the same as the proof of Lemma 11.3.17 of~\cite{PhNotes}.
\end{proof}

\begin{lem}\label{ByMemberCond}
Let $X$ be a \chs, let $h \colon X \to X$ be a \mh,
let $D$ be a \uca,
and let $\af \in \Aut (C (X, D))$ lie over~$h$.
Let $Y \subset X$ be closed with $\sint (Y) \neq \varnothing$.
Let
\[
b = (b_{0}, b_{1}, \ldots, b_{l})
  \in \bigoplus_{k = 0}^{l} C (Y_{k}, M_{n (k)} (D)).
\]
Then
$b \in \gm \bigl( C^* \bigl( \Z, \, C (X, D), \, \af \bigr)_Y \bigr)$
if and only if whenever
\begin{itemize}
\item
$r > 0$,
\item
$k,t_{1},\ldots,t_{r} \in \set{0,\ldots,l}$,
\item
$n (t_{1}) + n (t_{2}) + \cdots + n (t_{r}) = n (k)$,
\item
$x \in (Y_{k} \setminus Y_{k}^{\circ}) \cap
Y_{t_{1}} \cap h^{-n (t_{1})} (Y_{t_{2}}) \cap \cdots \cap
h^{-[n (t_{1}) + \cdots + n (t_{r - 1}) ]} (Y_{t_{r}})$,
\end{itemize}
then $b_{k} (x)$ is given by the block
diagonal matrix
\[
b_{k} (x) = \left[ \begin{array}{cccc}
b_{t_{1}} (x) & & &
\\
& \af^{-n (t_{1})} (b_{t_{2}}) (x) & &
\\
& &
\ddots &
\\
& & & \af^{-[n (t_{1}) + \cdots + n (t_{r - 1})]}
(b_{t_{r}}) (x) \end{array} \right].
\]
\end{lem}

\begin{proof}
The proof is analogous (with appropriate changes to notation
and exponents) to the proof of Lemma 11.3.18 of~\cite{PhNotes}.
\end{proof}

We are now in position to give a decomposition of
$C^* \bigl( \Z, \, C (X, D), \, \af \bigr)_Y$
as a recursive subhomogeneous algebra over $D$.

\begin{thm}\label{AyRSHA}
Let $X$ be a \chs, let $h \colon X \to X$ be a \mh,
let $D$ be a \uca,
and let $\af \in \Aut (C (X, D))$ lie over~$h$.
Let $Y \subset X$ be closed with $\sint (Y) \neq \varnothing$.
Adopt Notation \ref{ModifiedRokhlinTower} and the notation of
Proposition \ref{AyHom}.
Then the homomorphism
\[
\gm \colon C^* \bigl( \Z, \, C (X, D), \, \af \bigr)_Y
    \to \bigoplus_{k = 0}^{l} C (Y_{k}, M_{n (k)} (D))
\]
of Proposition \ref{AyHom}
induces an isomorphism of $C^* \bigl( \Z, \, C (X, D), \, \af \bigr)_Y$
with the recursive subhomogeneous algebra over~$D$ defined,
in the notation of Definition \ref{RSHAMachinery}, as follows:
\begin{enumerate}
\item\label{9223_AyRSHA_1}
$l$ and $n (0),n (1),\ldots,n (l)$ are as in Notation
\ref{ModifiedRokhlinTower}.
\item\label{9223_AyRSHA_2}
$X_{k} = Y_{k}$ for $0 \leq k \leq l$.
\item\label{9223_AyRSHA_3}
$X_{k}^{(0)} = Y_{k} \cap \bigcup_{j= 0}^{k - 1} Y_{j}$.
\item\label{9223_AyRSHA_4}
For $x \in X_{k}^{(0)}$ and $b = (b_{0}, b_{1},\ldots, b_{k - 1})$
in the image in $\bigoplus_{j= 0}^{k - 1} C (Y_{j}, M_{n (j)} (D))$
of the $k - 1$ stage algebra
$C^* \bigl( \Z, \, C (X, D), \, \af \bigr)_Y^{(k - 1)}$, whenever
\[
x \in (Y_{k} \setminus Y_{k}^{\circ}) \cap Y_{t_{1}} \cap
h^{-n (t_{1})} (Y_{t_{2}}) \cap \cdots \cap h^{-[n (t_{1}) + \cdots +
n (t_{r - 1}]} (Y_{t_{r}})
\]
with
$n (t_{1}) + n (t_{2}) + \cdots + n (t_{r}) = n (k)$,
then $\ph_{k} (b (x))$ is given by the block
diagonal matrix
\[
\ph_{k} (b (x))
 = \left[ \begin{array}{cccc}
b_{t_{1}} (x) & & &
\\
& \af^{-n (t_{1})} (b_{t_{2}}) (x) & &
\\
& & \ddots &
\\
& & & \af^{-[n (t_{1}) + \cdots + n (t_{r - 1})]}
(b_{t_{r}}) (x) \end{array} \right].
\]
\item\label{9223_AyRSHA_5}
$\rho_{k}$ is the restriction map.
\end{enumerate}
The topological dimension of this decomposition is $\dim (X)$, and
the standard representation of
$\sm \bigl( C^* \bigl( \Z, \, C (X, D), \, \af \bigr)_Y \bigr)$
is the inclusion map in
$\bigoplus_{k = 0}^{l} C (Y_{k}, M_{n (k)} (D))$.
\end{thm}

\begin{proof}
The proof is analogous to the proof of Theorem 11.3.19
of~\cite{PhNotes}, while the proof of the statement regarding the
topological dimension is the same as for the proof given
in~\cite{PhNotes} for the part of Theorem 11.3.14 that is not
included in Theorem 11.3.19.
\end{proof}

\section{Structural properties of the orbit breaking subalgebra
 and the crossed product}\label{StructAY}

In this section,
we give some general theorems on the structure of \ca{s}
of the form $C^* \big( \Z, \, C_0 (X, D), \, \af \big)$,
which we derive from results on the structure of
orbit breaking subalgebras.
We give many explicit examples in Section~\ref{Sec_InfRed}.

\begin{cnv}\label{N_9215_AAY_ntn}
In this section,
as in Definition~\ref{D_4Y12_OrbSubalg}
but with additional restrictions,
$X$ will be an infinite \cms,
$h \colon X \to X$ will be a \mh,
$D$ will be a simple unital \ca,
and $\af \in \Aut ( C (X, D))$
will be an automorphism which lies over~$h$.
\end{cnv}

For a few results, simplicity is not needed.

\begin{prp}\label{AyRSHADirLim}
Adopt Convention~\ref{N_9215_AAY_ntn}.
For any nonempty closed set $Y \subset X$,
$C^* \bigl( \Z, \, C (X, D), \, \af \bigr)_Y$ is a direct limit of
recursive subhomogeneous algebras over~$D$
of the form $C^* \bigl( \Z, \, C (X, D), \, \af \bigr)_{Y_n}$
for closed subsets $Y_n \S X$ with $\sint (Y_n) \neq \E$.
\end{prp}

\begin{proof}
Given $Y \subset X$ closed, choose a sequence
$(Y_{n})_{n = 1}^{\infty}$ of closed subsets of $X$
with $\sint (Y_{n}) \neq \varnothing$ and
$Y_{n + 1} \subset Y_{n}$ for $n \in \N$,
and with $\bigcap_{n = 1}^{\infty} Y_{n} = Y$.
That $C^* \bigl( \Z, \, C (X, D), \, \af \bigr)_Y$ is
a direct limit of separable recursive subhomogeneous algebras
over $D$ follows by applying Theorem~\ref{AyRSHA} and
Proposition~\ref{AyDirLim}.
\end{proof}

\begin{prp}\label{AYDStable}
Adopt Convention~\ref{N_9215_AAY_ntn}.
Let $\D$ be a separable unital \ca{}
that is strongly selfabsorbing in the sense of \cite{TomsWinter2}.
Assume that $D$ is separable and $\D$-stable.
Let $Y \subset X$ be closed and nonempty.
Then $C^* \bigl( \Z, \, C (X, D), \, \af \bigr)_Y$ is $\D$-stable.
\end{prp}

\begin{proof}
If $\sint (Y) \neq \varnothing$,
this follows immediately from Theorem \ref{AyRSHA} and
Proposition \ref{DStableRSHA}.
The case of a general nonempty closed set
now follows from Proposition~\ref{AyRSHADirLim} above
and Corollary 3.4 of~\cite{TomsWinter2}.
\end{proof}

In particular, we get $\JS$-stability.

\begin{cor}\label{EStableGivesZStable}
Adopt Convention~\ref{N_9215_AAY_ntn}.
Let $\D$ be a separable unital \ca{}
that is strongly selfabsorbing in the sense of \cite{TomsWinter2}.
Assume that $D$ is separable and $\D$-stable.
Let $Y \subset X$ be closed and nonempty.
Then $C^* \bigl( \Z, \, C (X, D), \, \af \bigr)_Y$ is $\JS$-stable.
\end{cor}

\begin{proof}
Strongly selfabsorbing \ca{s} are $\JS$-stable
by Theorem~3.1 of~\cite{WinterZStab},
and $C^* \bigl( \Z, \, C (X, D), \, \af \bigr)_Y$ is $\D$-stable by
Proposition~\ref{AYDStable}, so the conclusion is immediate.
\end{proof}

We now turn to $\JS$-stability
of the crossed product $C^* \bigl( \Z, \, C (X, D), \, \af \bigr)$.
When $D$ is nuclear, it can be obtained from the results of
\cite{ABP-Zstab}.
However, if nuclearity is not assumed, then the
main conclusion of that paper only implies
that $C^* \bigl( \Z, \, C (X, D), \, \af \bigr)$ is tracially
$\JS$-stable (in the sense of~\cite{HirOr}).

\begin{thm}\label{ATraciallyZStable}
Adopt Convention~\ref{N_9215_AAY_ntn}.
Assume that $D$ is a simple separable unital $\JS$-stable
\ca{} which has a quasitrace.
Then $C^* \bigl( \Z, \, C (X, D), \, \af \bigr)$
is tracially $\JS$-stable.
If, in addition, $D$ is nuclear,
then $C^* \bigl( \Z, \, C (X, D), \, \af \bigr)$ is $\JS$-stable.
\end{thm}

\begin{proof}
Theorem 3.3 of~\cite{HirOr} implies that $D$ has strict comparison
of positive elements, and so
$C^* \bigl( \Z, \, C (X, D), \, \af \bigr)_Y$
is a centrally large subalgebra
of $C^* \bigl( \Z, \, C (X, D), \, \af \bigr)$ by
Corollary \ref{AyCentLargeConds}(\ref{9212_AyCentLargeConds_StrCmp}).
Also, $C^* \bigl( \Z, \, C (X, D), \, \af \bigr)_Y$
is $\JS$-stable by Corollary~\ref{EStableGivesZStable},
so Theorem~2.2 of~\cite{ABP-Zstab} implies
that $C^* \bigl( \Z, \, C (X, D), \, \af \bigr)$ is
tracially $\JS$-stable.
If $D$ is nuclear, then it is $\JS$-stable
by Theorem 4.1 of~\cite{HirOr},
and $C^* \bigl( \Z, \, C (X, D), \, \af \bigr)$ is nuclear.
So $\JS$-stability of $C^* \bigl( \Z, \, C (X, D), \, \af \bigr)$
follows from Theorem 2.3
of~\cite{ABP-Zstab}.
\end{proof}

\begin{cor}
Adopt Convention~\ref{N_9215_AAY_ntn}.
Let $D$ be a simple separable unital nuclear $\JS$-stable
\ca{} which has a quasitrace.
Then $C^* \bigl( \Z, \, C (X, D), \, \af \bigr)$ has nuclear
dimension at most 1.
\end{cor}

\begin{proof}
This follows immediately from Theorem~\ref{ATraciallyZStable} above
and Theorem~B of~\cite{CETWW}.
\end{proof}

Corollary~\ref{EStableGivesZStable} and its consequences
require no
assumption about the underlying dynamical system other than minimality.
In particular, if $D$ is $\JS$-stable then
$C^* \bigl( \Z, \, C (X, D), \, \af \bigr)_Y$ is
$\JS$-stable when $X$ is infinite dimensional, and even when
$\mdim (X, h)$ (as in Notation~\ref{N_0108_DimX}) is strictly positive.
If moreover $D$ is nuclear
then $C^* \bigl( \Z, \, C (X, D), \, \af \bigr)$ is
$\JS$-stable.
Thus, crossed products of the form
$C^* \bigl( \Z, \, C (X, D), \, \af \bigr)$ are
$\JS$-stable
whenever $h$ is minimal and $D$ is
simple, separable, unital, nuclear, and $\JS$-stable,
regardless of anything else about the underlying dynamics.

If $D$ is not assumed to be $\JS$-stable,
then we must make assumptions about the dynamical system
$(X, h)$ besides just minimality to expect
$C^* \bigl( \Z, \, C (X, D), \, \af \bigr)_Y$ to have
tractable structure.
The hypotheses in Proposition~\ref{TsrRR0A_Y},
Theorem~\ref{P_9Z26_XisCantor},
Proposition~\ref{P_9Z21_dim1},
and Theorem~\ref{T_9222_Tsr1AndRR0_CrPrd}
(which also include conditions on~$D$)
are surely much stronger than needed.
They have the advantage that the proofs can easily by obtained from
results already in the literature.

\begin{prp}\label{TsrRR0A_Y}
Adopt Convention~\ref{N_9215_AAY_ntn},
and assume that $X$ is the Cantor set.
Let $Y \subset X$ be closed and nonempty.
\begin{enumerate}
\item\label{TsrRR0A_Y_tsr}
If $\tsr (D) = 1$, then
$\tsr \bigl( C^* \bigl( \Z, \, C (X, D), \, \af \bigr)_Y \bigr) = 1$.
\item\label{TsrRR0A_Y_RR}
If $\RR (D) = 0$, then
$\RR \bigl( C^* \bigl( \Z, \, C (X, D), \, \af \bigr)_Y \bigr) = 0$.
\end{enumerate}
\end{prp}

\begin{proof}
First suppose that $\sint (Y) \neq \E$.
Then, from Lemma~\ref{CantorSetCase} and the remark immediately
after it,
we see there are are AF algebras $B_{0}, B_{1}, \ldots, B_{k}$
such that
\[
C^* \bigl( \Z, \, C (X, D), \, \af \bigr)_Y
  \cong \bigoplus_{k = 0}^{l} B_{k} \ten D.
\]
The result follows immediately.

Since stable rank one and real rank zero are preserved by direct limits,
the general case now follows from
Proposition~\ref{AyRSHADirLim}.
\end{proof}

\begin{thm}\label{P_9Z26_XisCantor}
Adopt Convention~\ref{N_9215_AAY_ntn},
and assume that $X$ is the Cantor set.
Further assume that one of the following conditions holds:
\begin{enumerate}
\item\label{9217_GetTsr1_AppInn}
All elements of $\set{ \af_{x} \colon x \in X} \subset \Aut (D)$
are approximately inner.
\item\label{9217_GetTsr1_StrCmp}
$D$ has strict comparison of positive elements.
\item\label{9217_GetTsr1_SP}
$D$ has property (SP) and the order on projections over~$D$
is determined by quasitraces.
\item\label{9217_GetTsr1_CptGen}
The set
$\bigl\{ \af_x^n \colon {\mbox{$x \in X$ and $n \in \Z$}} \bigr\}$
is contained in a subset of $\Aut (D)$
which is compact in the topology
of pointwise convergence in the norm on~$D$.
\end{enumerate}
If $\tsr (D) = 1$, then
$\tsr \bigl( C^* \bigl( \Z, \, C (X, D), \, \af \bigr) \bigr) = 1$,
and if also $\RR (D) = 0$, then
$\RR \bigl( C^* \bigl( \Z, \, C (X, D), \, \af \bigr) \bigr) = 0$.
\end{thm}

The action in this theorem has the Rokhlin property.
However, it seems to be unknown whether the crossed product of
even a
simple unital \ca{} with stable rank one by a Rokhlin automorphism
still has stable rank one.
See Problem~\ref{Pb_9922_RokhlinTsr} and the discussion afterwards.

We also emphasize that the conclusion for real rank zero
is only valid if stable rank one is also assumed.
The reason is the stable rank one hypothesis
in Theorem~6.4 of~\cite{ArPh}.
That theorem likely holds without stable rank one,
but this has not yet been proved.
For now, this is no great loss,
since, other than purely infinite simple \ca{s}
(which are covered by our Theorem~\ref{T_0101_PICase}),
there are no known examples of simple unital \ca{s} which have
real rank zero but not stable rank one.

\begin{proof}[Proof of Theorem~\ref{P_9Z26_XisCantor}]
Choose any one point subset $Y \S X$.
Since $\tsr (D) = 1$,
Proposition \ref{TsrRR0A_Y}(\ref{TsrRR0A_Y_tsr}) implies
that the subalgebra $C^* \bigl( \Z, \, C (X, D), \, \af \bigr)_Y$
(see Definition~\ref{D_4Y12_OrbSubalg})
has stable rank one.
If also $\RR (D) = 0$,
then Proposition \ref{TsrRR0A_Y}(\ref{TsrRR0A_Y_RR})  implies
$\tsr \big( C^* \bigl( \Z, \, C (X, D), \, \af \bigr)_Y \big) = 1$.
Now $C^* \bigl( \Z, \, C (X, D), \, \af \bigr)_Y$
is a centrally large subalgebra of
$C^* \bigl( \Z, \, C (X, D), \, \af \bigr)$
by the appropriate part of Corollary \ref{AyCentLargeConds}.
So $\tsr \big( C^* \bigl( \Z, \, C (X, D), \, \af \bigr) \big) = 1$
by Theorem 6.3 of~\cite{ArPh},
and if $\RR (D) = 0$, then Theorem~6.4 of~\cite{ArPh}
gives $\RR \big( C^* \bigl( \Z, \, C (X, D), \, \af \bigr) \big) = 0$.
\end{proof}

With more restrictive assumptions on~$D$,
enough is known to get results when $\dim (X) = 1$.

\begin{prp}\label{P_9027_FibPrdT_sr1}
Let $A$, $B$, and $C$ be \ca{s},
and let $\ph \colon A \to C$ and $\ps \colon B \to C$ be \hm{s}.
Let $D = A \oplus_{C, \ph, \psi} B$ (Definition~\ref{pullback}).
Assume that $\ps$ is surjective,
and that $A$ and $B$ have stable rank one.
Then $D$ has stable rank one.
\end{prp}

\begin{proof}
Set $J = {\operatorname{Ker}} (\ps)$.
We have a commutative diagram with exact rows (maps defined below):
\[
\begin{CD}
0 @>>> J @>{\ld}>> D @>{\pi}>> A @>>> 0  \\
& &  @VV{\id_B}V  @VV{\sm}V  @VV{\gm}V          \\
0 @>>> J @>{\io}>> B @>{\kp}>> B/J @>>> 0.
\end{CD}
\]
Here $\io$ is the inclusion of $J$ in~$B$
and $\kp$ is the quotient map,
$\pi \colon D \to A$ and $\sm \colon D \to B$
are the restrictions of the coordinate projections
$(a, b) \to a$ and $(a, b) \to b$,
and $\ld (b) = (b, 0)$ for $b \in J$.
For $a \in A$,
define $\gm (a)$ by
first choosing $b \in B$ such that $\ps (b) = \ph (a)$,
and then setting $\gm (a) = \kp (b)$.
It is easy to check that $\gm$ is well defined,
and that the diagram commutes.
All parts of exactness are immediate except that
surjectivity of $\pi$ follows from
surjectivity of~$\ps$.

The algebra $A$ has stable rank one by hypothesis,
and $J$ has stable rank one by Theorem~4.4 of~\cite{Rffl6}.
Next,
${\mathrm{ind}} \colon K_1 (B / J) \to K_0 (J)$ is the zero map
by Corollary~2 of~\cite{Nag},
so $K_0 (\io)$ is injective by the long exact sequence in
K-theory for the bottom row.
Since $K_0 (\io) = K_0 (\sm) \circ K_0 (\ld)$,
it follows that $K_0 (\ld)$ is injective.
Therefore ${\mathrm{ind}} \colon K_1 (A) \to K_0 (J)$ is the zero map
by the long exact sequence in
K-theory for the top row.
Now $D$ has stable rank one by Corollary~2 of~\cite{Nag}.
\end{proof}

We state for convenient reference the result on the stable rank
of direct limits.

\begin{lem}\label{L_0111_tsr_lim}
Let $(A_{\ld})_{\ld \in \Ld}$ be a direct system of \ca{s}.
Then
$\tsr \bigl( \dirlim_{\ld} A_{\ld} \bigr)
  \leq \liminf_{\ld} \tsr (A_{\ld})$.
\end{lem}

We do not assume that the maps of the direct system are injective.

\begin{proof}[Proof of Lemma~\ref{L_0111_tsr_lim}]
When $\Ld = \N$,
this is Theorem~5.1 of~\cite{Rffl6}.
The proof for general direct limits is the same.
\end{proof}

\begin{prp}\label{P_9027_CXD_Tsr1}
Let $D$ be a simple unital \ca{}
with stable rank one and real rank zero,
and suppose that $K_1 (D) = 0$.
Let $X$ be a \chs{} with $\dim (X) \leq 1$.
Then $C (X, D)$ has stable rank one.
\end{prp}

\begin{proof}
If $X$ is a point, this is trivial.
If $X = [0, 1]$,
it follows from Theorem~4.3 of~\cite{NagOsPh2001}.
If $X$ is a one dimensional finite complex,
the result follows from these facts
and Proposition~\ref{P_9027_FibPrdT_sr1}
by induction on the number of cells.

Now suppose $X$ is a \cms.
By Corollary 5.2.6 of~\cite{SkiK},
there is an inverse system $(X_n)_{n \in \Nz}$
of one dimensional finite complexes
such that $X \cong \invlim_n X_n$.
Then $C (X, D) \cong \dirlim_n C (X_n, D)$.
Since $C (X_n, D)$ has stable rank one for all $n \in \Nz$,
it follows from Lemma~\ref{L_0111_tsr_lim}
that $C (X, D)$ has stable rank one.

Finally,
let $X$ be a general \chs{} with $\dim (X) \leq 1$.
By Theorem~1 of~\cite{Mdsc}
(in Section~3 of that paper),
there is an inverse system $(X_{\ld})_{\ld \in \Ld}$
of one dimensional \cms{s}
such that $X \cong \invlim_{\ld} X_{\ld}$.
It now follows that $C (X, D)$ has stable rank one
by the same reasoning as in the previous paragraph.
\end{proof}

\begin{prp}\label{P_9027_Rsha_Tsr1}
Let $D$ be a simple unital \ca.
Let $R$ be a recursive subhomogeneous \ca{} over~$D$
(Definition~\ref{RSHA})
with topological dimension at most~$1$
(Definition \ref{RSHAMachinery}(\ref{9223_RSHAMachinery_6})).
If $D$ has stable rank one and real rank zero,
and $K_1 (D) = 0$,
then $R$ has stable rank one.
\end{prp}

\begin{proof}
The proof is the same as that of Proposition~\ref{DStableRSHA},
except using Proposition~\ref{P_9027_FibPrdT_sr1}
in place of Lemma~\ref{DStablePullback},
and using Proposition~\ref{P_9027_CXD_Tsr1}.
\end{proof}

\begin{prp}\label{P_9Z21_dim1}
Adopt Convention~\ref{N_9215_AAY_ntn},
and assume that $\dim (X) \leq 1$,
that $D$ has stable rank one and real rank zero,
and that $K_1 (D) = 0$.
Let $Y \subset X$ be closed and nonempty.
Then
$\tsr \bigl( C^* \bigl( \Z, \, C (X, D), \, \af \bigr)_Y \bigr) = 1$.
\end{prp}

\begin{proof}
First suppose that $\sint (Y) \neq \E$.
Then the recursive subhomogeneous algebra over~$D$
in Theorem~\ref{AyRSHA}
has base spaces $Y_k$ which,
following Notation~\ref{ModifiedRokhlinTower},
are closed subsets of~$X$.
By Proposition 3.1.3 of~\cite{Prs},
they therefore have covering dimension at most~$1$.
The result now follows from Proposition~\ref{P_9027_Rsha_Tsr1}.

Stable rank one is preserved by direct limits
(Lemma~\ref{L_0111_tsr_lim}),
so the general case now follows from
Proposition~\ref{AyRSHADirLim}.
\end{proof}

The obvious examples of \mh{s}
of one dimensional spaces are irrational rotations on~$S^1$,
but there are others.
See, for example, \cite{GjdJhn},
and the work on \mh{s} of the product of the Cantor set
and the circle in \cite{LnMti1} and~\cite{LnMti3}.

\begin{thm}\label{T_9222_Tsr1AndRR0_CrPrd}
Adopt Convention~\ref{N_9215_AAY_ntn},
and assume that $\dim (X) \leq 1$,
that $D$ has stable rank one and real rank zero,
and that $K_1 (D) = 0$.
Further assume that one of the following conditions holds:
\begin{enumerate}
\item\label{9927_AyCentLargeConds_AppInn}
All elements of $\set{ \af_{x} \colon x \in X} \subset \Aut (D)$
are approximately inner.
\item\label{9927_AyCentLargeConds_SP}
The order on projections over~$D$
is determined by quasitraces.
\item\label{9927_AyCentLargeConds_CptGen}
The set
$\bigl\{ \af_x^n \colon {\mbox{$x \in X$ and $n \in \Z$}} \bigr\}$
is contained in a subset of $\Aut (D)$
which is compact in the topology
of pointwise convergence in the norm on~$D$.
\end{enumerate}
Then $C^* \bigl( \Z, \, C (X, D), \, \af \bigr)$
has stable rank one.
\end{thm}

\begin{proof}
Combine Theorem 6.3 of~\cite{ArPh},
Corollary~\ref{AyCentLargeConds},
and Proposition~\ref{P_9027_Rsha_Tsr1}.
\end{proof}

\section{Minimality of products of Denjoy homeomorphisms
 of the Cantor set}\label{Sec_9929_MinProd}

\indent
For use in examples in the next section,
we prove minimality of products of Denjoy homeomorphisms
with rationally independent rotation numbers.
We have not found this result in the literature.
The method we use here is not the traditional approach
to this kind of problem,
but we hope it will be useful elsewhere.

\begin{ntn}\label{N_9X04_Rotation}
For $\te \in \R$ we let $r_{\te} \colon S^1 \to S^1$
denote the rotation by $2 \pi \te$,
that is, $r_{\te} (\zt) = \exp (2 \pi i \te)  \zt$
for $\zt \in S^1$.
\end{ntn}

\begin{prp}\label{P_4X04_ProdRotation}
Let $I$ be a set,
and let $(\te_i)_{i \in I}$ be a family of elements of~$\R$.
Suppose that $1$ and the numbers $\te_i$ for $i \in I$ are
linearly independent over~$\Q$.
Then the product $r \colon (S^1)^{I} \to (S^1)^{I}$
of the rotations $r_{\te_i}$ is minimal.
\end{prp}

\begin{proof}
If $I$ is finite,
this is Proposition 1.4.1 of~\cite{KtHsb}.
The general case follows from the finite case by taking the inverse
limit of the products over finite subsets of~$I$,
by the discussion after Proposition II.4 of~\cite{Frtg}.
The proof is is written for actions of the semigroup~$\Nz$
and for countable inverse limits,
but the proof for actions of $\Z$ and arbitrary inverse limits
is the same.
\end{proof}

\begin{dfn}\label{D_9929_StrOpen}
Let $X$ and $Y$ be topological spaces,
and let $g \colon Y \to X$.
Let $y \in Y$.
We say that $g$ is {\emph{strictly open}} at~$y$
if
\[
\bigl\{ g^{-1} (U) \colon
 {\mbox{$U \S X$ is open and $g (y) \in U$}} \bigr\}
\]
is a neighborhood base for $y$ in~$Y$.
\end{dfn}

Some of the theory below works without assuming $g$ is \ct,
so we do not require that \nbhd{s} be open.

The next lemma gives an alternate interpretation of strict openness.

\begin{lem}\label{L_9929_WhenStrOpen}
In the situation in Definition~\ref{D_9929_StrOpen},
assume that $g$ is \ct{} and surjective, $Y$ is compact,
and $X$ and $Y$ are Hausdorff.
Then $g$ is strictly open at~$y$ \ifo{}
the following conditions hold:
\begin{enumerate}
\item\label{Item_D_9929_Inj}
$g^{-1} (\{ g (y) \}) = \{ y \}$.
\item\label{Item_D_9929_Open}
For every open set $V \S Y$ with $y \in V$,
the set $g (V)$ is a \nbhd{} of $g (y)$.
\end{enumerate}
\end{lem}

The lemma is false without compactness of~$Y$.
Take
\[
X = [0, 1]
\andeqn
Y = \{ (0, 0) \} \cup \bigl( (0, 1]\times [0, 1] \bigr) \S [0, 1]^2,
\]
let $g$ be projection to the first coordinate,
and take $y = (0, 0)$.
Then (\ref{Item_D_9929_Inj}) and~(\ref{Item_D_9929_Open}) hold,
but $g$ is not strictly open at~$y$.

\begin{proof}[Proof of Lemma~\ref{L_9929_WhenStrOpen}]
Assume that $g$ is strictly open at~$y$.
To prove~(\ref{Item_D_9929_Inj}),
let $z \in Y \setminus \{ y \}$,
choose $V \S Y$ open with $y \in V$ and $z \not\in V$,
and choose $U \S X$ open such that $g (y) \in U$
and $g^{-1} (U) \S V$.
Then $z \not\in g^{-1} (U)$, so $g (z) \neq g (y)$.

To prove~(\ref{Item_D_9929_Open}),
let $V \S Y$ be open with $y \in V$.
Choose $U \S X$ open such that $g (y) \in U$
and $g^{-1} (U) \S V$.
Surjectivity of~$g$ implies $U \S g ( g^{-1} (U) ) \S g (V)$,
so $g (V)$ is a \nbhd{} of $g (y)$.

For the converse,
assume~(\ref{Item_D_9929_Open}),
and suppose that $g$ is not strictly open at~$y$.
Let $\cN$ be the set of open subsets of~$X$ which contain $g (y)$,
ordered by reverse inclusion.
Since $g$ is \ct,
failure of strict openness means we can
choose $V \S Y$ open with $y \in V$
and such that for every $U \in \cN$,
we have $g^{-1} (U) \not\S V$.
So there is $z_U \in Y$ such that $z_U \not\in V$ but $g (z_U) \in U$.
The net $(z_U)_{U \in \cN}$
satisfies $g (z_U) \to g (y)$.
Choose a subnet $(t_{\ld})_{\ld \in \Ld}$
such that $t = \lim_{\ld} t_{\ld}$ exists.
Then $t \not\in V$.
Also $g (t_{\ld}) \to g (t)$ by continuity and $g (t_{\ld}) \to g (y)$
by construction, so $g (t) = g (y)$.
We have contradicted~(\ref{Item_D_9929_Inj}).
\end{proof}

\begin{lem}\label{L_9929_ExtMin}
Let $X$ and $Y$ be \chs{s},
and let $h \colon X \to X$ and $k \colon Y \to Y$ be \hme{s}.
Let $g \colon Y \to X$ be continuous and surjective,
and satisfy $g \circ k = h \circ g$.
Suppose that $h$ is minimal and there is a dense subset $S \S Y$
such that $g$ is strictly open at every point of~$S$.
Then $k$ is minimal.
\end{lem}

\begin{proof}
It is enough to show that for every $y \in Y$ and every
nonempty open set $V \S Y$,
there is $n \in \Z$ such that $k^n (y) \in V$.
Choose $z \in S \cap V$.
Since $g$ is strictly open at~$z$,
there is an open set $U \S X$ such that $g (z) \in U$
and $g^{-1} (U) \S V$.
Since $h$ is minimal,
there is $n \in \Z$ such that $h^n (g (y)) \in U$.
Then $g (k^n (y)) \in U$,
so $k^n (y) \in g^{-1} (U) \S V$.
\end{proof}

It is well known that Lemma~\ref{L_9929_ExtMin} fails
without strict openness,
even if one assumes that $g^{-1} (x)$ is finite for all $x \in X$
and has only one point for a dense subset of points $x \in X$.
The following example was suggested by B.~Weiss.

\begin{exa}\label{E_9X04_NeedStrict}
Fix $\te \in \R \SM \Q$.
Take $X = S^1$ and take $h = r_{\te}$
(rotation by $2 \pi \te$; Notation~\ref{N_9X04_Rotation}).
Define $Y \S S^1 \times [0, 1]$ to be
\[
Y = \bigl( S^1 \times \{ 0 \} \bigr)
 \cup \left\{ \left( e^{2 \pi i n \te}, \, \frac{1}{| n | + 1} \right)
     \colon n \in \Z \right\},
\]
and define $k \colon Y \to Y$ to be
$k (\zt, 0) = (r_{\te} (\zt), \, 0)$ for $\zt \in S^1$ and
\[
k \left( e^{2 \pi i n \te}, \, \frac{1}{| n | + 1} \right)
   = \left( e^{2 \pi i (n + 1) \te}, \, \frac{1}{| n + 1 | + 1} \right)
\]
for $n \in \Z$.
Let $g \colon Y \to X$ be projection to the first coordinate.
Then $g \circ k = h \circ g$,
$g$ is surjective,
$h$ is minimal,
$k$ is not minimal,
$g^{-1} (x)$ has at most two points for every $x \in X$,
and $g^{-1} (x)$ has one point  for all but countably many $x \in X$.
\end{exa}

The following two easy lemmas will be used in the construction of
examples.

\begin{lem}\label{L_9X04_Prod}
Let $I$ be a nonempty set,
and for $i \in I$ let $g_i \colon X_i \to Y_i$
and $y_i \in Y_i$ be as in Definition~\ref{D_9929_StrOpen},
with $g_i$ strictly open at~$y_i$.
Set
\[
X = \prod_{i \in I} X_i
\andeqn
Y = \prod_{i \in I} Y_i,
\]
let $g \colon Y \to X$ be given by $g (z) = ( g_i (z_i))_{i \in I}$
for $z = (z_i)_{i \in I} \in Y$,
and set $y = (y_i)_{i \in I}$.
Then $g$ is strictly open at~$y$.
\end{lem}

\begin{proof}
It is obvious that $g^{-1} ( \{ g (y) \} ) = \{ y \}$.
Now let $V \S Y$ be open with $y \in V$.
We need an open set $U \S X$
such that $g (y) \in U$, $g^{-1} (U)$ is a \nbhd{} of~$y$,
and $g^{-1} (U) \S V$.
\Wolog{} there are open subsets $V_i \S Y_i$
and a finite set $F \S I$
such that $V = \prod_{i \in I} V_i$,
$V_i = Y_i$ for all $i \in I \SM F$,
and $y_i \in V_i$ for all $i \in I$.
For $i \in F$ choose an open set $U_i \S X_i$ with $g_i (y_i) \in U_i$
and such that $g_i^{-1} (U_i) \S V_i$.
Set $U_i = X_i$ for $i \in I \SM F$.
Since $g_i^{-1} (U_i)$ is a \nbhd{} of~$y_i$
for all $i \in I$
and $U_i = X_i$ for $i \in I \SM F$,
the set $g^{- 1} (U) = \prod_{i \in I} g_i^{-1} (U_i)$
is a \nbhd{} of~$y$.
So the set $U = \prod_{i \in I} U_i$ is the required set.
\end{proof}

\begin{lem}\label{L_9X04_Subsp}
In the situation in Definition~\ref{D_9929_StrOpen},
assume that $g$ is strictly open at~$y$.
Let $Y_0 \S Y$ be a subspace such that $y \in Y_0$.
Then $g |_{Y_0}$ is strictly open at~$y$.
\end{lem}

\begin{proof}
The result is immediate from the definition of the subspace
topology.
\end{proof}

\begin{lem}\label{L_9929_S1toS1}
Let $g \colon S^1 \to S^1$ be \ct{} and surjective.
Suppose that there are disjoint closed arcs
$I_1, I_2, \ldots \S S^1$
such that:
\begin{enumerate}
\item\label{Item_L_9929_S1toS1_Const}
For $n \in \N$, the function $g$ is constant on~$I_n$.
\item\label{Item_L_9929_S1toS1_NI}
$\bigcup_{n = 1}^{\I} I_n$ is dense in~$S^1$.
\item\label{Item_L_9929_S1toS1_Inj}
With $R = S^1 \SM \bigcup_{n = 1}^{\I} I_n$,
the restriction $g |_R$ is injective.
\end{enumerate}
Then $g$ is strictly open at every point of~$R$.
\end{lem}

\begin{proof}
For distinct $\ld, \zt \in S^1$,
we denote the open, closed, and half open arcs from $\ld$ to~$\zt$
using interval notation:
$(\ld, \zt)$, $[\ld, \zt]$, $[\ld, \zt)$, and $(\ld, \zt]$.
Further write $I_n = [\bt_n, \gm_n]$ with $\bt_n, \gm_n \in S^1$,
and set $B = \{ \bt_n \colon n \in \N \}$
and $C = \{ \gm_n \colon n \in \N \}$.

The set $g (S^1 \SM R)$ is countable
and $g$ is surjective, so $S^1 \SM g (R)$ is countable.

First,
we claim that if $E, F \S S^1$ are disjoint,
then $\sint (g (E)) \cap \sint (g (F)) = \E$.
If the claim is false,
then $g (E) \cap g (F)$ is uncountable.
Now $g (E) \SM g (E \cap R) \S g (S^1 \SM R)$
which is countable, so $g (E) \SM g (E \cap R)$ is countable.
Similarly $g (F) \SM g (F \cap R)$ is countable.
So $g (E \cap R) \cap g (F \cap R)$ is uncountable.
This contradicts $E \cap F = \E$ and injectivity of $g |_R$.

Second, we claim that if $J \S S^1$ is any nonempty open arc such that
there is no $n \in \N$ with ${\overline{J}} \S I_n$,
then $R \cap J \neq \E$.
Suppose the claim is false.
Write $J = (\ld, \zt)$ with $\ld, \zt \in S^1$.
Define
\[
S = \begin{cases}
   \E & \hspace*{1em}        \ld \not\in R
        \\
   \{ \ld \} & \hspace*{1em} \ld \in R
\end{cases}
\andeqn
T = \begin{cases}
   \E & \hspace*{1em}        \zt \not\in R
        \\
   \{ \zt \} & \hspace*{1em} \zt \in R.
\end{cases}
\]
Then ${\overline{J}}$ is the disjoint union of the closed sets
\[
S, \, T, \, {\overline{J}} \cap I_1, \, {\overline{J}} \cap I_2, \ldots.
\]
Suppose there is $n \in \N$ such that ${\overline{J}} \cap I_m = \E$
for all $m \neq n$.
Since ${\overline{J}} \not\subset I_n$,
there is a nonempty open arc $L \S J$
such that $L \cap I_n = \E$,
whence $L \cap I_m = \E$ for all $m \in \N$.
This contradicts~(\ref{Item_L_9929_S1toS1_NI}).
Thus, at least two of
the sets ${\overline{J}} \cap I_n$ are nonempty.
But, according to
Theorem~6 in Part III of Section~47 (in Chapter~5) of~\cite{Krtw},
a connected \chs{} cannot be the disjoint union of closed subsets
$E_1, E_2, \ldots$ with at least two of them nonempty.
This contradiction proves the claim.

Third, we claim that, under the same hypotheses, $R \cap J$
is infinite.
Again write $J = (\ld, \zt)$ with $\ld, \zt \in S^1$.
Apply the previous claim
to choose $\rh_1 \in R \cap J$ and set $J_1 = (\rh_1, \zt)$.
Then $\rh_1 \not\in \bigcup_{n = 1}^{\I} I_n$
so there is no $n \in \N$ with ${\overline{J_1}} \S I_n$.
Apply the previous claim again, choosing $\rh_2 \in R \cap J_1$,
and showing that $J_2 = (\rh_2, \zt)$
satisfies ${\overline{J_1}} \not\S I_n$ for all $n \in \N$.
Proceed inductively.

Fourth, we claim that if $\ld \in R \cup C$ and $\zt \in R \cup B$
are distinct,
then $g \bigl( (\ld, \zt) \bigr)$ is an open arc.
For the proof, set $J = (\ld, \zt)$.
We can assume that $g (J) \neq S^1$.
Since $g (J)$ is connected, there are $\mu, \nu \in S^1$
such that $g (J)$ is one of
$(\mu, \nu)$, $[\mu, \nu]$
(in this case we allow $\mu = \nu$), $[\mu, \nu)$, or $(\mu, \nu]$.
Suppose $\mu \in g (J)$; we will obtain a contradiction.
Choose $\rh \in J$ such that $g (\rh) = \mu$.
There are two cases: $\rh \in R$ and $\rh \in \bigcup_{n = 1}^{\I} I_n$.

In the first case, let $L_0$ and $L_1$ be the nonempty arcs
$L_0 = (\ld, \rh)$ and $L_1 = (\rh, \zt)$.
For $j = 0, 1$,
since $\rh \in {\overline{L_j}}$,
the arc $L_j$ satisfies the hypotheses of the second claim.
So $R \cap L_j$ is infinite by the third claim.
Since $g |_R$ is injective, $g (L_j)$ is infinite.
Clearly $\mu \in {\overline{g (L_j)}} \S [\mu, \nu]$.
Also, $g (L_j)$ is an arc by connectedness.
So there is $\om_j \in (\mu, \nu)$ such that $(\mu, \om_j) \S g (L_j)$.
This implies $\sint ( g (L_0) ) \cap \sint ( g (L_1) ) \neq \E$,
contradicting the first claim.

In the second case,
let $n \in \N$ be the integer such that $\rh \in I_n$.
Both $\ld \in [\bt_n, \gm_n)$ and $\zt \in (\bt_n, \gm_n]$
are impossible, but $[\bt_n, \gm_n] \cap (\ld, \zt) \neq \E$,
so $[\bt_n, \gm_n] \subset (\ld, \zt)$.
Therefore the arcs $L_0 = (\ld, \bt_n)$ and $L_1 = (\gm_n, \zt)$
are nonempty and do not intersect~$I_n$.
If there is $m \in \N$ such that ${\overline{L_0}} \S I_m$,
then clearly $m \neq n$,
but then $\bt_n \in I_m \cap I_n$,
contradicting $I_m \cap I_n = \E$.
So no such $m$ exists,
whence $L_0$ satisfies the hypotheses of the second claim.
Similarly $L_1$ satisfies the hypotheses of the second claim.
Since $g$ is constant on~$I_n$,
we have $g (\bt_n) = g (\gm_n) = g (\rh) = \mu$.
Now we get a contradiction in the same way as in the first case.

So $\mu \not\in g (J)$.
Similarly $\nu \not\in g (J)$.
The claim is proved.

Fifth,
we claim that if $L \S S^1$ is an open arc,
then $g^{-1} (L)$ is an open arc.
To prove the claim,
since $g^{-1} (L)$ is open,
there are an index set $S$
and disjoint nonempty open arcs $J_s$ for $s \in S$
such that $g^{-1} (L) = \bigcup_{s \in S} J_s$,
and there are $\ld_s, \zt_s \in S^1$
such that $J_s = (\ld_s, \zt_s)$.

Let $s \in S$.
Suppose $\ld_s \in I_n$ for some $n \in \N$.
Since $g$ is constant on~$I_n$,
either $I_n \cap g^{-1} (L) = \E$
or $I_n \S g^{-1} (L)$.
In the second case, $I_n \cup (\ld_s, \zt_s)$
is a connected subset of $g^{-1} (L)$.
Since $(\ld_s, \zt_s)$ is a maximal connected subset of $g^{-1} (L)$,
we have $I_n \S (\ld_s, \zt_s)$.
This contradicts $\ld_s \in I_n$.
So the first case must apply, and therefore $\ld_s = \gm_n$.
Combining this with the possibility
$\ld_s \not\in \bigcup_{n = 1}^{\I} I_n$,
we conclude that $\ld_s \in R \cup C$.
Similarly, $\zt_s \in R \cup B$.
By the previous claim,
$g \bigl( (\ld_s, \zt_s) \bigr)$ is open.

The first claim now implies that the sets
$g \bigl( (\ld_s, \zt_s) \bigr)$ are all disjoint.
Since $L$ is connected,
it follows that $\card (S) = 1$,
which implies the claim.

To prove the lemma,
let $\et \in R$ and let $V \S S^1$ be open with $\et \in V$.
We need an open set $U \S S^1$
such that $g (\et) \in U$ and $g^{-1} (U) \S V$.
Choose $\ld_0, \zt_0 \in S^1$
such that $\et \in (\ld_0, \zt_0) \S V$.
Use the second claim to choose $\ld_1 \in (\ld_0, \et) \cap R$
and $\zt_1 \in (\et, \zt_0) \cap R$.
Then $U = g \bigl( (\ld_1, \zt_1) \bigr)$ is open by
the fourth claim,
and by the fifth claim there are $\ld_2, \zt_2 \in S^1$
such that $g^{-1} (U) = (\ld_2, \zt_2)$.
If $\ld_1 \in g^{-1} (U)$,
then $(\ld_2, \ld_1)$ is a nonempty open arc contained in $g^{-1} (U)$
with $\ld_1 \in R$.
The third claim and injectivity of $g |_R$
imply that $g \bigl( (\ld_2, \ld_1) \bigr)$
is not a point.
Since $g \bigl( (\ld_2, \ld_1) \bigr)$ is connected,
it has nonempty interior.
It is contained in $U = g \bigl( (\ld_1, \zt_1) \bigr)$,
contradicting the first claim.
So $\ld_1 \not\in (\ld_2, \zt_2)$.
Similarly $\zt_1 \not\in (\ld_2, \zt_2)$.
Since $\et \in (\ld_2, \zt_2)$,
$\et \in (\ld_1, \zt_1)$, and $(\ld_2, \zt_2)$ is connected,
it now follows that $(\ld_2, \zt_2) \S (\ld_1, \zt_1)$.
(In fact, we have equality, since $g \bigl( (\ld_1, \zt_1) \bigr) = U$.)
Thus
\[
g^{-1} (U)
 = (\ld_2, \zt_2)
 \S (\ld_1, \zt_1)
 \S (\ld_0, \zt_0)
 \S V.
\]
This completes the proof.
\end{proof}

We have not seen a term in the literature for the following
class of \hme{s}.

\begin{dfn}\label{D_9X04_RestrDj}
A {\emph{restricted Denjoy homeomorphism}}
is a homeomorphism of the Cantor set which is conjugate to
the restriction and corestriction of a Denjoy homeomorphism of~$S^1$,
in the sense of Definition~3.3 of~\cite{PtSmdSk},
to the unique minimal set of Proposition~3.4 of~\cite{PtSmdSk}.
The {\emph{rotation number}}
of a restricted Denjoy homeomorphism is the rotation number,
as at the beginning of Section~3 of~\cite{PtSmdSk},
of the Denjoy homeomorphism of~$S^1$
of which it is the restriction;
this number is well defined by Remark~3 at the end of Section~3
of~\cite{PtSmdSk}.
\end{dfn}

\begin{lem}\label{L_9X04_DjSemiCj}
Let $X$ be the Cantor set,
and let $k \colon X \to X$ be a restricted Denjoy homeomorphism
with rotation number~$\te$,
as in Definition~\ref{D_9X04_RestrDj}.
Then there are a \ct{} surjective map $g \colon X \to S^1$
and a dense subset $R \S X$ such that
$g$ is strictly open at~$x$ for every $x \in R$
and, following Notation~\ref{N_9X04_Rotation},
$g \circ h = r_{\te} \circ g$.
\end{lem}

\begin{proof}
We may assume that $k_0 \colon S^1 \to S^1$ is a Denjoy homeomorphism
as in Definition~3.3 of~\cite{PtSmdSk},
that $X \S S^1$ is the unique minimal set for~$k_0$,
and that $k$ is the restriction and corestriction of $k_0$ to~$X$.
Corollary~3.2 of~\cite{PtSmdSk}
gives $h \colon S^1 \to S^1$
such that $h \circ k_0 = r_{\te} \circ h$.
Set $g = h |_X$.
Since $X \S S^1$ is compact and invariant and $r_{\te}$ is minimal,
$g$ must be surjective.
The discussion after Proposition~3.4 of~\cite{PtSmdSk}
implies that $h$ satisfies the hypotheses of Lemma~\ref{L_9929_S1toS1},
with $X = S^1 \SM \bigcup_{n = 1}^{\I} \sint (I_n)$.
Moreover, the set $R$ of Lemma~\ref{L_9929_S1toS1}
is nonempty by
Theorem~6 in Part III of Section~47 (in Chapter~5) of~\cite{Krtw},
and $k$-invariant,
hence dense in~$X$.
By Lemma~\ref{L_9929_S1toS1} and Lemma~\ref{L_9X04_Subsp},
$g$ is strictly open at every point of~$R$.
\end{proof}

\begin{prp}\label{P_4X04_ProdRstDj}
Let $I$ be a set,
let $(\te_i)_{i \in I}$ be a family of elements of~$\R$,
and for $i \in I$ let $k_i \colon X_i \to X_i$
be a restricted Denjoy homeomorphism with rotation number~$\te_i$,
as in Definition~\ref{D_9X04_RestrDj}.
Suppose that $1$ and the numbers $\te_i$ for $i \in I$ are
linearly independent over~$\Q$.
Then the product $k \colon \prod_{i \in I} X_i \to \prod_{i \in I} X_i$
of the \hme{s} $k_i$ is minimal.
\end{prp}

\begin{proof}
For $i \in I$,
let $g_i \colon X_i \to S^1$ be as in Lemma~\ref{L_9X04_DjSemiCj}.
By Lemma~\ref{L_9X04_DjSemiCj},
there is a dense subset $R_i \S X_i$
such that $g_i$ is strictly open at every point of~$R_i$.
Let $g \colon \prod_{i \in I} X_i \to (S^1)^{I}$
be the product of the maps~$g_i$.
Lemma~\ref{L_9X04_Prod} implies that
$g$ is strictly open at every point of $\prod_{i \in I} R_i$,
this set is dense,
and the product $r \colon (S^1)^{I} \to (S^1)^{I}$
of the rotations $r_{\te_i}$ is minimal
by Proposition~\ref{P_4X04_ProdRotation}.
So $k$ is minimal
by Lemma~\ref{L_9929_ExtMin}.
\end{proof}

\section{Examples}\label{Sec_InfRed}

\indent
To take full advantage of Proposition~\ref{AyRSHADirLim},
we need information on the structure of simple
direct limits of recursive subhomogeneous algebras
over the algebra $D$ which appears there.
For example,
we need conditions for such direct limits to have stable rank~$1$
(as originally done for $D = \C$ in~\cite{DNNP})
and to have real rank zero
(as originally done for $D = \C$ in~\cite{BDR}).
We intend to study this problem in a future paper.
Even without generalizations of those theorems,
three cases are accessible now.
They are $h \colon X \to X$ wild
(for example, with strictly positive mean dimension,
as in Notation~\ref{N_0108_DimX})
but with $D$ being $\JS$-stable;
$X$ is the Cantor set
and $D$ has stable rank one but is otherwise wild;
and $X = S^1$,
$h$ is an irrational rotation,
and $D$ has stable rank one, real rank zero,
and trivial $K_1$-group,
but is otherwise wild.

We also give examples for Theorem~\ref{T_0101_PICase},
although only for $G = \Z$.

In the first type of example,
if $X$ is \fd{} then the action has a
higher dimensional Rokhlin property,
and results on crossed products by such actions can be applied.
As far as we know, however,
there are no previously known general theorems
which apply when $X$ is in\fd.
In the second type of example,
the action has the Rokhlin property.
Since the algebra is not simple,
results of~\cite{OP1} can't be applied,
and, when $D$ is not $\JS$-stable
and does not have finite nuclear dimension,
Theorems 4.1 and~5.8 of~\cite{HWZ} can't be applied.
We know of no previous general theorems which apply in this case.
Actions in the third type of example
are further away from known results:
the situation is like for the second type of example,
but the Rokhlin property must be weakened to
finite Rokhlin dimension with commuting towers.

Since we will use quasifree automorphisms of reduced C*-algebras
of free groups
and the free shift on $C^*_{\mathrm{r}} (F_{\infty})$
in several examples,
we establish notation for them separately.

\begin{ntn}\label{N_9214_QFreeFn}
For $n \in \N \cup \{ \I \}$,
we let $F_n$ denote the free group on $n$~generators.
We let $u_1, u_2, \ldots, u_n \in C^*_{\mathrm{r}} (F_n)$
(when $n = \I$,
for $u_1, u_2, \ldots \in C^*_{\mathrm{r}} (F_{\infty})$)
be the ``standard'' unitaries in $C^*_{\mathrm{r}} (F_n)$,
obtained as the images of the standard generators of~$F_n$.
For $\zt = (\zt_1, \zt_2, \ldots, \zt_n) \in (S^1)^n$
(when $n = \I$,
for $\zt = (\zt_n)_{n \in \N} \in (S^1)^{\N}$),
we let $\ph_{\zt} \in \Aut ( C^*_{\mathrm{r}} (F_n) )$
be the (quasifree) automorphism
determined by $\ph_{\zt} (u_k) = \zt_k u_k$
for $k = 1, 2, \ldots, n$
(when $n = \I$,
for $k = 1, 2, \ldots)$.
It is well known that $\zt \mapsto \ph_{\zt}$ is \ct.
\end{ntn}

\begin{ntn}\label{N_9214_FreeShift}
Take the standard generators of the free group $F_{\infty}$
to be indexed by~$\Z$,
and for $n \in \Z$ let $u_n \in C^*_{\mathrm{r}} (F_{\infty})$
be the unitary obtained as the image of the corresponding
generator of~$F_{\infty}$.
We denote by $\sm$ the free shift on $C^*_{\mathrm{r}} (F_{\infty})$,
that is, the automorphism
$\sm \in \Aut ( C^*_{\mathrm{r}} (F_{\infty}) )$
determined by $\sm (u_n) = u_{n + 1}$
for $n \in \Z$.
\end{ntn}

\begin{exa}\label{E_9214_ShiftAndUHF}
Set $X_0 = (S^1)^{\Z}$.
Let $h_0 \colon X_0 \to X_0$
be the (forwards) shift,
defined for $\zt = (\zt_n)_{n \in \Z} \in X_0$ by
$h_0 (\zt) = (\zt_{n - 1})_{n \in \Z}$.
Let $D = \bigotimes_{n \in \N} M_2$
be the $2^{\infty}$~UHF algebra.
Choose a bijection $\sm \colon \N \to \Z$,
and for $\zt = (\zt_n)_{n \in \Z} \in X_0$
define
\[
\af_{\zt}^{(0)}
 = \bigotimes_{n \in \N}
  \Ad \left( \left( \begin{matrix}
       1 & 0 \\
       0 & \zt_{\sm (n)}
     \end{matrix} \right) \right)
 \in \Aut (D).
\]
Then $\zt \mapsto \af_{\zt}^{(0)}$ is \ct.

We have $\mdim (h_0) = 1$ (see Notation~\ref{N_0108_DimX})
by Proposition~3.3 of~\cite{LndWs},
and we can use $X_0$, $h_0$, $D$, and $\zt \mapsto \af^{(0)}_{\zt}$
in Lemma~\ref{L_4Y12_GetAuto}.
However, $h_0$ is not minimal.
We therefore proceed as follows.
Identify $[0, 1]$ with a closed arc in~$S^1$,
say via $\ld \mapsto \exp (\pi i \ld)$.
Use this identification to identify
$[0, 1]^{\Z}$ with a closed subset of $X_0 = (S^1)^{\Z}$.
This identification is equivariant when both spaces
are equipped with the shift \hme{s}.
Let $X \S X_0$ and $h = h_0 |_{X} \colon X \to X$
be the minimal subshift
in~\cite{GlKr},
which can be taken to have mean dimension
arbitrarily close to~$1$.
For $\zt \in X$ let $\af_{\zt} = \af_{\zt}^{(0)}$.
Let $\af \colon \Z \to \Aut (C (X, D))$
be the corresponding action as in Lemma~\ref{L_4Y12_GetAuto}.
Then $\af$ lies over the free minimal action of $\Z$ on~$X$
generated by~$h$.

The crossed product
$C^* \bigl( \Z, \, C (X, D), \, \af \bigr)$ is simple
by Proposition~\ref{P_4Y15_CPSimple},
and is nuclear,
so it is $\JS$-stable by Theorem~\ref{ATraciallyZStable}.
\end{exa}

The action in Example~\ref{E_9214_ShiftAndUHF}
can also be described as follows.
Realize $D$ as the algebra of the canonical anticommutation
relations on generators $a_k$ for $k \in \N$,
following Section~5.1 of~\cite{Brtl}.
Then $\af_{\zt}^{(0)}$ is the gauge automorphism
$\af_{\zt}^{(0)} (a_k) = \zt_{\sm (k)} a_k$.
In~\cite{Brtl}, see Section~5.1 and the proof of Lemma~5.2.

The next example is a slightly different version,
with larger mean dimension.

\begin{exa}\label{E_9Z26_ShiftAndUHF_v2}
Let $X_0$, $h_0$, $D$, and $\zt \mapsto \af_{\zt}^{(0)}$
be as in Example~\ref{E_9214_ShiftAndUHF}.
We will, however, use the \hme{} $h_0^2$ of~$X_0$,
which is the shift on $(S^1 \times S^1)^{\Z}$.
Let $(X, h)$ be the minimal subspace of the shift on $([0, 1]^2)^{\Z}$
constructed in Proposition 3.5 of \cite{LndWs},
which satisfies $\mdim (h) > 1$ (Notation~\ref{N_0108_DimX}).
Use an embedding of $[0, 1]^2$ in~$(S^1)^2$
to choose an equivariant \hme{} $g$
from $(X, h)$ to an invariant closed subset of $(X_0, h_0^2)$.
For $x \in X$ let $\af_{x} = \af_{g (x)}^{(0)}$.
Let $\af \colon \Z \to \Aut (C (X, D))$
be the corresponding action as in Lemma~\ref{L_4Y12_GetAuto}.
The crossed product
$C^* \bigl( \Z, \, C (X, D), \, \af \bigr)$ is again
simple by Proposition~\ref{P_4Y15_CPSimple},
so $\JS$-stable by Theorem~\ref{ATraciallyZStable}.
\end{exa}

\begin{exa}\label{E_9214_ShiftAndCStFI_2}
Let $X$, $h$, $D$, and $\zt \mapsto \af_{\zt}$ be as in
Example~\ref{E_9214_ShiftAndUHF}.
Define $E = D \otimes C^*_{\mathrm{r}} (F_{\infty})$,
and, following Notation~\ref{N_9214_QFreeFn},
for $\zt \in X$ define
$\bt_{\zt} = \af_{\zt} \otimes \ph_{\zt}$.
Apply Lemma~\ref{L_4Y12_GetAuto}
to get an action $\bt \colon \Z \to \Aut (C (X, E))$
which lies over the free minimal action of $\Z$ on~$X$
generated by~$h$.

We claim that
$C^* \big( \Z, \, C (X, E), \, \bt \big)$ is tracially $\JS$-stable
and has stable rank one.
Tracial $\JS$-stability follows from Theorem~\ref{ATraciallyZStable}.
For stable rank one,
$E$ is exact,
so has strict comparison of positive elements
by Corollary 4.6 of~\cite{Rdm7}.
Choose any one point subset $Y \S X$.
Then $C^* \big( \Z, \, C (X, E), \, \bt \big)_Y$
(see Definition~\ref{D_4Y12_OrbSubalg})
is $\JS$-stable
by Corollary~\ref{EStableGivesZStable}.
This algebra
is simple by Proposition~\ref{P_4Y15_CPSimple},
so by Theorem 6.7 of~\cite{Rdm7} it has stable rank one.
The algebra $C^* \big( \Z, \, C (X, E), \, \bt \big)_Y$
is centrally large
in $C^* \big( \Z, \, C (X, E), \, \bt \big)$
by
Corollary \ref{AyCentLargeConds}(\ref{9212_AyCentLargeConds_StrCmp}),
so Theorem 6.3 of~\cite{ArPh} implies
that $C^* \big( \Z, \, C (X, E), \, \bt \big)$ has stable rank one.
\end{exa}

In Example~\ref{E_9214_ShiftAndCStFI_2}, we don't know whether
$C^* \big( \Z, \, C (X, E), \, \bt \big)$ is $\JS$-stable.
We also don't know whether tracial $\JS$-stability implies
stable rank one, although this is expected to be true.

There is nothing special about the specific formulas
for $x \mapsto \af_x$
in Example~\ref{E_9214_ShiftAndUHF}
and Example~\ref{E_9Z26_ShiftAndUHF_v2},
and $x \mapsto \bt_x$ in Example~\ref{E_9214_ShiftAndCStFI_2}.
They were chosen merely to show that interesting examples exist.

Example~\ref{E_9214_ShiftAndUHF},
Example~\ref{E_9Z26_ShiftAndUHF_v2},
and Example~\ref{E_9214_ShiftAndCStFI_2}
were constructed so that the \hme~$h$
does not have mean dimension zero.
If $X$ is \fd,
then one can get all we do by using known results
for crossed products by actions with finite Rokhlin dimension
with commuting towers.
With \fd~$X$, one even gets $\JS$-stability in
examples like Example~\ref{E_9214_ShiftAndUHF},
Example~\ref{E_9Z26_ShiftAndUHF_v2},
and Example~\ref{E_9214_ShiftAndCStFI_2},
by Theorem~5.8 of~\cite{HWZ}.
However,
we know of no results which apply to examples of this type
when $X$ is in\fd{} and $h$ has mean dimension zero.

The next most obvious choice for a \mh{} of an in\fd{} space~$X$
seems to be as follows.
Take $X = (S^1)^{\N}$,
fix $\te = (\te_n)_{n \in \N} \in \R^{\N}$,
and define $h \colon X \to X$ by
$h \bigl( (\zt_{n})_{n \in \N} \bigr)
 = \bigl( e^{2 \pi i \te_n} \zt_{n} \bigr)_{n \in \N}$.
If $1, \te_1, \te_2, \ldots$ are linearly independent over~$\Q$,
then $h$ is minimal
by Proposition~\ref{P_4X04_ProdRotation}.
However, by considering the action in just one coordinate,
one sees that, regardless of $\zt \mapsto \af_{\zt}$,
the action has finite Rokhlin dimension with commuting towers.
By Theorem 6.2 of~\cite{HWZ}, irrational rotations
have finite Rokhlin dimension with commuting towers.
Remark 6.3 of~\cite{HWZ}, according to which the result
extends to any homeomorphism which has an irrational
rotation as a factor,
applies equally well to any automorphism of $C (S^1) \otimes B$,
for any unital \ca~$B$,
which lies over an irrational rotation on~$S^1$
in the sense of Definition~\ref{D_4Y12_OverX}.
If we start with $C \bigl( (S^1)^{\N}, \, D \bigr)$,
take $B$ above to be $C \bigl( (S^1)^{\N \SM \{ 1 \}}, \, D \bigr)$.
Therefore Theorem~5.8 of~\cite{HWZ} applies,
and our results give nothing new.

We now give some examples of the second type discussed
in the introduction to this section.
For easy reference,
we recall a result on the stable rank of reduced free products.
Many reduced free products have stable rank~$1$.
The following result is from~\cite{DHR}.
(It is not affected by the correction~\cite{DHRc}.)
Reduced free products of unital \ca{s}
in~\cite{DHR}
are implicitly taken to be amalgamated over~$\C$;
see Section~2.2 of~\cite{DHR}.

\begin{prp}[Corollary~3.9 of~\cite{DHR}]\label{P_5416_TsrFrGp}
Let $G$ and $H$ be discrete groups
with $\card (G) \geq 2$
and $\card (H) \geq 3$.
Then $C^*_{\mathrm{r}} (G \star H)$ has stable rank~$1$.
\end{prp}

\begin{exa}\label{E_5416_CSrFGp}
Let $X$ be the Cantor set
and let $h$ be an arbitrary minimal
homeomorphism of~$X$.
Let $\sm \in \Aut ( C^*_{\mathrm{r}} (F_{\infty}) )$
be as in Notation~\ref{N_9214_FreeShift}.
Let
$\af \in \Aut \big( C (X) \otimes C^*_{\mathrm{r}} (F_{\infty}) \big)$
be the tensor product of the automorphism
$f \mapsto f \circ h^{-1}$ of $C (X)$ and~$\sm$.
Then
$C^* \big( \Z, \, C (X) \otimes C^*_{\mathrm{r}} (F_{\infty}),
    \, \af \big)$
is simple by Proposition~\ref{P_4Y15_CPSimple}.

We claim that
$C^* \big( \Z, \, C (X) \otimes C^*_{\mathrm{r}} (F_{\infty}),
    \, \af \big)$
has stable rank~$1$.

To prove the claim,
we apply Theorem~\ref{P_9Z26_XisCantor} with
$D = C^*_{\mathrm{r}} (F_{\infty})$ and
$\af$ as given.
Use Theorems 9.2.6 and 9.2.7 of \cite{PhNotes}
to see that $D$ is simple and has a tracial state.
By Proposition~\ref{P_5416_TsrFrGp}, we have $\tsr (D) = 1$.
By Proposition 6.3.2 of \cite{Rbt},
$D$ has strict comparison of positive elements.
By construction, $\af$ lies over~$h$.
Now apply
Theorem~\ref{P_9Z26_XisCantor}(\ref{9217_GetTsr1_StrCmp}).
\end{exa}

Apparently no previously known results give
anything about the crossed product in this example.
In particular, as discussed after Problem~\ref{Pb_9922_RokhlinTsr},
knowing that the action has the Rokhlin property
doesn't seem to help.

There are many other automorphisms
of $C^*_{\mathrm{r}} (F_{\infty})$
which could be used in place of~$\sm$.
For example,
one could take a quasifree automorphism,
as in Notation~\ref{N_9214_QFreeFn}.
Here is a more interesting version.

\begin{exa}\label{E_5416_DenjFst}
Fix any $n \in \set{ 2, 3, \ldots, \infty }$
and take $D = C^*_{\mathrm{r}} (F_{n})$.
Adopt Notation~\ref{N_9214_QFreeFn}.
Let $h \colon X \to X$ be a restricted Denjoy \hme{}
with rotation number $\te \in \R \setminus \Q$,
as in Definition~\ref{D_9X04_RestrDj}.
Let $\zt \colon X \to S^1$ be the \ct{} surjective map
with $\zt (h (x)) = e^{2 \pi i \te} \zt (x)$
of Lemma~\ref{L_9X04_DjSemiCj}
(gotten from Corollary~3.2 and Proposition~3.4 of~\cite{PtSmdSk}).
Following Notation~\ref{N_9214_QFreeFn} for the generators
of~$D$,
for $x \in X$
let $\af_x \in \Aut ( D )$
be determined
by $\af_x (u_k) = \zt (x) u_k$
for $k = 1, 2, \ldots, n$ (or $k \in \N$ if $n = \I$).
Apply Lemma~\ref{L_4Y12_GetAuto}
to get an action
$\af \colon \Z \to \Aut \big( C (X, D ) \big)$,
which we can think of as
a kind of noncommutative Furstenberg transformation.
The algebra $C^* \big( \Z, \, C (X, D ), \, \af \big)$
is simple by Proposition~\ref{P_4Y15_CPSimple}.

We claim that
$C^* \big( \Z, \, C (X, D ), \, \af \big)$
has stable rank one.
To prove the claim,
use Theorems 9.2.6 and 9.2.7 of \cite{PhNotes}
to see that $D$ is simple and has a tracial state.
By Proposition~\ref{P_5416_TsrFrGp}, we have $\tsr (D) = 1$.
Now apply
Theorem~\ref{P_9Z26_XisCantor}(\ref{9217_GetTsr1_CptGen}).
\end{exa}

Again, apparently no previously known results give
anything about the crossed product.
Since $C^*_{\mathrm{r}} (F_{n})$ is not $\JS$-stable,
knowing that the action has the Rokhlin property
doesn't seem to help.

We can generalize Example~\ref{E_5416_DenjFst} as follows.

\begin{exa}\label{E_9217_MultltDenj}
Fix $n \in \{ 2, 3, \ldots, \I \}$.
Fix $\te_1, \te_2, \ldots, \te_n \in \R$
(or $\te_1, \te_2, \ldots \in \R$)
such that $1, \te_1, \te_2, \ldots, \te_n$
(or $1, \te_1, \te_2, \ldots$)
are linearly independent over~$\Q$.
For $k = 1, 2, \ldots, n$,
let $h \colon X_k \to X_k$ be a restricted Denjoy \hme{}
with rotation number $\te_k \in \R \setminus \Q$,
as in Definition~\ref{D_9X04_RestrDj}.
Take $X = \prod_{k = 1}^n X_k$,
and let $h \colon X \to X$
act as $h_k$ on the $k$-th factor.
Then~$h$ is minimal by Proposition~\ref{P_4X04_ProdRstDj}.
Define $\zt_k \colon X_k \to S^1$
analogously to the definition of $\zt$
in Example~\ref{E_5416_DenjFst},
and take $\af_x (u_k) = \zt_k (x_k) u_k$.
Apply Lemma~\ref{L_4Y12_GetAuto}
to get an action
$\af \colon
 \Z \to \Aut \big( C (X, \, C^*_{\mathrm{r}} (F_{n}) ) \big)$.
Then
$C^* \big( \Z, \, C (X, \, C^*_{\mathrm{r}} (F_{n}) ),
    \, \af \big)$
is simple and has stable rank one
for the same reasons as in Example~\ref{E_5416_DenjFst}.
\end{exa}

\begin{exa}\label{R_5416_More}
Let $X$ be the Cantor set
and let $h$ be an arbitrary minimal
homeomorphism of~$X$.
Choose any decomposition $X = X_1 \amalg X_2$
of $X$ as the disjoint union of two nonempty closed subsets.
Set $D = C^*_{\mathrm{r}} (F_{\infty})$,
with the generators of $F_{\I}$ indexed by~$\Z$.
Fix any $\zt = (\zt_n)_{n \in \Z} \in (S^1)^{\N}$.
For $x \in X_1$ take
$\af_{x}$ to be the automorphism
$\ph_{\zt}$ of Notation~\ref{N_9214_QFreeFn},
except using $\Z$ in place of $\N$,
and for $x \in X_2$ take
$\af_{x}$ to be the automorphism
$\sm$ of Notation~\ref{N_9214_FreeShift}.
Apply Lemma~\ref{L_4Y12_GetAuto}
to get an action
$\af \colon \Z \to \Aut ( C (X, D) )$.
The algebra $C^* (\Z, \, C (X) \otimes D, \, \alpha)$
is simple by Proposition~\ref{P_4Y15_CPSimple}.
The algebra $D$ has strict comparison of positive elements by
Proposition 6.3.2 of~\cite{Rbt},
so $C^* (\Z, \, C (X) \otimes D, \, \alpha)$ has
stable rank one
by Theorem~\ref{P_9Z26_XisCantor}(\ref{9217_GetTsr1_StrCmp}).
\end{exa}

Example~\ref{R_5416_More} admits many variations.
Here are several.

\begin{exa}\label{E_9217_AltDenj}
Let $h \colon X \to X$ be a restricted Denjoy \hme,
as in Example~\ref{E_5416_DenjFst},
and let $\zt \colon X \to S^1$ be as there.
Choose any decomposition $X = X_1 \amalg X_2$
of $X$ as the disjoint union of two nonempty closed subsets.
Following Notation~\ref{N_9214_QFreeFn} for the generators
of $C^*_{\mathrm{r}} (F_{2})$,
for $x \in X_1$
let $\af_x \in \Aut ( C^*_{\mathrm{r}} (F_{n}) )$
be determined
by $\af_x (u_1) = \zt (x) u_1$ and $\af_x (u_2) = u_2$,
and for $x \in X_2$
let $\af_x \in \Aut ( C^*_{\mathrm{r}} (F_{n}) )$
be determined
by $\af_x (u_1) = u_1$ and $\af_x (u_2) = \zt (x) u_2$.
The algebra
$C^* \big( \Z, \, C (X)
  \otimes C^*_{\mathrm{r}} (F_{2}),  \, \af \big)$
is simple by Proposition~\ref{P_4Y15_CPSimple}.

We claim that
this algebra
has stable rank one.
To prove the claim,
use Theorems 9.2.6 and 9.2.7 of \cite{PhNotes}
to see that $D$ is simple and has a tracial state.
By Proposition~\ref{P_5416_TsrFrGp}, we have $\tsr (D) = 1$.
Now apply
Theorem~\ref{P_9Z26_XisCantor}(\ref{9217_GetTsr1_CptGen}).
\end{exa}

\begin{exa}\label{E_9217_DenjAndSw}
In Example~\ref{E_9217_AltDenj},
replace the definition of $\af_x$
for $x \in X_2$ with
$\af_x (u_1) = u_2$ and $\af_x (u_2) = u_1$.
Proposition~\ref{P_4Y15_CPSimple}
and Theorem~\ref{P_9Z26_XisCantor}(\ref{9217_GetTsr1_CptGen})
still apply,
so
$C^* \big( \Z, \, C (X)
  \otimes C^*_{\mathrm{r}} (F_{2}),  \, \af \big)$
again is simple and has stable rank one.
\end{exa}

We now give examples in which we also get real rank zero.
The following lemma will be used for them.

\begin{lem}\label{L_0916_MakeSimple}
Let $M$ be factor of type $\mathrm{II}_1$,
let $\af \colon G \to \Aut (M)$
be an action of a countable group~$G$ on~$M$,
and let $P \S M$ be a separable C*-subalgebra.
Then there exists a simple separable unital C*-subalgebra
$D \subset M$ which contains~$P$,
is invariant under ${\overline{\sm}}$,
has a unique tracial state
(the restriction to $D$ of the unique tracial state on $M$),
has real rank zero and stable rank one,
and such that the order on projections over~$D$
is determined by traces.
\end{lem}

\begin{proof}
The proof is the same as that of Proposition~3.1 of~\cite{Ph57},
except for $G$-invariance.
We construct by induction on $n \in \Nz$ separable unital subalgebras
$A_n, B_n, C_n, D_n, E_n, F_n \S M$ with
\[
P \S A_0 \S B_0 \S C_0 \S D_0 \S E_0 \S F_0 \S \cdots
  \S A_n \S B_n \S C_n \S D_n \S E_n \S F_n \S \cdots
\]
such that $A_n$, $B_n$, $C_n$, $D_n$, and $E_n$ have the
properties in the proof of Proposition~3.1 of~\cite{Ph57}
($A_n$ is simple, etc.),
and such that $F_n$ is $G$-invariant.
In the induction step, after constructing $E_n$,
we take $F_n$ to be the C*-subalgebra generated by
$\bigcup_{g \in G} \af_g (E_n)$.
Then ${\overline{\bigcup_{n = 0}^{\I} F_n}}$ is $G$-invariant.
The proof in~\cite{Ph57} now gives the desired conclusion,
except that the conclusion is that $K_0 (D) \to K_0 (M)$
is an order isomorphism onto its range
instead of that the order on projections over~$D$
is determined by traces.
But the conclusion we get implies the conclusion we want
if $A$ has cancellation,
and stable rank one implies cancellation
by Proposition 6.4.1 and Proposition 6.5.1 of~\cite{Blkd3}.
\end{proof}

\begin{exa}\label{E_5416_DenjRR0}
Let $X$ be the Cantor set
and let $h$ be an arbitrary minimal
homeomorphism of~$X$.
Let $\sm \in \Aut ( C^*_{\mathrm{r}} (F_{\infty}) )$
be as in Notation~\ref{N_9214_FreeShift}.

We regard $C^*_{\mathrm{r}} (F_{\infty})$
as a subalgebra of the group
von Neumann algebra $W^* (F_{\infty})$
in the usual way,
and we let ${\overline{\sm}} \in \Aut (W^* (F_{\infty}))$
be the von Neumann algebra automorphism
which shifts the generators of $F_{\infty}$
in the same way that $\sm$ does.
Thus ${\overline{\sm}} |_{C^*_{\mathrm{r}} (F_{\infty})} = \sm$.
Choose a ${\overline{\sm}}$-invariant
subalgebra $D \subset W^* (F_{\infty})$
which contains $C^*_{\mathrm{r}} (F_{\infty})$
as in Lemma~\ref{L_0916_MakeSimple},
with the properties there.
Set $\gm = {\overline{\sm}} |_D \in \Aut (D)$.
Then define $\af \in \Aut ( C (X) \otimes D )$
in the same was as in Example~\ref{E_5416_CSrFGp},
using $\gm$ in place of~$\sm$.
The algebra $C^* \big( \Z, \, C (X) \otimes D, \, \af \big)$
is simple by Proposition~\ref{P_4Y15_CPSimple}.
It has stable rank~$1$ and real rank zero
by Theorem \ref{P_9Z26_XisCantor}(\ref{9217_GetTsr1_SP}).

The algebra $D$ is not nuclear
because the Gelfand-Naimark-Segal representation
from its tracial state gives a nonhyperfinite factor.
\end{exa}

It is perhaps interesting to point out that every quasitrace
on the algebra~$D$ in Example~\ref{E_5416_DenjRR0} is a trace.
This is a consequence of the following lemma.

\begin{lem}\label{L_9929_UniqQTr}
Let $A$ be a unital \ca{} with real rank zero.
Suppose that $\ta$ is a quasitrace on~$A$ and
that whenever $p, q \in M_{\I} (A)$ are \pj{s}
with $\ta (p) < \ta (q)$,
then $p \precsim q$.
Then $\ta$ is the only quasitrace on~$A$.
\end{lem}

\begin{proof}
Suppose the conclusion is false,
and let $\sm$ be some other quasitrace on~$A$.
By definition, for any quasitrace~$\rh$ on~$A$
and any $a, b \in A_{\mathrm{sa}}$,
we have $\rh (a + i b) = \rh (a) + i \rh (b)$.
Therefore there is $a \in A_{\mathrm{sa}}$
such that $\sm (a) \neq \ta (a)$.
Since quasitraces are continuous,
it follows from real rank zero that there is $a \in A_{\mathrm{sa}}$
such that $a$ has finite spectrum and $\sm (a) \neq \ta (a)$.
Since quasitraces are linear on commutative C*-subalgebras,
there is a \pj{} $p \in A$ such that $\sm (p) \neq \ta (p)$.

Suppose $\sm (p) < \ta (p)$.
Choose $m, n \in \Nz$ such that $m \sm (p) < n < m \ta (p)$.
Define \pj{s} $e, f \in M_{\I} (A)$ by $e = 1_{M_m} \otimes p$ and
$f = 1_{M_n} \otimes 1_A$.
Then $\ta (f) < \ta (e)$, so $f \precsim e$.
But $\sm (f) > \sm (e)$, a contradiction.
Similarly, $\sm (p) > \ta (p)$ is also impossible.
This contradiction shows that $\sm$ does not exist.
\end{proof}

\begin{exa}\label{E_9217_BiggerCstFn}
Let $n \in \set{ 2, 3, \ldots }$.
Following Notation~\ref{N_9214_QFreeFn} for the generators
of $C^*_{\mathrm{r}} (F_{n})$,
for $\rh$ in the symmetric group~$S_n$
let $\ps_{\rh} \in \Aut ( C^*_{\mathrm{r}} (F_{n}) )$
be the automorphism determined
by $\ps_{\rh} (u_k) = u_{\rh^{-1} (k)}$ for $k = 1, 2, \ldots, n$,
and let ${\overline{\ps_{\rh} }}$ be the corresponding automorphism
of the group von Neumann algebra $W^* (F_{n})$.

By Theorems 9.2.6 and 9.2.7 of \cite{PhNotes},
the algebra $C^*_{\mathrm{r}} (F_{n})$ is simple
and has a unique tracial state.
Use Lemma~\ref{L_0916_MakeSimple}
to find a simple separable unital \ca{}
$D \subset W^* (F_{n})$ which contains $C^*_{\mathrm{r}} (F_{n})$,
is invariant under all the automorphisms
${\overline{\ps_{\rh} }}$ for $\rh \in S_n$,
and has the other properties given in Lemma~\ref{L_0916_MakeSimple}.

Let $X$ be the Cantor set, and let $h \colon X \to X$ be any \mh.
Choose any decomposition $X = \coprod_{\rh \in S_n} X_{\rh}$
of $X$ as the disjoint union of nonempty closed subsets.
For $x \in X_{\rh}$ let $\af_x = {\overline{\ps_{\rh} }} |_D$.
Let $\af \colon \Z \to \Aut (C (X, D))$
be the corresponding action as in Lemma~\ref{L_4Y12_GetAuto}.

The algebra $C^* \bigl( \Z, \, C (X, D), \, \af \bigr)$
is simple by Proposition~\ref{P_4Y15_CPSimple}.
It has stable rank one and real rank zero by
Theorem~\ref{P_9Z26_XisCantor}(\ref{9217_GetTsr1_CptGen}).
The algebra $D$ is not nuclear
for the same reason as in Example~\ref{E_5416_DenjRR0}.
\end{exa}

The next example is of the third type discussed
in the introduction to this section.

\begin{exa}\label{E_9217_GenPType}
Let $(\pi_1, \pi_2, \ldots)$
be a sequence of unital \fd{} representations
of $C^* (F_2)$
such that, for every $n \in \N$,
the representation
$\bigoplus_{k = n}^{\infty} \pi_k$ is faithful.
For $n \in \Nz$,
let $d (n)$ be the dimension of $\pi_n$,
and define
$l (n) = d (n) + 4$
and $r (n) = \prod_{k = 1}^n l (k)$.
Let $u_1, u_2 \in C^* (F_2)$
be the ``standard'' unitaries in $C^* (F_2)$,
obtained as the images of the standard generators of~$F_2$,
as in Notation~\ref{N_9214_FreeShift}
except that we are now using the full \ca{} instead
of the reduced \ca.
Let $\gm_1, \gm_2, \gm_3 \in \Aut (C^* (F_2))$
be the automorphisms determined by
\[
\gm_1 (u_1) = u_1^{-1}
\andeqn
\gm_1 (u_2) = u_2,
\]
\[
\gm_2 (u_1) = u_1
\andeqn
\gm_2 (u_2) = u_2^{-1},
\]
and
\[
\gm_3 (u_1) = u_1^{-1}
\andeqn
\gm_3 (u_2) = u_2^{-1}.
\]
For $n \in \N$
define
$D_n = M_{r (n)} \otimes C^* (F_2)$,
which we identify as
\[
M_{l (n)} \otimes M_{l (n - 1)} \otimes \cdots
    \otimes M_{l (1)} \otimes C^* (F_2),
\]
and define
$\gm_{n, n - 1} \colon D_{n - 1} \to D_n$
by,
for $a \in D_{n - 1} = M_{r (n - 1)} \otimes C^* (F_2)$,
\begin{align*}
\gm_{n, n - 1} (a)
& = \diag \bigl( a, \, (\id_{M_{r (n - 1)}} \otimes \ph_1 ) (a),
        \, (\id_{M_{r (n - 1)}} \otimes \ph_2 ) (a),
\\
& \hspace*{3em} {\mbox{}}
        \, (\id_{M_{r (n - 1)}} \otimes \ph_3 ) (a),
        \, (\pi_n \otimes \id_{M_{r (n - 1)}} (a) \otimes 1
     \bigr)
\\
& \in M_{l (n)} \otimes D_{n - 1} = D_n.
\end{align*}
Let $D$ be the direct limit of the resulting direct system.
It follows by methods of~\cite{Ddlt}
that $D$ is simple, separable, and has tracial rank zero.
In particular, $D$ has stable rank one
and real rank zero by Theorem 3.4 of~\cite{HLinTrAF},
and the order on projections over $D$ is determined by traces
by Theorem 6.8 of~\cite{HLinTTR}.
Moreover, using the known result for $K_* (C^* (F_n))$
(see~\cite{CuFree}),
one gets $K_1 (D) = 0$.

For $\zt \in S^1$ define inductively
unitaries $u_n (\zt) \in D_n$ as follows.
Set $u_0 (\zt) = 1$,
and, given $u_{n - 1} (\zt) \in D_{n - 1}$,
set
\[
u_n (\zt)
 = \bigl( \diag \bigl( \zt, 1, 1, \ldots, 1 \bigr)
             \otimes 1_{D_{n - 1}}  \bigr)
         \gm_{n, n - 1} ( u_{n - 1} (\zt) )
 \in M_{l (n)} \otimes D_{n - 1} = D_n.
\]
Then the actions
$\zt \mapsto \Ad (u_n (\zt)) \in \Aut (D_n)$
of $S^1$ on~$D_n$
are compatible with the direct system,
and so yield a \ct{} map
(in fact, an action)
$\zt \mapsto \af_{\zt} \colon S^1 \to \Aut (D)$.
Let $h \colon S^1 \to S^1$ be an irrational rotation.
Apply Lemma~\ref{L_4Y12_GetAuto}
to get an action
$\af \colon \Z \to \Aut ( C (S^1, D) )$.

The algebra $C^* \big( \Z, \, C (S^1, D), \, \af \big)$
is simple by Proposition~\ref{P_4Y15_CPSimple}
and has stable rank one
by
Theorem \ref{T_9222_Tsr1AndRR0_CrPrd}(\ref{9927_AyCentLargeConds_SP}).
\end{exa}

Methods of~\cite{NiuWng} will probably show that
the algebra $D$ in Example~\ref{E_9217_GenPType} is not $\JS$-stable
(see Question~\ref{Q_9217_GenPType_IsZStab}),
although it definitely is tracially $\JS$-stable.

Finally, we give purely infinite examples.

\begin{exa}\label{E_9214_ShiftAndOI}
Let $(X, h)$ be the minimal subshift
(of the shift on $(S^1)^{\Z}$) with nonzero mean dimension
used in Example~\ref{E_9214_ShiftAndUHF}.
Let $D = {\mathcal{O}}_{\I}$,
but with standard generators indexed by~$\Z$,
say $s_k$ for $k \in \Z$.
Let $\af_{\zt}$ be the gauge automorphism,
given by $\af_{\zt} (s_k) = \zt_k s_k$.
Let $\af \colon \Z \to \Aut (C (X, D))$
be the corresponding action as in Lemma~\ref{L_4Y12_GetAuto}.
Then $C^* \bigl( \Z, \, C (X, D), \, \af \bigr)$ is
purely infinite and simple by Theorem~\ref{T_0101_PICase}.
Since $C^* \bigl( \Z, \, C (X, D), \, \af \bigr)$
is nuclear,
it is necessarily ${\mathcal{O}}_{\I}$-stable
by Theorem 3.15 of \cite{KiPh}.
\end{exa}

In Example~\ref{E_9214_ShiftAndOI},
Theorems 4.1 and~5.8 of~\cite{HWZ} don't apply,
since there is no reason to think that $\af$
has finite Rokhlin dimension with commuting towers.

\begin{exa}\label{E_9214_OnPINotZStab}
Let $D$ be the reduced free product
$D = (M_2 \otimes M_2) \star_{\mathrm{r}} C ([0, 1])$,
taken with respect to the Lebesgue measure state on $C ([0, 1])$
and the state on $M_2 \otimes M_2$
given by tensor product of the usual tracial state ${\mathrm{tr}}$
with the state
$\rho (x) = {\operatorname{tr}} \big(
 {\mathrm{diag}} \big( \frac{1}{3}, \frac{2}{3} \big) x \big)$
on $M_2$.
It is shown in Example 5.8 of~\cite{AmGlJmPh}
that $D$ is purely infinite and simple
but not ${\mathcal{Z}}$-stable.

Take $X = (S^1)^4$,
and for $\zt = (\zt_1, \zt_2, \zt_3, \zt_4) \in X$
take $\af_{\zt}$
to be the free product automorphism
which is given by
\[
\Ad \left( \left( \begin{matrix}
       \zt_1 & 0 \\
       0     & \zt_2
     \end{matrix} \right) \right)
\otimes \Ad \left( \left( \begin{matrix}
       \zt_3 & 0 \\
       0     & \zt_4
     \end{matrix} \right) \right)
\]
on $M_2 \otimes M_2$
and is trivial on $C ([0, 1])$.
Choose $\te_1, \te_2, \te_3, \te_4 \in \R$
such that $1, \te_1, \te_2, \te_3, \te_4$
are linearly independent over~$\Q$.
Take $h \colon X \to X$
to be
\[
(\zt_1, \zt_2, \zt_3, \zt_4)
  \mapsto \bigl( e^{2 \pi i \te_1} \zt_1, \,
       e^{2 \pi i \te_2} \zt_2, \,
       e^{2 \pi i \te_3} \zt_3, \,
       e^{2 \pi i \te_4} \zt_4  \bigr).
\]
Then $C^* \bigl( \Z, \, C (X, D), \, \af \bigr)$ is
purely infinite and simple by Theorem~\ref{T_0101_PICase}.
\end{exa}

The action in Example~\ref{E_9214_OnPINotZStab}
has finite Rokhlin dimension with commuting towers.
However, we don't know any theorem on pure infiniteness
for crossed products by such actions
when the original algebra is not ${\mathcal{O}}_{\I}$-stable.
We address this in Question~\ref{Pb_0111_FinRDim_PI} below.
Our result also does not imply that the crossed product
is ${\mathcal{O}}_{\I}$-stable, or even ${\mathcal{Z}}$-stable,
although it seems plausible that it might be.

In the next two examples, $D$ is again not $\JS$-stable.
Also, $X$ isn't \fd, and $h$ doesn't even have mean dimension zero.
So a positive answer to Question~\ref{Pb_0111_FinRDim_PI}
presumably would not help.

\begin{exa}\label{E_9214_OnPINotZStab_2}
Let $D$ be as in Example~\ref{E_9214_OnPINotZStab},
and let $(X, h)$ be as in Example~\ref{E_9214_ShiftAndUHF}
(and reused in Example~\ref{E_9214_ShiftAndOI}).
For $\zt = (\zt_n)_{n \in \Z} \in X$
take $\af_{\zt}$ to be the free product automorphism
which is given by
\[
\Ad \left( \left( \begin{matrix}
       \zt_1 & 0 \\
       0     & \zt_2
     \end{matrix} \right) \right)
\otimes \Ad \left( \left( \begin{matrix}
       \zt_3 & 0 \\
       0     & \zt_4
     \end{matrix} \right) \right)
\]
on $M_2 \otimes M_2$
and is trivial on $C ([0, 1])$.
(We are only using the coordinates with indexes $1, 2, 3, 4$.)
Then $C^* \bigl( \Z, \, C (X, D), \, \af \bigr)$ is
purely infinite and simple by Theorem~\ref{T_0101_PICase},
but may well not be ${\mathcal{Z}}$-stable.
\end{exa}

With a small modification,
we can give an example of this type in which the
underlying action of $(S^1)^{\Z}$
is effective.

\begin{exa}\label{E_9Z26_OnPINotZStab_v2}
Let $B$ be the $2^{\infty}$~UHF algebra,
and let $\ta$ be its (unique) \tst.
Let $\tr$ be the \tst{} on~$M_2$,
and let $\rh$ be the state on $M_2$ given by
$\rho (x)
 = \tr \big( \diag \big( \frac{1}{3}, \frac{2}{3} \big) x \big)$.
Let $D$ be
the reduced free product
$D = (B \otimes M_2) \star_{\mathrm{r}} C ([0, 1])$,
taken with respect to the state $\ta \otimes \rh$ on $B \otimes M_2$
and the Lebesgue measure state on $C ([0, 1])$.
We claim that $A$ is purely infinite and simple
but not ${\mathcal{Z}}$-stable.

To prove pure infiniteness,
in Examples 3.9(iii) of~\cite{Dkm2}
take $A_1 = B$
with the state~$\ta$,
take $F = M_2$
with the state~$\rh$,
and take $B = C ([0, 1])$
with the state given by Lebesgue measure.
These choices satisfy the hypotheses there.
So $A$ is is purely infinite and simple.

To prove that $A$ is not ${\mathcal{Z}}$-stable,
let $\pi$ be the Gelfand-Naimark-Segal representation of~$B$
associated with~$\ta$,
set $N = \ta (B)''$,
and also write $\ta$ for the corresponding \tst{} on~$N$.
In Proposition 5.6 of~\cite{AmGlJmPh},
take $P_1 = N \otimes M_2$
with the state $\ta \otimes \rh$,
and take $P_2 = L^{\infty} ([0, 1])$
with the state given by Lebesgue measure.
Define
\[
a = \left[ \left( \begin{matrix} 1 & 0 \\ 0 & -1 \end{matrix} \right)
        \otimes 1 \otimes 1 \otimes \cdots \right]
      \otimes 1
    \in \left( \bigotimes_{n = 1}^{\I} M_2 \right) \otimes M_2
    = B \otimes M_2
    \S P_1.
\]
Take $G_1 = \{ 1,  a\} \S P_1$,
and take $G_2 \S P_2$ to be the set of functions
$\ld \mapsto e^{2 \pi i n \ld}$ for $n \in \Z$.
These choices
satisfy the hypotheses there.
Moreover, $G_1 \S B \otimes M_2$ and $G_2 \S C ([0, 1])$.
Therefore the reduced free product
$A = (B \otimes M_2) \star_{\mathrm{r}} C ([0, 1])$
is a subalgebra of the algebra~$P$
in Proposition 5.6 of~\cite{AmGlJmPh}
which contains $a$, $b$, and~$c$,
so is not ${\mathcal{Z}}$-stable by Proposition 5.6 of~\cite{AmGlJmPh}.
The claim is proved.

Let $(X, h)$ be
as in Example~\ref{E_9214_ShiftAndUHF}
(and reused in Example~\ref{E_9214_ShiftAndOI}).
Let $\sm \colon \N \to \Z$ be a bijection
(as in Example~\ref{E_9214_ShiftAndUHF}).
For $\zt = (\zt_n)_{n \in \Z} \in X$
take $\af_{\zt}$
to be the free product automorphism
which is the identity on $C ([0, 1])$
and which is the infinite tensor product
\[
\left[ \bigotimes_{n \in \N}
  \Ad \left( \left( \begin{matrix}
       1 & 0 \\
       0 & \zt_{\sm (n)}
     \end{matrix} \right) \right) \right]
     \otimes \Ad \left( \left( \begin{matrix}
       \zt_{\sm (1)} & 0 \\
       0             & \zt_{\sm (2)}
     \end{matrix} \right) \right)
  \in \Aut (B \otimes M_2)
\]
on~$B \otimes M_2$.
Then $C^* \bigl( \Z, \, C (X, D), \, \af \bigr)$ is
purely infinite and simple by Theorem~\ref{T_0101_PICase},
but may well not be ${\mathcal{Z}}$-stable.
\end{exa}

\section{Open problems}\label{Sec_Qs}

In this section,
we collect some open questions
suggested by the examples and results in this paper.

\begin{pbm}\label{Pb_9922_RokhlinTsr}
Let $A$ be a unital \ca,
and let $\af$ be an action of $\Z$ on~$A$
which has finite Rokhlin dimension with commuting towers.
Suppose that $A$ has stable rank one and $C^* (\Z, A, \af)$ is simple.
Does it follow that $C^* (\Z, A, \af)$
has stable rank one?
\end{pbm}

This seems to be unknown even if $A$ is simple
(in which case simplicity of $C^* (\Z, A, \af)$ is automatic)
and $\af$ has the Rokhlin property.
Without assuming simplicity of $C^* (\Z, A, \af)$,
the answer is definitely no.
For aperiodic \hme{s} of the Cantor set
whose \tgca{s} don't have stable rank one,
see Theorem~3.1 of~\cite{Pt1}
(it is easy to construct examples there
which have more than one minimal set)
or Example~8.8 of~\cite{Phl10}.
Corollary~2.6 of~\cite{Szb}
implies that the corresponding actions of~$\Z$
have Rokhlin dimension with commuting towers at most~$1$.
In fact, though,
at least in Example~8.8 of~\cite{Phl10} we get the Rokhlin property.

\begin{lem}\label{L_9922_HasRokhlin}
Let $h \colon X \to X$ be the aperiodic \hme{}
of the Cantor set in Example~8.8 of~\cite{Phl10}.
Then the induced automorphism of $C (X)$ has the Rokhlin property.
\end{lem}

\begin{proof}
Recall from~\cite{Phl10} that $X_1 = \Z \cup \{\pm \I \}$,
$h_1 \colon X_1 \to X_1$ is $n \mapsto n + 1$,
$X_2$ is the Cantor set,
$h_2 \colon X_2 \to X_2$ is minimal,
$X = X_1 \times X_2$,
and $h = h_1 \times h_2$.

Let $N \in \N$.
The standard first return time construction provides
\[
n, r (1), r (2), \ldots, r (n) \in \N
\]
with $N < r (1) < r (2) < \cdots < r (n)$,
and compact open subsets $Z_1, Z_2, \ldots, Z_n \S X_2$,
such that
\[
X_2 = \coprod_{k = 1}^{n} \coprod_{j = 0}^{r (k) - 1} h_{2}^j (Z_k).
\]
Then the sets
\[
X_1 \times Z_1, \quad X_1 \times Z_2,
 \quad \ldots, \quad X_1 \times Z_n
\]
form a system of Rokhlin towers for~$h$,
with heights $r (1), r (2), \ldots, r (n)$,
all of which exceed~$N$.
\end{proof}

In fact,
it seems to be known that any aperiodic \hme{}
of the Cantor set~$X$
induces an automorphism of $C (X)$ with the Rokhlin property.
We have not found a reference,
and we do not prove this here.

The following question asks for a plausible generalization
of Lemma~\ref{L_4Y11_Comp}.

\begin{qst}\label{Qst1_20190103D}
Let $D$ be a simple \uca.
Let $X$ be a \cms, let $G$ be a discrete group,
and let $(g,x) \mapsto gx$ be a minimal
and essentially free action of $G$ on $X$.
Let $\af \colon G \to \Aut (C (X, D))$ be an action of $G$
which lies over the action of $G$ on $X$
and which is pseudoperiodically generated.
Set $A = C^* \bigl( \Z, \, C (X, D), \, \af \bigr)$.

Does it follow that for every $a \in C (X, D)_{+} \setminus \{ 0 \}$
there exists $f \in C (X)_{+} \setminus \{ 0 \} \subset A$
such that $f \precsim_A a$?
\end{qst}

We expect the techniques in the proof of  Lemma \ref{L_4Y11_Comp}
(taking $Y$ to be the empty set) can be used to answer
this question in the affirmative.

\begin{qst}\label{Q_9212_ExistsNotPseudoper}
Is there a simple unital \ca~$D$
such that $\Aut (D)$
is not pseudoperiodic in the sense of Definition~\ref{D_4Y12_PsPer}?
\end{qst}

Presumably such examples exist,
but we don't know of any.
Indeed,
we don't see any reason why there should not be $\af \in \Aut (D)$
such that $\set{\af^n \colon n \in \Z}$ is not pseudoperiodic.

\begin{qst}\label{Q_9214_IsZStab}
Consider the crossed product
$C^* \big( \Z, \, C (X, E), \, \bt \big)$
in Example~\ref{E_9214_ShiftAndCStFI_2}.
Is this algebra $\JS$-stable?
What about crossed products by similarly constructed actions?
\end{qst}

\begin{qst}\label{Q_9214_IsZStab_2}
Consider the crossed product
$C^* \big( \Z, \, C (X)
  \otimes C^*_{\mathrm{r}} (F_{\infty}),  \, \af \big)$
in Example~\ref{E_5416_CSrFGp}.
Is this algebra $\JS$-stable?
What about crossed products by similarly constructed actions?
What about Example~\ref{E_5416_DenjFst}?
\end{qst}

\begin{qst}\label{Q_9217_GenPType_IsZStab}
Consider the crossed product
$C^* \big( \Z, \, C (S^1, D), \, \af \big)$
in Example~\ref{E_9217_GenPType}.
Is this algebra $\JS$-stable?
\end{qst}

The following question
is motivated by Example~\ref{E_9214_OnPINotZStab}.

\begin{qst}\label{Pb_0111_FinRDim_PI}
Let $A$ be a nonsimple unital \ca{} which is purely infinite
in the sense of Definition~4.1 of~\cite{KiRo2000}.
Let $\af \colon \Z \to \Aut (A)$
be an action with finite Rokhlin dimension with commuting towers.
Does it follow that $C^* (\Z, A, \af)$
is purely infinite?
\end{qst}

If $A$ is ${\mathcal{O}}_{\I}$-stable,
then $C^* (\Z, A, \af)$ is at least $\JS$-stable
by Theorem~5.8 of~\cite{HWZ}.
If $A$ is also exact,
then so is $C^* (\Z, A, \af)$
(by Proposition 7.1(v) of~\cite{Krh2}),
and $C^* (\Z, A, \af)$ is traceless
(as at the beginning of Section~5 of~\cite{Rdm7})
because $A$ is,
so $C^* (\Z, A, \af)$ is purely infinite
by Corollary~5.1 of~\cite{Rdm7}.
(In this case, one obviously wants ${\mathcal{O}}_{\I}$-stability.
See Question~\ref{Pb_0111_PI_ZStab_to_OIStab} below.)
But the question as stated seems to be open,
even if one assumes that $C^* (\Z, A, \af)$ is simple.
If $A$ itself is simple,
then $C^* (\Z, A, \af)$
is purely infinite even just assuming that $\af$ is pointwise outer,
by Corollary~4.4 of~\cite{JngOsk}.
Provided one uses the reduced crossed product,
this remains true if $\Z$ is replaced by any discrete group.

\begin{qst}\label{Pb_0111_OIStabQ}
Consider the crossed product
$C^* \bigl( \Z, \, C (X, D), \, \af \bigr)$
in Example~\ref{E_9214_OnPINotZStab_2}.
Is this algebra $\JS$-stable?
Is it ${\mathcal{O}}_{\I}$-stable?
\end{qst}

We hope that $\JS$-stability should come from the
action in the ``$(S^1)^4$~direction''.
We suppose that if a purely infinite simple \ca{}
is $\JS$-stable,
then it is probably ${\mathcal{O}}_{\I}$-stable,
but this seems to be open in general.

\begin{qst}\label{Pb_0111_PI_ZStab_to_OIStab}
Let $B$ be a $\JS$-stable purely infinite simple \ca.
Does it follow that $B$ is ${\mathcal{O}}_{\I}$-stable?
\end{qst}

If $B$ is separable and nuclear,
the answer is yes, by Theorem~5.2 of~\cite{Rdm7}.
But this result doesn't help with Example~\ref{E_9214_OnPINotZStab_2},
even if we prove $\JS$-stability there.

\end{document}